\documentclass[a4paper]{amsart}
\usepackage[utf8]{inputenc}
\usepackage[left=3cm, right=3cm, top = 3.5cm, bottom=3.5cm]{geometry}

\usepackage[dvipsnames]{xcolor}
\definecolor{darkblue}{rgb}{0, 0, 0.6}

\usepackage[pdftitle={Defects in skein theory and TQFT}, pdfauthor={Patrick Kinnear and Ingo Runkel}, pdfpagelayout=OneColumn, ocgcolorlinks, colorlinks=true, linkcolor=darkblue, urlcolor=NavyBlue, citecolor=darkblue, filecolor=darkblue]{hyperref}

\usepackage[safeinputenc, style=alphabetic, citestyle=alphabetic, maxnames=5, maxalphanames=5, backend=biber]{biblatex}
\DeclareLabelalphaTemplate{
  \labelelement{
    \field[final]{shorthand}
  }
  \labelelement{
    \field[strwidth=1,strside=left]{labelname}
  }
}

\DeclareBibliographyCategory{needsurl}
\renewbibmacro*{url+urldate}{%
  \ifcategory{needsurl}
    {\printfield{url}%
     \iffieldundef{urlyear}
       {}
       {\setunit*{\addspace}%
        \printurldate}}
    {}}
\newcommand{\entryneedsurl}[1]{\addtocategory{needsurl}{#1}}

\DeclareFieldFormat{url}{\newline\mkbibacro{URL}\addcolon\nobreakspace\url{#1}}

\AtEveryBibitem{%
  \clearfield{pubstate}
  \clearfield{issn}
  \clearfield{urlyear}
  \clearfield{eprintclass}
  \ifentrytype{misc}
  {\clearfield{doi}
  \clearfield{year}}
  {}%
}

\entryneedsurl{Kel05BasicConceptsEnriched}

\usepackage{stmaryrd}

\usepackage{quiver}

\usepackage{pkinne}

\usepackage{subcaption}
\usepackage{comment}
\usepackage{adjustbox} 

\usepackage[notransparent, inkscapepath=svgsubdir]{svg}

\graphicspath{{images/}}

\def\Bord{\mathrm{Bord}}
\def\SSigma{\mathbb{S}}
\def\Surface{S}
\def\RT{\mathrm{RT}}
\def\Rex{\mathbf{Rex}}
\def\Irr{\mathrm{Irr}}
\def\Ev{\mathrm{Ev}}
\def\M{\mathbb{M}}
\newcommand{\interior}[1]{\mathrm{int}{#1}}
\newcommand{\Cylinder}[1]{\mathrm{Cyl}(\M, {#1})}

\makeatletter
\providecommand*{\twoheadrightarrowfill@}{%
  \arrowfill@\relbar\relbar\twoheadrightarrow
}
\providecommand*{\twoheadleftarrowfill@}{%
  \arrowfill@\twoheadleftarrow\relbar\relbar
}
\providecommand*{\xtwoheadrightarrow}[2][]{%
  \ext@arrow 0579\twoheadrightarrowfill@{#1}{#2}%
}
\providecommand*{\xtwoheadleftarrow}[2][]{%
  \ext@arrow 5097\twoheadleftarrowfill@{#1}{#2}%
}
\makeatother

\title{Defects in skein theory and TQFT}
\author{Patrick Kinnear and Ingo Runkel}
\address{Fachbereich Mathematik, Universit\"{a}t Hamburg, Bundesstra{\ss}e 55, 20146 Hamburg, Germany}
\email{patrick.kinnear@uni-hamburg.de, ingo.runkel@uni-hamburg.de}

\subjclass[2020]{57K16, 57K31, 18M30, 18M15}

\begin{document}

    \begin{abstract}
      Given a 3-manifold $M$ with a network of line and point defects in its boundary, we define the skein module of this configuration, generalizing the well-studied case of 3-manifolds which only admit point defects in the boundary. We prove that when all defects are labelled by semisimple data, our skein module is isomorphic to the state space of $\partial M$ in the defect version of the Reshetikhin--Turaev TQFT constructed by Carqueville--Runkel--Schaumann. Our defect skein modules follow naturally by globalizing the graphical calculus of module categories and functors thereof,  and generalize the possible defect data considered in the defect TQFT beyond the semisimple case. 
\end{abstract}

    \maketitle

    \tableofcontents

    \section{Introduction}
    In physics, \emph{defects} are submanifolds along which an excitation or disruption occurs, where extended operators are localized, or where differing regimes meet. Far from indicating deficiency, it is now understood that defects play an indispensable role as carriers of symmetry in quantum field theory: see \cite{FFRS04KramersWannierDualityConformal,GKSW15GeneralizedGlobalSymmetries,FMT23TopologicalSymmetryQuantum} for some works in this direction.

    Moreover, there are interesting recent applications of defect skein constructions in quantum topology, such as \cite{BJ25ParabolicSkeinModules} which uses a so-called parabolic defect in skein theory to obtain a powerful (and algorithmically computable) description of the quantum $A$-polynomial, and \cite{Gar26HOMFLYParabolicRestriction}, which uses an interpolating version of a parabolic defect to categorify Turaev's coproduct on the HOMFLY skein algebra. Skein modules for bicategories were developed in \cite{FSY23StringnetModelsPivotal} and have applications to correlators in two-dimensional defect conformal field theory.
 
    In the context of $(2  + 1)$-dimensional topological quantum field theories, defects of all codimension were introduced in TQFTs of Turaev--Viro--Barrett--Westbury type in \cite{Meu23StateSumModels}. For Reshetikhin--Turaev type TQFTs, state spaces of surfaces admit a description by skein modules \cite{BHMV95TopologicalQuantumField}. TQFTs of this type can be moreover be equipped with defects of all codimension using the techniques developed in \cite{CRS19OrbifoldsNdimensionalDefect, CRS18LineSurfaceDefects,CRS20OrbifoldsReshetikhinTuraevTQFTs}, where defects were used to implement and also gauge 
    possibly non-invertible symmetries    
    in the RT theory. In this paper we are concerned with giving an interpretation of these defect constructions in terms of skein theory. In particular, we introduce the defect skein module of a 3-manifold with defects in its boundary, and identify this with the defect RT state space of the boundary surface.

    \medskip

    The ordinary RT theory is based on a modular fusion category $\cC$, and is defined for marked surfaces and 3d bordisms containing $\cC$-coloured ribbon graphs. As explained in \cite{CRS18LineSurfaceDefects}, this bordism category can be augmented to include line and point defects in surfaces, and surface and line defects in 3d bordisms, labelled with appropriate data in $\cC$. The state space of a defect surface is calculated by producing from the defect arrangement an idempotent ribbon bordism which is topologically a cylinder, and whose ribbons end at marked points corresponding to a triangulation of the defects. The image of the induced projector on the (ordinary) RT state space of the marked surface is the state space of the defect theory $Z^{\mathrm{def}}$.

    The goal of the present paper is to describe this defect state space using skein theory. Already skein modules can be used to describe state spaces in the ordinary RT theory: if a marked surface is parameterized as the boundary of a 3-manifold $M$, then the $\cC$-skein module $\Sk(M)$ is isomorphic to the state space of the surface.
    Any modular fusion category (indeed any ribbon category) admits a 3d graphical calculus, and the skein module can be seen as globalizing this calculus: it is the span of all topological $\cC$-coloured diagrams in $M$ modulo the relations holding in the graphical calculus for $\cC$, which are applied in any embedded ball.
    
    To connect such graphical calculi with the defect construction of \cite{CRS18LineSurfaceDefects}, we note that the line defects of \cite{CRS18LineSurfaceDefects} are labelled by $\Delta$-separable symmetric Frobenius algebras, which by \cite{Ost01ModuleCategoriesWeak,Sch13TracesModuleCategories} correspond to semisimple $\cC$-module categories equipped with $\cC$-module trace; and the point defects are labelled by multimodules, which by a version of the Eilenberg--Watts theorem (see e.g.\ \cite{BJS21DualizabilityBraidedTensor}) correspond to $\cC$-module functors. Module categories (and their functors) also admit a graphical calculus, this time naturally defined on the boundary of a $\cC$-labelled bulk, and where plain module categories are considered (as opposed to tensor or braided tensor categories), confined to a 1-dimensional locus. 
    
    There is therefore a natural way to generalize the skein module to the \emph{defect skein module} of a 3-manifold with line and point defects in its boundary, labelled by module categories and functors thereof. We define this object, and denote it by $\Sk(\M)$ where $\M$ is notation for the manifold $M$ together with its collection of defects. When the defects in $\M$ are simply marked points in the boundary, we recover the usual definition of skein module, corresponding to the ordinary RT theory. Our definition is the natural result of meditating on the graphical calculi of defects, and as such we expect it will be unsurprising to experts: similar definitions are given, for instance, in \cite{BJ25ParabolicSkeinModules,Gar26HOMFLYParabolicRestriction}, for different flavours of defect to ours. Indeed, the shape that defect skein theory should take more generally was implicitly outlined, for example, in \cite[\S 6.4]{MW12BlobComplex}. However, a fully worked and explicit definition for boundary line defects in the language of ribbon categories and their modules has not appeared in the literature to the best of our knowledge, and we consider this a novel contribution of the present work.

      \begin{figure}
      \centering
      \includesvg[width=16cm]{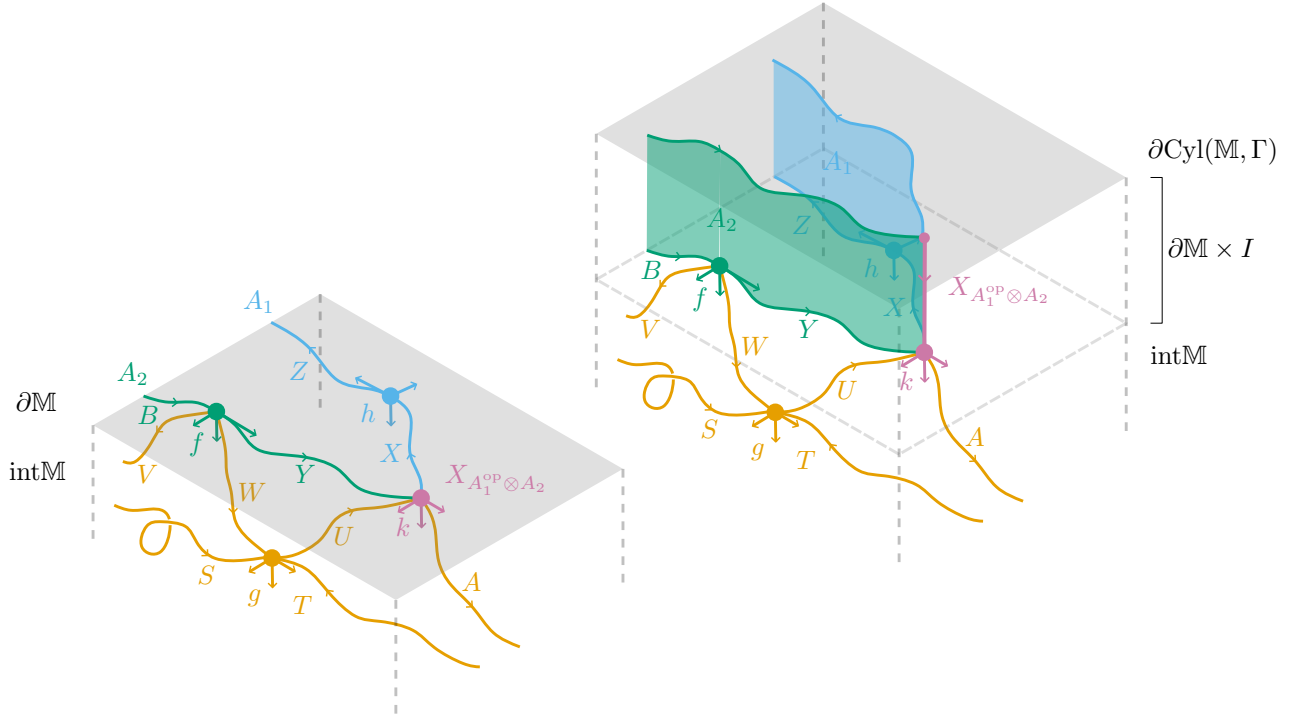}
      \caption{The definition of $\Phi$. Left: a section of a ribbon graph $\Gamma$ near a point defect. The interior edges and vertices are  coloured orange, the boundary edges and vertices are coloured blue, green and pink, and the boundary region  is shaded. Right: the resulting bordism after attaching a cylinder over the line and point defects. This is (part of) a defect bordism $\Cylinder{\Gamma} :  \emptyset \to \partial \M$.}
      \label{f-phi-def}
    \end{figure}

    Elements of $\Sk(\M)$ are represented by graphs which along line defects are coloured by modules $M_i \in \cM_i = A_i \Mod_{\cC}$, and at point defects by maps of multimodules between the algebras labelling coincident lines. These are appropriate boundary conditions in the defect RT theory for surface defects labelled by algebras $A_i$. By coupling defect skeins $[\Gamma]$ in $\M$ to a cylindrical defect bordism $\partial \M \times I$ along the line and surface defects, we then obtain defect bordisms $\Cylinder{\Gamma} : \emptyset \to \partial \M$, which induce elements of $Z^{\mathrm{def}}(\partial \M)$,
    see Figure~\ref{f-phi-def}.   
    This assembles into a well-defined linear map
    \[
      \Phi : \Sk(\M) \to Z^{\mathrm{def}}(\partial \M)
    \]
    to the state space of the defect TQFT. The main theorem of this paper is then stated as follows.   

    \begin{thm}{(\ref{t-main-thm-in-text})}
      \label{t-main-thm}
      Let $\M$ be a 3-manifold with line and surface defects in its boundary. Then $\Phi$ defines an isomorphism
      \[
        \Phi : \Sk(\M) \xrightarrow{\cong} Z^{\mathrm{def}}(\partial \M).
      \]
    \end{thm}

    Let us remark that, while RT theory and its defect versions are defined for modular \emph{fusion} categories and their \emph{semisimple} module categories, our definition of defect skein modules allows for defects coloured by non-semisimple module categories, strictly enlarging the scope of the projector-based treatment of defects beyond semisimplicity. 

\medskip
\noindent
\textbf{Outlook.}
     There is now vigorous interest in non-semisimple TQFTs and associated skein constructions: see \cite{DGG+223DimensionalTQFTsNonsemisimple,CGHP23Skein3+1TQFTsNonsemisimple,CGP23AdmissibleSkeinModules} for a selection of recent developments. We intend to explore in future work how our skein-theoretic approach can be used to introduce defects in non-semisimple TQFTs.

     It is often useful to situate defect TQFTs in a higher categorical framework, with \cite{CMS203dimensionalDefectTQFTs} a notable example of extracting the higher categorical essence of defect TQFTs in dimension 3. Both semisimple \cite{Coo23ExcisionSkeinCategories} and non-semisimple \cite{BH24SkeinCategoriesNonsemisimple,RST24ExcisionSpacesAdmissible} skein theory can be related to factorization homology: an elegant higher categorical tool which makes precise the sense in which skein theory ``globalizes'' graphical calculi as a colimit. We expect that, through the connections established in this paper, we will be able to yield new insight on defect RT theories by interpreting them through the lens of factorization homology, in particular the process of gauging symmetry defects.

     \begin{figure}
      \centering
      \includesvg[width=5cm]{BrTens.svg}  
      \caption{A schematic in the style of \cite{BJS21DualizabilityBraidedTensor}, showing how our defect data fits into a higher algebraic setting such as $\mathbf{BrTens}$. The $\cC$-labelled bulk represents an object, corresponding to the 4d CY theory. Its regular boundary (shown as an interface with the trivial theory, labelled $\One$), also labelled $\cC$, represents a 1-morphism $\One \to \cC$, and corresponds to the RT theory. Our defects lie within the regular boundary. They are not the boundary of interior surface or line defects.}    \label{fig:BrTensCube}
     \end{figure}

    Let us remark that the TQFTs obtained from the above factorization homology constructions are (categorified versions of) the 4d Crane--Yetter theory, originating in \cite{CY93CategoricalConstruction4d}. The state spaces of this theory are isomorphic to skein modules, a connection going back to \cite{BHMV95TopologicalQuantumField}, revisited in the non-semisimple setting in \cite{CGHP23Skein3+1TQFTsNonsemisimple}. The CY theory is the anomaly theory for the RT theory, which is regarded as its regular boundary \cite{Tha21CategoryBoundaryValues}. Since our defects lie only in the boundary of the 3-manifold $M$, and do not extend to the bulk, they should be considered self-defects of the RT theory which do not come from defects in the CY anomaly theory. This is in contrast to the surface defects which appear in \cite{BJ25ParabolicSkeinModules,Gar26HOMFLYParabolicRestriction}, which are defects in the CY theory itself. 
    In the context of extended TQFTs valued in the higher Morita theory $\mathbf{BrTens}$, the surface defects of  \cite{BJ25ParabolicSkeinModules,Gar26HOMFLYParabolicRestriction} correspond to 1-morphisms $\cC \to \cC$, whereas our line defects would correspond to 2-morphisms $D \to D$, where $D$ is a 1-morphism $\One \to \cC$ corresponding to the regular right $\cC$-module (alias Dirichlet boundary condition), see Figure~\ref{fig:BrTensCube}.
    One can also consider different 1-morphisms $D,E : \One \to \cC$ to label the two-dimensional boundary regions, and accordingly 2-morphisms $D \to E$ for the line defects. The corresponding construction in RT TQFT with defects has been developed in \cite{FSV13BicategoriesBoundaryConditions,KMRS22DomainWalls3d,CM26OrbifoldCompletion3Categories}, and it is natural to generalize our defect skein module construction and the TQFT comparison to this case.
    
    We expect that the stratified version of factorization homology of \cite{AFTFactorizationHomologyStratified2016} will provide an appropriate vehicle to interpret our defect constructions as coming from such higher categorical data; an alternative technology, well-suited to derived enhacements of the defect story, is found in \cite{MW12BlobComplex}. We intend to probe these connections in future work.

    From a different perspective, the projector-based approach to defects in the semisimple setting may help to better understand these newly emerging higher-algebraic and skein-theoretic defect constructions. The projector approach has the advantage of presenting state spaces in terms of an explicit linear map: in a sense, this construction tells us that the added complexity we encounter when adding defects to skein theory is on the order of linear algebra. We hope to explore the extent to which the projector approach can be generalized to yield a feasibly computable approach to non-semisimple skein theory and factorization homology.

\medskip
\noindent
\textbf{Outline.}
    The outline of the paper is as follows. We recall the necessary algebraic and topological background in Section~\ref{s-background}. The algebraic background is standard, and the topological background is a minor adaptation of the theory of stratified manifolds of \cite{CRS19OrbifoldsNdimensionalDefect} to our purposes. In Section~\ref{s-defect-skein} we introduce defect skein modules of 3-manifolds with defects in their boundary labelled by module categories and functors thereof. In Section~\ref{s-defect-RT} we recall the construction of \cite{CRS18LineSurfaceDefects} of the defect version of RT theory. The proof of Theorem~\ref{t-main-thm}, that our defect skein modules are isomorphic to these defect state spaces, is the content of Section \ref{s-main-thm}.

\medskip
\noindent
\textbf{Acknowledgements.}
    We would like to thank Kevin Walker, for illuminating discussions on the formalism of defect skeins. We thank Nils Carqueville and David Jordan for helpful comments on a draft of this paper. The authors acknowledge support by the Deutsche Forschungsgemeinschaft (DFG, German Research Foundation) as part of SFB 1624 – ``Higher structures, moduli spaces and integrability'' – 506632645. IR also acknowledges support by the DFG under Germany's Excellence Strategy - EXC 2121 ``Quantum Universe'' - 390833306.

    \section{Notation and conventions}
    \label{s-background}

    \subsection{Algebra}

    We work over a field $\mathbf{k}$, and assume all $1$-categories to be essentially small.

Recall that
      a category
      is \emph{$\mathbf{k}$-linear} if it is enriched over $\Vect_{\mathbf{k}}$.

    \begin{defn}
      The $(2, 1)$-category $\Rex$ consists of finitely cocomplete $\mathbf{k}$-linear categories, right exact functors, and natural transformations. 
    \end{defn}

    Note that categories in $\Rex$ need not be abelian.
    In the sequel we will assume all categories to be in $\Rex$ and all functors to be $\mathbf{k}$-linear and right exact.

    \begin{defn}
      Given two categories $\cA, \cB \in \Rex$, the \emph{Deligne--Kelly tensor product} is the universal object $\cA \bt \cB \in \Rex$ which receives a $\mathbf{k}$-bilinear functor from $\cA \times \cB$ which is right exact in each variable.
    \end{defn}

    It is shown in \cite{Kel05BasicConceptsEnriched} that such an object exists for any $\cA, \cB \in \Rex$. In the more restrictive setting that $\cA, \cB$ are abelian and have finite-dimensional $\Hom$-spaces and finite length composition series, a tensor product construction was given in \cite{Del90CategoriesTannakiennes}, and was shown to agree with the tensor product in $\Rex$ in \cite{Lop13TensorProductsFinitely}. It was shown in \cite{Kel89ElementaryObservations2categorical} that the Deligne-Kelly tensor product equips $\Rex$ with the structure of a symmetric closed monoidal $(2, 1)$-category.

    We recall that the notions of (braided) monoidal, ribbon, and module categories make sense in various categorical settings. Here we are interested in such structures defined in $\Rex$. We recall here the required data, and do not recall in detail the conditions this data must satisfy, giving instead pointers to where this can be found in a standard textbook reference \cite{EGNO15TensorCategories}.

    \begin{defn}
      A \emph{monoidal category} is a category $\cC \in \Rex$ equipped with a functor $\otimes : \cC \bt \cC \to \cC$, a distinguished object $\One \in \cC$ called the monoidal unit, a natural isomorphism $\alpha : (- \ot -) \ot - \xrightarrow{\sim} - \ot (- \ot -)$ called the \emph{associator}, and natural isomorphisms $\lambda : \One \ot - \to \Id, \rho : - \ot \One \to \Id$ called the \emph{left and right unitors}, such that this data satisfies the pentagon and unit axioms of, e.g., \cite[Eqns.\,2.2,\,2.10]{EGNO15TensorCategories}.
    \end{defn}

    \begin{defn}
      Let $\cC$ be any monoidal category. A \emph{(right) $\cC$-module category} for $\cC$ is a category $\cM \in \Rex$ equipped with an action functor $\lhd : \cM \bt \cC \to \cM$ and a natural isomorphism
      \[
        \{\mu_{X, V, W} : X \lhd (V \ot W) \xrightarrow{\sim} (X \lhd V) \lhd W\}_{X \in \cM, V, W \in \cC}
      \]
      called the \emph{modulator}, satisfying the pentagon axiom of, e.g.\ \cite[Eqn.\,7.2]{EGNO15TensorCategories}, and such that the functor $M \mapsto M \lhd \One$ is an autoequivalence of $\cM$. A \emph{left $\cC$-module category} is defined similarly.
    \end{defn}

    If unspecified, a \emph{$\cC$-module category} means a right $\cC$-module category. 
    We note that as per our convention, $\otimes : \cC \bt \cC \to \cC$ and $\lhd : \cM \bt \cC \to \cM$ are right exact functors.

    \begin{defn}
      A $\cC$-module functor between $\cC$-module categories $(\cM, \lhd_{\cM}, \mu_{\cM})$ and  $(\cN, \lhd_{\cN}, \mu_{\cN})$ is a pair $(F, \rho)$ where $F : \cM \to \cN$ is a functor and $\rho : F(-) \lhd_{\cN} - \xrightarrow{\sim} F(- \lhd_{\cM} -)$ is a natural isomorphism intertwining $\mu_{\cM}$ and $\mu_{\cN}$, called the \emph{module coherence}.
    \end{defn}

     \begin{defn}
      Given a monoidal category $\cC$ and right- and left-module categories $\cM, \cN$, their \emph{relative Deligne-Kelly tensor product} is the universal category $\cM \bt_{\cC} \cN \in \Rex$ which receives a functor $F$ from $\cM \bt \cN$ together with a natural isomorphism
      \[
        (X \lhd V) \bt Y \xrightarrow{\beta} X \bt (V \rhd Y)
      \] 
      called a \emph{balancing}, satisfying the pentagon axiom of \cite[Eqn.\,2]{ENO10FusionCategoriesHomotopy}.
    \end{defn}

    The existence of such an object for any $\cM, \cN \in \Rex$ was established in \cite{BBJ18IntegratingQuantumGroups}. In \cite[Prop.\,3.5]{BBJ18IntegratingQuantumGroups} it is shown that $\Rex$ is closed under small 2-colimits and $\bt$ preserves these in each entry. The relative tensor product can then be defined as the 2-colimit of a bar resolution, but in a 2-category this colimit is equivalent to the colimit of the diagram
    \[\begin{tikzcd}
      {\cM \bt \cC \bt \cN} && {\cM \bt \cN}
      \arrow[shift left, from=1-1, to=1-3]
      \arrow[shift right, from=1-1, to=1-3]
    \end{tikzcd}
    \ ,
    \]
    which unpacks to the definition above. In more restricted settings, the original construction of this balanced tensor product appeared in \cite{ENO10FusionCategoriesHomotopy}.

    To take the relative tensor product of two right $\cC$-module categories, we need a way to turn one of them into a left module category. This is possible when $\cC$ is braided.

    \begin{defn}
      A \emph{braided monoidal category} $\cC$ is a monoidal category equipped with a natural isomorphism $\{\sigma_{X, Y} : X \ot Y \xrightarrow{\sim} Y \ot X\}_{X, Y \in \cC}$ satisfying the hexagon axioms of, e.g., \cite[Eqns.\,8.1,\,8.2]{EGNO15TensorCategories}.
    \end{defn}
    
    \begin{defn}
      \label{d-left-right-C-mod-cat}
      Let $\cC$ be a braided monoidal category with braiding $\sigma$, and $\cM$ a right $\cC$-module category with modulator $\mu$. Then there are left module category structures $(\cM, \rhd, \nu^{\pm})$ with action 
      \[
        V \rhd X := X \lhd V
      \]
      and modulator
      \[
        \nu^{\pm}_{V, W, X} = \mu_{X, W, V} \circ (\id_X \lhd 
        \sigma^{\pm})
        : (V \ot W) \rhd X \xrightarrow{\sim} V \rhd (W \rhd X) \ ,
      \]
where $\sigma^+ = \sigma_{V,W}$ and $\sigma^- = \sigma^{-1}_{W, V}$. 
      There are similarly two different right module categories which can be formed from any left module category, again given by a choice of whether to use the usual or inverse braiding of $\cC$ to define a modulator.
    \end{defn}

    \begin{conv}
      \label{conv-positive-br-for-rel-tp}
      We will take the convention that, whenever two right $\cC$-module categories $\cM, \cN$ are considered, the category $\cN$ is regarded as a left module category via $\nu^{+}$ for the purposes of forming the relative tensor product $\cM \bt_{\cC} \cN$.
    \end{conv}

    A (not necessarily braided) monoidal category is said to be \emph{rigid} if every object has both left and right duals (see e.g.\ \cite[\S 2.10]{EGNO15TensorCategories}).
    
    \begin{defn}
      A \emph{twist isomorphism} on a braided monoidal category $\cC$ is a natural isomorphism $\theta \in \Aut(\id_{\cC})$ satisfying
      \[
        \theta_{X \ot Y} = (\theta_X \ot \theta_Y) \circ \sigma_{Y, X} \circ \sigma_{X, Y}.
      \]
      If $\cC$ is rigid and is equipped with a twist such that
      \[
        \theta_{X^{\vee}} = (\theta_X)^{\vee}
      \]
      for all $X \in \cC$, then $\cC$ is called a \emph{ribbon category}.
    \end{defn}
    
    Recall a rigid monoidal category is called \emph{pivotal} if equipped with a monoidal natural isomorphism $p : \Id \to (-)^{\vee \vee}$. Clearly in a pivotal category, left and right duals coincide up to isomorphism. Moreover, any ribbon tensor category has a canonical pivotal structure induced by the twist (see e.g.\ \cite[p. 218]{EGNO15TensorCategories}).
    
    In the presence of a ribbon structure, we can also define natural module category structures on the opposite category $\cM^{\op}$ of a $\cC$-module category.

    \begin{defn}
      \label{d-module-structure-on-op}
      Let $\cC$ be a ribbon category and $(\cM, \lhd, \mu)$ a right $\cC$-module category. There is a left $\cC$-module category $(\cM^{\op}, \rhd, \nu)$ defined by
      \[
        V \rhd X := X \lhd V^{\vee}
      \]
      and
      \[
         \nu_{V, W, X} : (V \ot W) \rhd X = X \lhd (V \ot W)^{\vee} \xrightarrow{\sim} X \lhd (W^{\vee} \ot V^{\vee}) \xrightarrow{\mu^{-1}} (X \lhd W^{\vee}) \lhd V^{\vee} = V \rhd (W \rhd X)
      \]
      where $\mu^{-1}$  is regarded as a morphism in $\cM^{\op}$ in the above.
    \end{defn}

    \begin{conv}
      \label{conv-op-as-right-module}
      Let $\cC$ be a ribbon category. Given a right $\cC$-module category $\cM$, then unless otherwise stated we regard $\cM^{\op}$ as a right $\cC$-module category by first regarding it as a left $\cC$-module category as in Definition~\ref{d-module-structure-on-op}, and then converting this to a right $\cC$-module category as in Definition~\ref{d-left-right-C-mod-cat} using the inverse braiding convention for the modulator.
    \end{conv}

Note that by using the inverse braiding, when combined with Convention~\ref{conv-positive-br-for-rel-tp} for turning a right module category into a left one, the resulting left $\cC$-module structure on $\cM^{\op}$ just involves duals and no braidings.

    For the construction of the defect TQFT in Section~\ref{s-defect-RT}, we need the following type of braided monoidal category:

    \begin{defn}
      A \emph{modular fusion category} is a ribbon category $\cC$ which is moreover a finite semisimple abelian category, has simple monoidal unit, and has 
      \[
        Z_2(\cC) \simeq \Vect_{\mathbf{k}}
      \]
      where $Z_2(\cC) = \{ X \in \cC : \forall Y \in \cC, \sigma_{Y, X} \circ \sigma_{X, Y} = \id\}$ is the M\"{u}ger centre of $\cC$.
    \end{defn}

    \begin{notn}
      \label{n-Irr-coend}
      When $\cC$ is a modular fusion category, we will denote by $\Irr(\cC)$ the set of isomorphism classes of simple objects of $\cC$. The \emph{canonical coend} of $\cC$ is the object $\cL = \bigoplus_{i \in \Irr(\cC)} X_i^{\vee} \ot X_i$ which universally receives a morphism from every $X^{\vee} \ot X$. 
    \end{notn}

    Recall that when $\cC$ is a finite semisimple monoidal category with simple monoidal unit, and $\cM$ a finite semisimple module category, then $\cM \simeq A\Mod_{\cC}$ as $\cC$-module categories, for some algebra object $A$ \cite[Thm.\,1]{Ost01ModuleCategoriesWeak}. In our setting, we will consider module categories of the form $\cM \simeq A\Mod_{\cC}$ where $A$ has moreover the structure of a $\Delta$-separable symmetric Frobenius algebra, which corresponds to a $\cC$-module trace on $\cM$ (see \cite{Sch13TracesModuleCategories} for this notion). The structure of $\Delta$-separable Frobenius algebra does not need semisimplicity for its definition.

    \begin{defn}
      Let $\cC$ be a ribbon category. A \emph{Frobenius algebra} in $\cC$ is a tuple $(A, \mu, \eta, \Delta, \epsilon)$ such that $(A, \mu, \eta)$ is an associative algebra object in $\cC$, $(A, \Delta, \epsilon)$ is a coassociative coalgebra object, and the Frobenius property 
      \[
        \includesvg[width=6cm]{Frob.svg}
      \]
      holds (here using the graphical calculus of $\cC$, interpreting vertices as $\mu$ or $\Delta$ by the orientations of incident edges). A Frobenius algebra is called \emph{symmetric} if 
      \[
        \epsilon \circ \mu = \epsilon \circ \mu \circ \sigma_{A, A} \circ (\id_A \ot \theta_A).
      \] 
      A Frobenius algebra is called \emph{$\Delta$-separable} if
      \[
        \mu \circ \Delta = \id_A.
      \]
    \end{defn}

    Recall that given algebra objects $(A_i, \mu_i, \eta_i)_{i \in \{1, 2\}}$, we have an algebra structure on $A_1 \ot A_2$ with multiplication given by
    \[
      \includesvg[width=4cm]{prod-mult.svg}
    \]
    and unit $\eta_1 \ot \eta_2$. Iteratively, we can form the tensor product $A_1 \ot \dots \ot A_n$ of several algebras.

    \begin{defn}
      Let $\cC$ be a braided tensor category and $(A_i)_{i = 1}^n$ algebra objects in $\cC$. Then a right (resp. left) \emph{multimodule} over the algebras $(A_i)_{i = 1}^n$ is a right module object for the algebra $A_1 \ot \dots \ot A_n$. 
    \end{defn}

    It is shown in \cite[Definition-Lemma 2.2]{CRS18LineSurfaceDefects} that a multimodule structure on $X$ is equivalent to $X$ having the structure of an $A_i$-module for $1 \leq i \leq n$ and where, for all $1 \leq i < j \leq n$, we have 
    \[
      \includesvg[width=6cm]{multimodule.svg}.
    \]

    \subsection{Topology}

    All manifolds are assumed smooth and oriented and denoted by capital letters.

    \begin{conv}
      We use the \emph{inward normal first} convention to orient the boundary of an $n$-manifold compatibly. Then given a $3$-manifold $M$ with boundary $\Surface$, the oriented bundles $\R \oplus T\Surface$ and $TM\res_{\Surface}$ are isomorphic by identifying the standard framing of the constant bundle $\R$ with the inward-pointing normal. A positively oriented framing of $T\Surface$ will be denoted in the coordinates $(y, z)$ and a positively oriented framing of $TM$ will be denoted in coordinates $(x, y, z)$.
    \end{conv}

    \begin{defn}
      \label{d-strat-manifold}
      A \emph{stratified $n$-manifold} is an $n$-manifold $M$ without boundary, equipped with a filtration $M = F_n \supset F_{n-1} \supset \dots \supset F_0 \supset \emptyset$ such that
      \begin{itemize}
        \item Each $F_j \backslash F_{j-1}$ is a union of finitely many smoothly embedded connected $j$-manifolds, called \emph{$j$-strata}, each of which is oriented (and the orientation agrees with that of $M$ for $j = n$).
        \item For $s$ an $i$-stratum and $t$ a $j$-stratum with $s \cap \overline{t} \neq \emptyset$, then $s \subset \overline{t}$. In this case necessarily $i < j$ and we say $s$ and $t$ are \emph{incident} to each other.
      \end{itemize} 
    \end{defn}

    \begin{defn}
      \label{d-strat-manifold-boundary}
      A \emph{stratified $n$-manifold with boundary} is an $n$-manifold $M$, possibly with boundary, with a filtration as in Definition~\ref{d-strat-manifold} making $\interior{M}$ a stratified manifold, and such that each stratum satisfies $\partial s = s \cap \partial M$. In particular, this makes $\partial M$ a stratified manifold.
    \end{defn}

    \begin{defn}
      A map $f : M \to N$ of stratified manifolds with boundary is a smooth map that sends strata to strata.
    \end{defn}

    \begin{rmk}
      The definitions above match those of \cite[\S 2.1]{CRS19OrbifoldsNdimensionalDefect}, except that in \cite{CRS19OrbifoldsNdimensionalDefect} the ambient manifold is considered as a topological manifold, whereas we consider the ambient manifold as a smooth manifold. This is so that we can use tangential structures to keep track of how strata (and later ribbon graphs) interact with 0-strata. In \cite{CRS19OrbifoldsNdimensionalDefect} the ambient dimension is arbitrary, whereas we only work in dimensions 2 and 3, where all topological manifolds admit a canonical smoothing.
    \end{rmk}

    We only consider stratified manifolds satisfying a regularity condition on the possible local neighbourhoods at each point, and equipped with suitable framing data. The regularity condition can be defined generally in an inductive way \cite[\S 2.2]{CRS19OrbifoldsNdimensionalDefect}, but we give it directly in the dimension $2$ (and in dimension $3$ in \S \ref{ss-defect-bord}). 

    \begin{figure}
      \centering
      \includesvg[width=12cm]{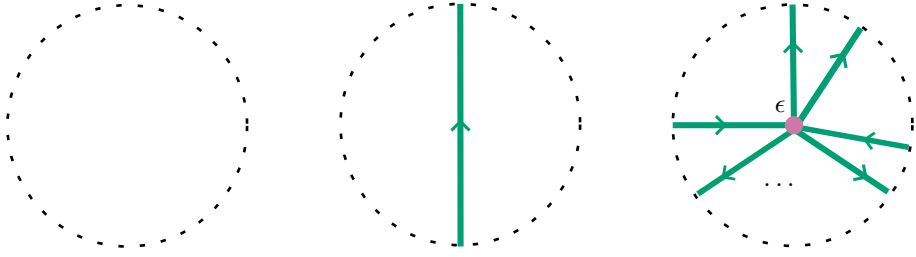}
      \caption{The possible neighbourhoods of points for a regular stratified surface. Note that the rightmost example is one of an infinite family of such neighbourhoods, where arbitrarily many line defects may meet a point defect in a star-like configuration, all having arbitrary orientation.}
      \label{f-nbhds}
    \end{figure}

    \begin{defn}
      \label{d-reg-strat-surf}
      A \emph{regular stratified surface} is a closed, compact stratified surface such that each point of $\Surface$ has a neighbourhood which is diffeomorphic (as a stratified manifold) to one of the open sets shown in Figure \ref{f-nbhds}. We denote a regular stratified surface by $(\Surface, \mathbf{L}, \mathbf{P})$, for $\mathbf{L}$ the set of 1-strata and $\mathbf{P}$ the set of 0-strata.
    \end{defn}
  
    To specify the framing data we require, we must carefully treat 1-strata. 

    \begin{defn}
      \label{d-almost-embedded}
      Let $M$ be a smooth manifold and  $e : [0, 1] \to M$ a continuous map where possibly $e(0) = e(1)$, such that 
      \begin{itemize}
        \item $e \res_{(0, 1)}$ is a smooth embedding,
        \item if $e(0) = e(1)$, then $e$ descends to a continuous embedding of $S^1$ in $M$, and otherwise $e$ is a continuous embedding of $[0, 1]$ in $M$.
      \end{itemize}
      We call such a map an \emph{almost-embedded 1-manifold}, and write $e$ also for the image of this map.
    \end{defn}

    For $e$ an almost-embedded 1-manifold with $e(0) = e(1) = P$, we note that each segment 
$e([0,\epsilon))$ and $e((1-\epsilon,1])$    
    incident to $P$ is separately smoothly embedded, and we can still talk about the tangent vector and normal bundle for each segment separately. We can therefore still talk of the tangent and normal bundles for $e$ as bundles over $[0, 1]$, except that they do not embed into the tangent bundle of $M$, and when talking of the tangent vector etc at $v$ we must really specify a segment of $e$ at $P$. In the sequel we do so without further comment.

    \begin{figure}
      \centering
      \includesvg[width=8cm]{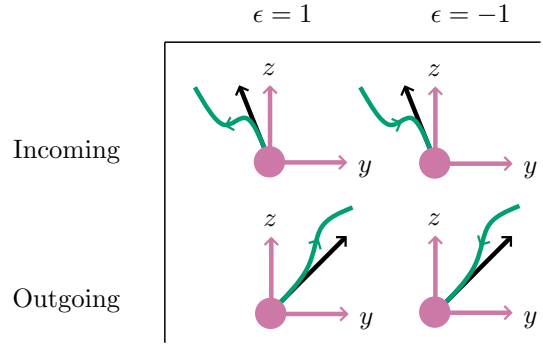}
      \caption{At a framed point $P$ (shown as an enlarged dot with its framing $(y, z)$), a segment $s$ (green) is incident. Whether $s$ is incoming or outgoing depends on whether the tangent vector $\partial_s$ makes a negative or positively oriented basis for the tangent space when combined with $z$. Inother words, $z$ bisects the tangent space into an incoming and an outgoing direction, with $y$ pointing to the outgoing direction. The sign $\epsilon$ of $s$ at $P$ records whether the orientation of $s$ agrees with that of $\partial_s$ or not.}
      \label{f-tangents}
    \end{figure}

    \begin{defn}
      \label{d-tangent-vectors}
      Given a point $P$ in a manifold $M$, and a segment $s$ of an almost-embedded 1-manifold incident to $P$, this defines a tangent vector $\partial_s \in T_P M$. If $s$ is oriented, we define the \emph{sign} of $s$, denoted $\epsilon$, to be $1$ if $\partial_s$ and $s$ are similarly oriented, and $-1$ if they are oppositely oriented. We call $\partial_s$ the \emph{tangent vector} of $s$ at $P$ and $\epsilon \partial_s$ the \emph{oriented tangent vector} of $s$ at $P$. See Figure \ref{f-tangents}.
    \end{defn}

    \begin{defn}
      \label{d-p-fr-strat-surf}
      A \emph{p-framed} (short for point-framed) stratified surface is a regular stratified surface $(\Surface, \mathbf{L}, \mathbf{P})$ such that  
      \begin{itemize}
        \item each $P \in \mathbf{P}$ is equipped with a positively oriented framing $(y_P, z_P)$ of $T_P\Surface$
        \item at each $P \in \mathbf{P}$, the tangent vectors of incident segments of 1-strata are all distinct and not colinear with $z_P$.
      \end{itemize}
    \end{defn}

    \begin{defn}
      \label{d-in-out-surf}
      Let $(S, \mathbf{L}, \mathbf{P})$ be a p-framed stratified surface, and $P$ be a 0-stratum with framing $(y_P, z_P)$ and $s$ a segment incident to $P$. We say that $s$ is \emph{incoming} to $P$ if $(\partial_s, z_v)$ is a negatively oriented basis of $T_P \Surface$ and is \emph{outgoing} from $P$ if $(\partial_s, z_v)$ is a positively oriented basis. See Figure \ref{f-tangents}.
    \end{defn}

    \begin{defn}
      \label{d-standard-ordering-surf}
      The set of incoming interior segments at a 0-stratum $P$ are ordered by increasing $z_P$-coordinate of their normalized tangent vectors, as are the set of outgoing boundary segments. In the sequel, we will call this the \emph{standard ordering} on whichever of these sets of segments is under consideration.
    \end{defn}

    \begin{defn}
      A \emph{boundary-stratified $3$-manifold} is a $3$-manifold whose boundary $\partial M$ is a p-framed stratified surface.
    \end{defn}

    \begin{figure}
      \centering
      \includesvg[width=10cm]{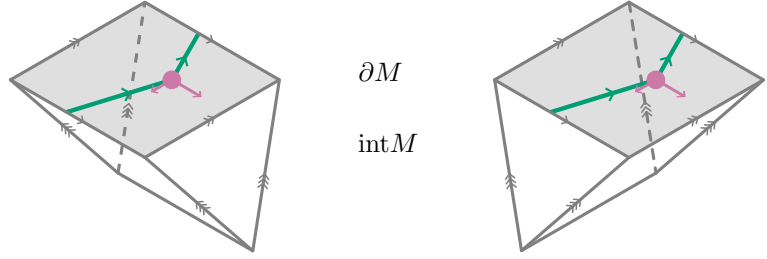}
      \caption{The solid tori of Example \ref{eg-solid-tori}. Left: $K_{-}$. Right: $K_{+}$. The longer framing vector at the point defect is the $y$-vector, and the shorter vector is the $z$-vector.}
      \label{f-solid-tori}
    \end{figure}

     \begin{eg}
      \label{eg-solid-tori}
      Consider a 2-torus $T$ with a chosen basepoint $P$, and let $\alpha, \beta$ be generators of $\pi_1(T^2, P) \cong \Z\alpha \oplus \Z\beta$. We consider the $p$-framed stratified surface $(T^2, L, P)$ which is a 2-torus with a single 1-stratum $L$ homotopic to $\alpha$, where $P$ is endowed with a framing $(y, z)$ such that the first segment of $L$ after $z$ is oriented outwards from $P$, and the next segment is oriented inwards, in the standard ordering, and both segments are incoming at $P$. We consider two mapping classes of diffeomorphisms $\delta_{\pm} : T^2 \xrightarrow{\sim} \partial(D^2 \times S^1)$ realizing $(T^2, L, P)$ as the boundary of two different boundary-stratified 3-manifolds. Here, $\delta_{+}$ maps $\alpha$ to a longitude of $D^2 \times S^1$, and $\beta$ to the meridian; $\delta_{-}$ maps $\alpha$ to the meridian of $D^2 \times S^1$, and $\beta$  to a longitude. We denote the respective boundary-stratified 3-manifolds by $K_{\pm}$. See Figure \ref{f-solid-tori}.
    \end{eg}

    \section{Defect skein modules}
    \label{s-defect-skein}

    We define the defect skein module of a boundary-stratified 3-manifold $M$. We first define the notion of a ribbon graph in such a manifold. We then describe how the strata are marked by algebraic data (becoming \emph{defects}), and how ribbon graphs are compatibly coloured. 
    We impose local relations on our coloured ribbon graphs, the \textit{skein relations}, which results in the skein module of a defect manifold. 
    We identify important special cases of these relations, and close with some example calculations.

    \subsection{Ribbon graphs in stratified manifolds}

    \begin{defn}
      By a \emph{graph} we mean a collection of edges and vertices, where edges are oriented and end at either 0, 1 or 2 vertices. By an \emph{embedded graph} $\Gamma$ in a smooth manifold $M$, we mean an injective map $\Gamma \hookrightarrow M$ under which each edge is either a smoothly embedded circle or an almost-embedded 1-manifold (in the sense of Definition~\ref{d-almost-embedded}).
    \end{defn}

    \begin{defn}
      \label{d-well-framed}
      Given an almost-embedded 1-manifold $e$ in an oriented 3-manifold $M$, a \emph{framing} of $e$ is a trivialization $\nu_e$ of the normal bundle of $e \hookrightarrow M$. Suppose $p$ is a point which meets $e$ and is equipped with a positively oriented framing $(x_p, y_p, z_p)$ of $T_pM$. We say $e$ is \emph{well-framed} with respect to the framing at $p$ if it is equipped with a framing $\nu_e$ such that $(\nu_e, \epsilon \partial_e, z_p)$ is a positively oriented framing of $T_pM$, where $\epsilon \partial_e$ is the oriented tangent vector of $e$ at $p$.
    \end{defn}

    \begin{defn}
      \label{d-strat-rib-graph}
      Let $M$ be a boundary-stratified 3-manifold. We say a \emph{ribbon graph} is an embedded graph $\Gamma$ in $M$ such that 
      \begin{itemize}
        \item The intersection $\Gamma \cap \partial M$ is precisely along the strata $\bigcup_{\mathbf{L}} L \cup \bigcup_{\mathbf{P}} P$, with the orientation of edges agreeing with the orientation of the 1-strata which they meet.
        \item Every vertex $v$ is equipped with a positively oriented framing $(x_v, y_v, z_v)$ of the tangent space $T_v M$.
        \item At every 0-stratum $P$ having framing $(y_P, z_P)$, there is a vertex $v$ whose framing agrees with $(\nu_P, y_p, z_P)$ for $\nu_P$ the inward-pointing normal at $P$.
        \item Every edge $e$ is equipped with a framing such that $e$ is well-framed with respect to the framings of the vertices to which it is incident.
        \item For every segment $s$ incident to a vertex $v$, the tangent vector $\partial_s$ is in the  $x_v y_v$-plane, unless $s$ lies in $\partial M$ (in which case the tangent vector necessarily lies in the $y_v z_v$-plane). 
        \item For every vertex $v$ in the interior of a line defect, with framing $(x_v, y_v, z_v)$, the vector $y_v$ agrees with the signed tangent vector $\epsilon \partial_s$ of the boundary segments incident to $v$ (there are precisely two such segments, and they have the same signed tangent vector).
        \item At every vertex $v$, tangent vectors $\{\partial_s\}$ of all incident segments are distinct and not colinear with $x_v$.
      \end{itemize}
      See Figures \ref{f-graph-coupon} and \ref{f-approach-planes} for examples of vertices and incident edges in a ribbon graph.
    \end{defn}

    \begin{notn}
      When $\Gamma$ is a ribbon graph, we call the edges and vertices which are contained in the boundary of $M$ 
      the \emph{boundary edges/vertices/segments}, and all other ribbons and coupons the \emph{interior edges/vertices/segments}.
    \end{notn}

    \begin{defn}
      By an \emph{isotopy} of ribbon graphs, we mean an ambient isotopy of $M$ which fixes the 0-strata, and fixes the 1-strata non-pointwise. In particular, in the interior of $M$ this is just an ambient isotopy, while in the boundary, vertices may be isotoped along the stratum in which they sit. In the sequel, we will always mean such an isotopy when discussing ribbon graphs.
    \end{defn}

    We saw in Definitions \ref{d-in-out-surf} and \ref{d-standard-ordering-surf} the notion of incoming/outgoing line defects and the standard ordering of line defects at a point defect. This notion can be extended to vertices of a ribbon graph.

    \begin{defn}
      \label{d-in-out}
      Let $v$ be a vertex of a ribbon graph with framing $(x_v, y_v, z_v)$ and $s$ a segment incident to $v$. We say that $s$ is \emph{incoming} to $v$ if $(x_v, \partial_s, z_v)$ is a negatively oriented framing of $T_v M$ and is \emph{outgoing} from $v$ if $(x_v, \partial_s, z_v)$ is a positively oriented framing. See Figure \ref{f-graph-coupon}.
    \end{defn}

    \begin{defn}
      \label{d-standard-ordering}
      The set of incoming interior segments at a vertex $v$ (where $v$ may be in the boundary or not) are ordered by increasing $x_v$-coordinate of their normalized tangent vectors, as are the set of outgoing interior segments. 
      Where $v$ lies at a point defect, then the set of incoming boundary segments are ordered as in Definition \ref{d-standard-ordering-surf}, as are the set of outgoing boundary segments. In the sequel, we will call this the \emph{standard ordering} on whichever of these sets of segments is under consideration. See Figures \ref{f-graph-coupon} and \ref{f-approach-planes}.
    \end{defn}

    \begin{notn}
      \label{n-ribbons-and-coupons}
      We can thicken edges in the direction of the framing into embedded rectangles, called \emph{ribbons}. At vertices which are not at 0-strata, we can insert a small copy of $[0, 1] \times [0, 1]$ in the $x_v y_v$-plane, and connect incoming ribbons to the bottom edge, and outgoing ribbons to the top edge. The ribbons are connected in the standard ordering. See Figure \ref{f-graph-coupon}. This alternative ribbons-and-coupons description is standard in much of the skein theory literature, and so we will also refer to the edges of ribbon graphs as \emph{ribbons}, and the vertices as \emph{coupons}, and will sketch the ribbon graphs in this formalism where convenient. We may also omit framings from some diagrams where appropriate.
    \end{notn}

    \begin{figure}
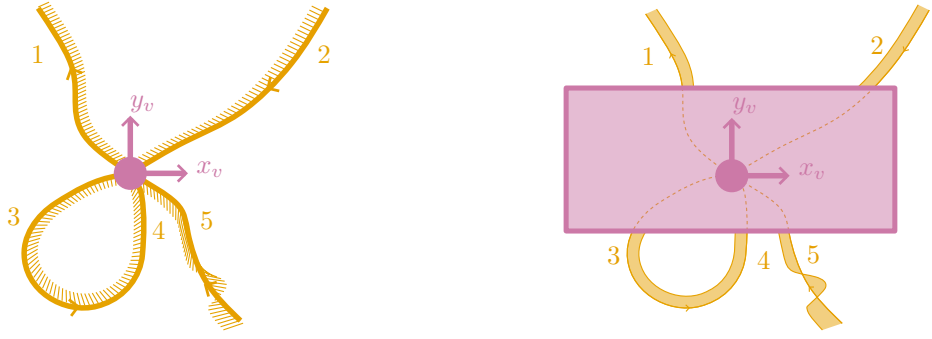

      \begin{subfigure}{0.48\textwidth}
        \centering
        \includesvg[width=5cm]{graphs-coupons-1.svg}
        \caption{An interior vertex of a ribbon graph, where the $z$-coordinate is
        pointing out of the page.}
        \label{f-graphs-coupons-1}
      \end{subfigure}
      \begin{subfigure}{0.48\textwidth}
        \centering
        \includesvg[width=5cm]{graphs-coupons-2.svg}
        \caption{The same arrangement described in terms of ribbons and coupons.}
        \label{f-graphs-coupons-2}
      \end{subfigure}
     
      \caption{(\ref{f-graphs-coupons-1}) shows an interior vertex of a ribbon graph, in the sense of Definition \ref{d-strat-rib-graph}. The edges 1 and 2 are outgoing, and are numbered in the standard ordering. There are only 2 edges attached as incoming, but 3 incoming segments, numbered 3, 4, 5 in the standard ordering, with segments 3 and 4 coming from the same edge. The framing of each edge is indicated. (\ref{f-graphs-coupons-2}) shows the same vertex in the ribbons-and-coupons representation, see Remark \ref{n-ribbons-and-coupons}.}
      \label{f-graph-coupon}
    \end{figure}

    \begin{figure}
      \centering
      \includesvg[width=6cm]{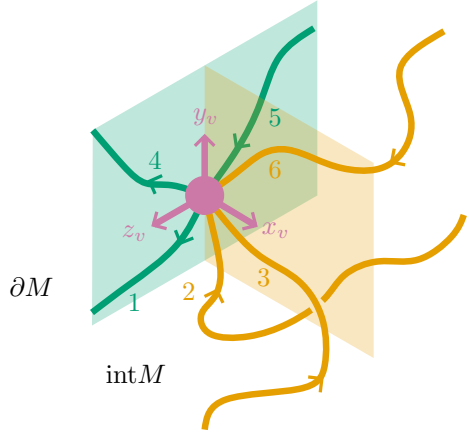}
      \caption{A vertex at a point defect. Boundary edges (green) necessarily approach in the plane which is tangent at the boundary, shaded green, while interior edges (orange) approach in the orthogonal $xy$ plane, shaded orange. The edges 1, 2, 3 are incoming while edges 4,  5, 6  are outgoing. The standard ordering on incoming interior edges is 2 then 3; the standard ordering on outgoing boundary edges is 5 then 4.}
      \label{f-approach-planes}
    \end{figure}

    \subsection{Markings and colourings}  

    We now describe how to decorate the stratification of $M$, and more generally, ribbon graphs in $M$, with algebraic data. We first explain the decorations for strata, which we call markings. We will organize the possible markings into a system $\mathfrak{D}^{\cC}$, where we have modified the definitions introduced in \cite{CRS19OrbifoldsNdimensionalDefect, CMS203dimensionalDefectTQFTs}. We refer to a marked stratum as a \emph{defect}. We then go on to define decorations for ribbon graphs which are compatible with a fixed marking, which we call \emph{colourings} of the defect graph.

\medskip 

Recall that by convention all categories and functors are in $\Rex$.

  \begin{defn}
    Let $\cC$ be a ribbon category.  
    We define the associated \emph{defect system} $\mathfrak{D}^{\cC}$ to be the following tuple:
    \begin{itemize}
      \item $\mathfrak{D}^{\cC}_3 = \{ \cC \}$
      \item $\mathfrak{D}^{\cC}_2 = \{ \text{right }\cC\text{-module categories}\}$
      \item $\mathfrak{D}^{\cC}_1 = \coprod_{(M, N) \in \Z_{\geq 0}^2} \mathfrak{F}_{M, N}$ where
      \begin{align*}
        \mathfrak{F}_{M,N} = \{ ((\cM_i)_{i = 1}^M, &(\cN_j)_{j = 1}^N, (F, \rho)) \mid \cM_i, \cN_j \in \cD^{\cC}_2,\\ &(F, \rho) \in \mathrm{Fun}_{\cC}(\cM_1 \bt_{\cC} \dots \bt_{\cC} \cM_M, \cN_1 \bt_{\cC} \dots \bt_{\cC} \cN_N)\}.
      \end{align*}
      For $M$ or $N = 0$ we have the convention that the $\cC$-relative tensor product of $0$ $\cC$-module categories is $\cC$.
    \end{itemize}
  \end{defn}

  \begin{defn}
    Write $T_M(A)$ for the set of length $M$ tuples of elements of a set $A$. Then the \emph{source map} for $\mathfrak{D}^{\cC}$ is the map
    \begin{align*}
      \mathfrak{s} : \mathfrak{D}^{\cC}_1 &\to \coprod_{M \in \Z_{\geq 0}} T_M(\mathfrak{D}^{\cC}_2)\\
      ((\cM_i)_{i = 1}^M, (\cN_j)_{j = 1}^N, (F, \rho)) &\mapsto (\cM_i)_{i = 1}^M.
    \end{align*}
    The \emph{target map} is the map 
    \begin{align*}
      \mathfrak{t} : \mathfrak{D}^{\cC}_1 &\to \coprod_{N \in \Z_{\geq 0}} T_N(\mathfrak{D}^{\cC}_2)\\
      ((\cM_i)_{i = 1}^M, (\cN_j)_{j = 1}^N, (F, \rho)) &\mapsto (\cN_i)_{i = 1}^N.
    \end{align*}
  \end{defn}

   Let us fix once and for all a ribbon category $\cC$.

    \begin{defn}
      \label{d-D^c-marking}
      Given a p-framed stratified surface $\Surface$ 
      (or given a boundary-stratified 3-manifold with boundary $\Surface$) with sets of 1-strata $\mathbf{L}$ and 0-strata $\mathbf{P}$, we say a \emph{$\mathfrak{D}^{\cC}$-marking} is a choice of:
      \begin{itemize}
        \item For each $L \in \mathbf{L}$, a right $\cC$-module category $\cM \in \mathfrak{D}^{\cC}_2$.
        \item For each $P \in \mathbf{P}$ 
        with $M$ incoming segments $(L_i)_{i = 1}^M$ of 1-strata, and $N$ outgoing segments $(L'_j)_{j = 1}^N$ (each tuple ordered in the standard ordering, Definition \ref{d-standard-ordering-surf}), a choice of $D \in \mathfrak{D}^{\cC}_1$ where $\mathfrak{s}(D) = (\cM_{i}^{-\epsilon_{i}})_{i = 1}^M$ and $\mathfrak{t}(D) = (\cN_{j}^{\epsilon_{j}})_{j = 1}^N$. Here $\epsilon$ is the sign of the segment as in Definition \ref{d-tangent-vectors}, and $\cM^1 = \cM, \cM^{-1} = \cM^{\op}$.
        We will usually refer to $D = ((\cM_i)_{i = 1}^M, (\cN_j)_{j = 1}^N, (F, \rho))$ by the functor $(F, \rho)$, sometimes omitting the module coherence $\rho$.     
      \end{itemize}
      We denote the defect marking as $\mathbf{D} = (\{(L_i, \cM_i)\}_{L_i \in \mathbf{L}}, \{(P_j, F_j)\}_{P_j \in \mathbf{P}})$, and refer to the collection $\SSigma = (\Surface, \mathbf{D})$ as a $\mathfrak{D}^{\cC}$-marked manifold, using blackboard bold for defect-marked manifolds. We call marked 0-strata \emph{point defects} and marked 1-strata \emph{line defects}.
    \end{defn}

    \begin{eg}
      \label{eg-marked-solid-tori}
      Let $K_{\pm}$ be the boundary-stratified solid tori of Example \ref{eg-solid-tori}. For $\cM$ any choice of $\cC$-module category,   and  $F : \cM^{\op}  \bt_{\cC}  \cM \to \cC$ a $\cC$-module functor, then $\mathbf{D} = ((L, \cM), (P, F))$ is a defect marking, for $L$ the unique 1-stratum of $K_{\pm}$  and $P$ the unique 0-stratum. We refer to the resulting $\mathfrak{D}^{\cC}$-marked manifolds as $\mathbb{K}_{\pm}$. 
    \end{eg}

    \begin{figure}
      \centering
      \includesvg[width=8cm]{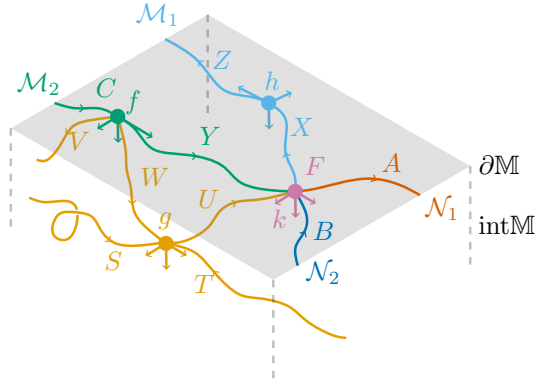}
      \caption{An example of a section of a coloured ribbon graph. Edges and vertices  are coloured according to the category (or functor) marking the stratum in which they lie.}      
      \label{f-point-defect-config}
    \end{figure}

    \begin{defn}
      \label{d-def-col-graph}
      Let $\M = (M, \mathbf{D})$ be a boundary-stratified $\mathfrak{D}^{\cC}$-marked 3-manifold. Given a ribbon graph $\Gamma$ in $M$, we say a \emph{colouring} is a choice of:
      \begin{enumerate}
        \item For each ribbon $e$ in $\Gamma \cap \interior{M}$, an object of $\cC$.
        \item For each ribbon $e$ which is contained in 
        a line defect $(L, \cM)$, an object of $\cM$. 
        \item For each coupon $v$ in $\Gamma \cap \interior{M}$, a choice of morphism in
        \[
          \Hom_{\cC}(\bigotimes_{i} V_i^{-\epsilon_i}, \bigotimes_j W_j^{\epsilon_j})
        \]
        where: $V_i$ are the objects colouring incoming segments and $W_j$ colour the outgoing segments; the tensor products are ordered according to the standard ordering; the $\epsilon_i$ are the signs of the segments, and $V_i^{1} = V_i, V_i^{-1} = V_i^{\vee}$ etc.
        \item For each coupon meeting $\partial M$ away from point defects, with incoming ribbons coloured by $V_i  \in \cC, X \in \cM$ and outgoing edges coloured by $W_i \in \cC, Y \in \cM$, a choice of morphism in 
        \[
        \Hom_{\cM}(X \lhd \bigotimes_{i} V_i^{-\epsilon_i}, Y \lhd \bigotimes_j W_j^{\epsilon_j})
        \]
        with ordering and duality conventions as above.
        \item For each point defect $(P, F)$ with incoming boundary ribbons labelled by objects $(X_i)_{i=1}^M$, incoming interior ribbons labelled by objects $(V_k)_{k=1}^K$, outgoing boundary ribbons labelled by $(Y_j)_{j = 1}^N$, and outgoing interior ribbons labelled by $(W_i)_{l=1}^L$, a choice of morphism
        \[
          f :  F(X \lhd V) \to Y \lhd W.
        \]
        Here $X = X_1 \bt_{\cC} \dots \bt_{\cC} X_M$, and $V = \bigotimes_{k} V_k^{-\epsilon_k}$, both using the standard ordering. The symbols $Y$ and $W$ stand for similar products.
        When $M$ (resp. $N$) is 0, we have that the source (resp. target) of $F$ is $\cC$, and then we require a morphism $f$ as above with $X$ (resp. $Y$) equal to $\One_{\cC}$.
        When $M = N = 0$, we
        require that the point defect is coloured by $\id_{F(\One)}$.
      \end{enumerate}
      We say a ribbon graph equipped with a colouring is \emph{coloured}.
    \end{defn}

     \begin{figure}
      \centering
      \includesvg[width=8cm]{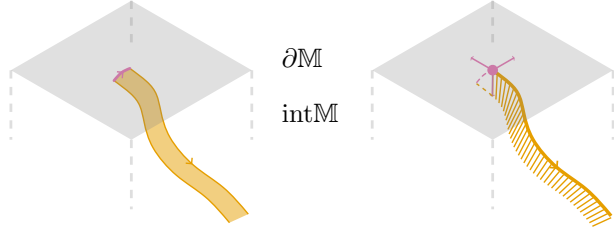}
      \caption{Left: a marked arc (pink) meeting a ribbon from the interior (orange). Right: the arc is replaced by a point defect, given a framing, and the ribbon framing modified to match, as in Remark \ref{rk-0-point-defects-are-markings}.}
      \label{f-turning}
    \end{figure}

    \begin{rmk}
      \label{rk-0-point-defects-are-markings}
      An important situation considered in the skein theory literature is where $M$ does not have a stratification but has a collection of oriented arcs in the boundary labelled by objects of $\cC$, and we consider ribbon graphs which may end compatibly at these arcs. See e.g.\ \cite{Tur10QuantumInvariantsKnots}. In our situation, we can consider such marked arcs as point defects. For $s$     
      an arc which meets an inward-oriented ribbon labelled $X$, let $P$ be the end of $s$ such that the sign $\epsilon$ of $\partial_s$ is $-1$ at $P$. We choose any tangent vector $z_P \in T_P M$ such that $(\nu_P, \epsilon \partial_s, z_P)$ is a positively oriented basis, for $\nu_P$ the inward normal. We mark $P$ by the functor $- \lhd X : \cC \to \cC$ where $X$ was the object labelling the arc and $\lhd$ is the regular right action of $\cC$ on itself. Then a ribbon graph which ended at this arc can be replaced by a ribbon graph in our sense by modifying the framing at $P$ to match our conventions, and adding the label $u_X$ at $P$, for $u$ the inverse to the left unitor of $\cC$. See Figure \ref{f-turning}. If the ribbon was outward-oriented at the arc $s$, we replace $X$ by $X^{\vee}$ in this construction. It follows that our definition subsumes the well-known definition of ribbon graphs ending at marked arcs.
    \end{rmk}

    \subsection{Valid cubes and skein relations}

    We now define the skein relations. These are \emph{local} relations implementing relations in the algebraic structures marking the manifold $M$. By locality, we mean that the relations should hold in embedded balls, which without loss of generality can be taken to be the interiors of links of points in $M$. The defects on the link of the point should specify the structure whose relations are to be implemented.

    We have the following cases:
    \begin{itemize}
      \item If $P$ is a point defect, its link contains incoming and outgoing line defects, which locally can be combined into a single incoming defect $\cM$ and outgoing defect $\cN$. The skein relations at $P$ should implement relations in $\Hom_{\cN}(F(-), -)$ where $F$ is the $\cC$-module functor marking $P$.
      \item If $P$ is a point on a line defect labelled $\cM$, its link contains points marked $\cM$, and the skein relations at $P$ should implement relations in $\Hom_{\cM}(-, -)$.
      \item If $P$ does not lie on point or line defects, then its link is an unstratified $\cC$-marked sphere. Then the skein relations at $P$ should implement the realtions in $\Hom_{\cC}(-, -)$.
    \end{itemize}

    The blanks in the above $\Hom$-space expressions should be filled in by the labels of ribbon graphs intersecting the link of $P$ transversely. Since we want to use $\Hom$-spaces, which have a natural notion of source and target, then instead of embedded balls we will use embedded cubes, where ribbon graphs intersect opposing faces which specify source and target. Examples of cubes corresponding to the three cases above are shown in the top row of Figure \ref{f-valid-cubes}
    
    We also note that each of the above cases is a special case of the last. The relations in $\Hom_{\cM}(-, -)$ are simply the relations in $\Hom_{\cM}(F(-), -)$ for $F$ the identity functor; the relations in $\Hom_{\cC}(-, -)$ are simply the relations in $\Hom_{\cM}(-, -)$ for $\cM = \cC$. We will therefore focus on defining nicely embedded cubes as links of point defects, and skein relations therein. Skein relations along line defects can then be obtained by formally inserting an identity-marked point defect, framed orthogonally to the tangent vector along the line defect. Skein relations on unstratified regions, including the interior, arise by inserting an $\Id_{\cC}$-marked point defect in our definition.

    \begin{figure}
      \centering
      \includesvg[width=9cm]{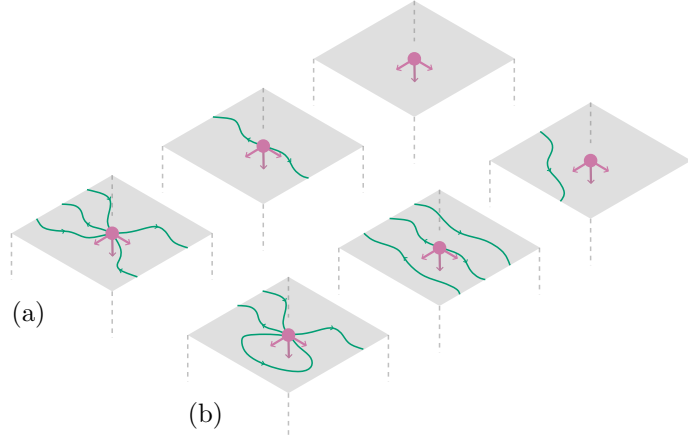}
      \caption{(a) The back row shows valid cubes around a point defect, a point in a line defect, and a point not lying in either kind of defect, respectively. The $yz$-face of the cube is shaded. (b) The front row shows examples of invalid cubes. In each case, the failure is due to the cube not being local enough.}
      \label{f-valid-cubes}
    \end{figure}

    \begin{defn}
      \label{d-ord-cube}
      Let $\M = (M, \mathbf{D})$ be a boundary-stratified $\mathfrak{D}^{\cC}$-marked 3-manifold. Let $C = \{ (x, y, z) \in \R^3 : x \in [0, 1], y, z \in [-1/2, 1/2] \}$ and $\mathbf{C} : C \hookrightarrow M$ be a cube embedded in $M$ such that:
      \begin{itemize}
        \item $\mathbf{C}$ either meets $\partial M$ in the $x = 0$ face or not at all.
        \item If $\mathbf{C}$ meets $\partial M$, then $\mathbf{C}$ meets a unique point defect $P$ at $\mathbf{C}(0, 0, 0)$ such that the pushforward under  $\mathbf{C}$ of the standard framing of $\R^3$ is the framing at $P$. We allow that $P$ is a formally inserted identity defect as described above.
        \item In the above situation, $\mathbf{C}$ only meets line defects which are incident to $P$. Incoming defects meet $\partial \mathbf{C}$ transversely along the edge $y = -1/2$, and outgoing defects meet $\partial \mathbf{C}$ transversely along the edge $y = 1/2$.
      \end{itemize}
      Call a cube satisfying these conditions \emph{valid}. See Figure \ref{f-valid-cubes} for examples and non-examples.
    \end{defn}

    \begin{figure}
      \centering
      \includesvg[width=10cm]{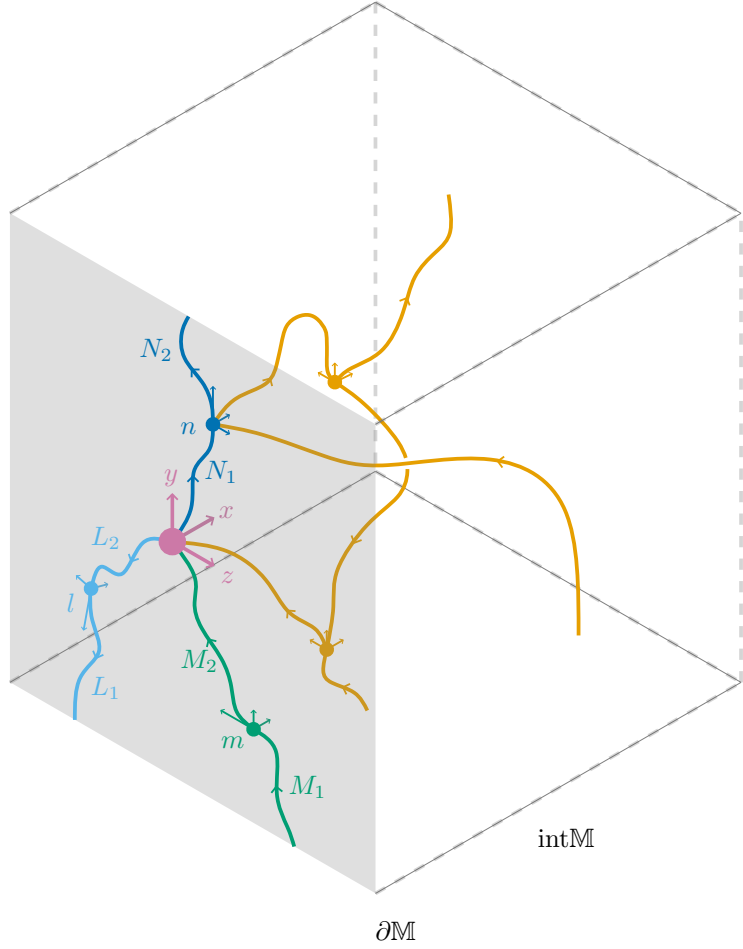}
      \caption{An example of a valid cube, and a section of ribbon graph having good intersection. Note that the cube is parameterized so that the framing at the point defect also specifies the coordinates of the cube. We omit the $\mathfrak{D}^{\cC}$-colours of the interior graph and the point defect.}
      \label{f-valid-cube}
    \end{figure}

    For the purposes of the following definitions, if $\mathbf{C}$ is a valid cube that does not meet $\partial M$, we formally view the $x = 0$ face of $\mathbf{C}$ as a boundary face with an $\Id_{\cC}$-marked point defect at $\mathbf{C}(0, 0, 0)$.

    \begin{defn}
      Let $\mathbf{C}$ be a valid cube in $\M$ and $\Gamma$ a coloured ribbon graph in $\M$. We say that $\Gamma$ has \emph{good intersection} with $\mathbf{C}$ if: 
      \begin{itemize}
        \item $\Gamma \cap \interior{M}$ intersects $\partial \mathbf{C}$ in the $y = -1/2$ and $y = 1/2$ faces, where the $x$-coordinates of the intersections with each face are all distinct.
        \item $\Gamma \cap \mathbf{C}$ is well-framed with respect to its intersection points with the $y = \pm 1/2$-faces of $\mathbf{C}$, where these intersection points are framed by the pushforward of the standard framing of $\R^3$ under $\mathbf{C}$. These intersection points are not permitted to be vertices of $\Gamma$.
      \end{itemize}
    \end{defn}

    \begin{defn}
      \label{d-generic pos}
      Let $\mathbf{C}$ be a valid cube, and $\Gamma$ a coloured ribbon graph having good intersection with $\mathbf{C}$. Denote by $\pi$ the projection to the $xy$-plane of $\mathbf{C}$. We say that $\Gamma$ is in \emph{generic position} (with respect to $\mathbf{C}$) if:
      \begin{itemize}
        \item The only points of $\pi(\Gamma)$ with nontrivial fibres are either double points given by transverse crossings of two interior ribbons, or points with multiple line defects in the fibre.
        \item The vertices of $\Gamma$ have distinct $y$-coordinates.
        \item The tangent vector field along any edge is never colinear with the $z$-coordinate.
        \item Let $v_i$, $m \leq i < 0$ the vertices of $\Gamma$ meeting $\{x = 0, y < 0\}$, and $v_i$, $0 < i \leq n$ the vertices of $\Gamma$ meeting $\{x = 0, y > 0\}$, ordered by increasing $y$-coordinate. We require that, if $v_i$ is labelled by a morphism in $\cM_k$, and $v_j$ is labelled by a morphism in $\cM_l$, then if $k < l$ in the standard ordering, we have $i < j$, and similarly for coupons with positive $y$-coordinate.
        \item For $v_i$ as above, let $t_i$ the $y$-coordinate of $v_i$ under $\pi$. Let also $t_0 = 0$ and $v_0$ the vertex at the point defect. Then we require that for $m \leq i \leq n$, the only vertex in the intersection $\Gamma \cap \{y = t_i \}$ is $v_i$, and all segments not meeting $v_i$ have distinct $x$-coordinate in this intersection, and $\Gamma$ is well-framed with respect to all points in the intersection, where the points distinct from $v_i$ are framed by the push-forward of the standard framing on $\R^3$.
      \end{itemize}
    \end{defn}

    \begin{figure}
      \centering
      \includesvg[width=10cm]{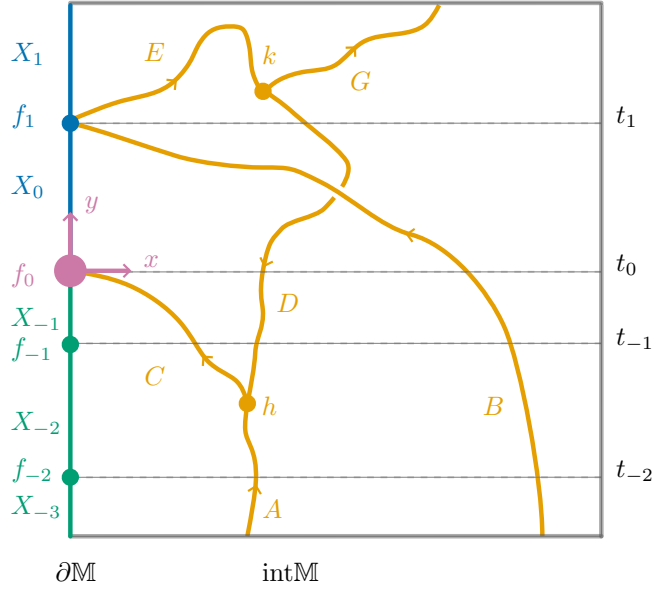}
      \caption{The ribbon graph $\Gamma$ of Figure \ref{f-valid-cube} can be put in generic position so that $\pi(\Gamma)$ is as shown here. Here and in the sequel, when graphs are shown in planar projection, we will assume that the framing of vertices is given by the standard framing with the $xy$-plane being the page, and the $z$-direction pointing out of the page, and we do not depict the framing vectors.}
      \label{f-valid-cube-project}
    \end{figure}

    It is clear that $\Gamma$ can always be brought into generic position by an isotopy in the interior of $\mathbf{C}$. Assume $\Gamma$ is in generic position with respect to $\mathbf{C}$. Write $\cM = \cM_1 \bt_{\cC} \dots \bt_{\cC} \cM_M$ for the source of the functor $F$ labelling the point defect, and $\cN = \cN_1 \bt_{\cC} \dots \bt_{\cC} \cN_N$ for the target. Let the boundary vertices be enumerated $v_i$ as above, with $y$-coordinates $t_i$, and let $t_{m-1} = -1/2, t_{n+1} = 1/2$. We describe an evaluation procedure for the projection $\pi(\Gamma)$ using the following notation (see Example \ref{eg-Ev-notation} and Figure \ref{f-valid-cube-project} for an illustration).
    \begin{itemize}
      \item Let $m-1 \leq i \leq n$, and let $\cX = \cM$ if $i < 0$ and $\cX = \cN$ if $i \geq 0$. The part of $\Gamma$ contained in      
      $\{0\} \times (t_i, t_{i+1}) \times [-1/2, 1/2]$ contains no vertices, and the edges of this part of $\Gamma$ define a unique object of $\cX$. We denote this object by $X_i$, and use this to label the part of $\pi(\Gamma)$ contained in $\{0\} \times (t_i, t_{i+1})$. In  the case that the point defect has no incoming (resp. outgoing) line defects, then we set $X_{-1}$ (resp. $X_0$) equal to  $\One_{\cC}$ in  our notation.
      \item For  $m-1 \leq i \leq n + 1$, the interior edges of $\pi(\Gamma)$ which intersect $(0, 1] \times \{ t_i\}$ define an object $U_i$ of $\cC$, which is the tensor product of the labels of the edges ordered by $x$-coordinate and signed with respect to their intersection point.
      \item For $m \leq i \leq n, i \neq 0$, the point $(0, t_i)$ is labelled by a morphism $f_i : X_{i-1} \lhd V_i \to X_i \lhd W_i$. For $i = 0$, the point $(0, 0)$ is labelled by a morphism $f_0 : F(X_{-1} \lhd V_0) \to X_0 \lhd W_0$. Here $X_i$ are as above. In our notation $V_i$ is the tensor product of the incoming interior edges at $f_i$, ordered by $x$-coordinate and signed with respect to the intersection with $(0, t_i)$; $W_i$ is the tensor product of outgoing interior edges at $f_i$, similarly ordered and signed. We also take the convention that $W_{m-1} = V_{n+1} = \One$.
      \item For each $m-1 \leq i \leq n$, the part of $\pi(\Gamma) \cap [t_i, t_{i+1}]$ which does not meet the $x = 0$ line describes a morphism $g_i$ in $\cC$ via the graphical calculus of $\cC$. Then in this segment we have a morphism
      \[
          \id_{X_i} \lhd g_i \in \Hom_{\cX}(X_i \lhd (W_i \ot U_i) , X_i \lhd (V_{i+1} \ot U_{i+1})).
      \]
      \item We also have for each $m \leq i \leq n, i \neq 0$ the morphism 
      \begin{align*}
        \tilde{f_i} &:= \mu^{-1}_{X_i, W_i, U_i} \circ (f_i \lhd \id_{U_i}) \circ \mu_{X_{i-1}, V_i, U_i}\\
        &\in \Hom_{\cX}(X_{i-1} \lhd (V_i \ot U_i), X_i \lhd (W_i \ot U_i))
      \end{align*}
      where $\mu$ is the appropriate modulator, and $f_i$ is the morphism labelling $v_i$.
      \item For $i = 0$, the corresponding morphism is 
      \begin{align*}
        \tilde{f_0} &:= \mu^{-1}_{X_{0}, W_0, U_0} \circ (f_0 \lhd \id_{U_0}) \circ \rho^{-1}_{X_{-1} \lhd V_0, U_0} \circ F(\mu_{X_{-1}, V_0, U_0})\\
        &\in \Hom(F(X_{-1} \lhd (V_0 \ot U_0)), X_0 \lhd (W_0 \ot U_0))
      \end{align*}
      where $\rho$ is the module coherence of $F$.
    \end{itemize}
    The composition of the morphisms $\id_{X_i} \lhd g_i, \tilde{f_i}$ for $i < 0$ in the ordering given by the $y$-coordinate defines a morphism 
    \[
      g_{< 0} :  X_{m-1} \lhd U_{m-1} \to X_{-1} \lhd (V_0 \ot U_0)
    \]
    in $\cM$. Similarly, composing morphisms $\tilde{f_i}, \id_{X_i} \lhd g_i$ for $i \geq 0$ gives a morphism
    \[
    g_{\geq 0} : F(X_{-1} \lhd (V_0 \ot U_0)) \to X_{n} \lhd U_{n+1}
    \] 
    in $\cN$. 

    \begin{defn}
      Let $\mathbf{C}$ be a valid cube, and $\Gamma$ be a coloured ribbon graph having good intersection with $\mathbf{C}$, with the point defect $P$ at $\mathbf{C}(0, 0, 0)$ marked by a $\cC$-module functor
      \[
        (F, \rho) : \cM \to \cN.
      \]
      We define the morphism
      \begin{align*}
        \Ev_{\mathbf{C}}(\Gamma) := g_{\geq 0} \circ F(g_{<0}) \circ \rho_{X_{m-1}, U_{m-1}} \in \Hom_{\cN}(F(X_{m-1}) \lhd U_{m-1}, X_n \lhd U_{n+1})
      \end{align*}
      by choosing
      an
    isotopy of $\Gamma$ relative to $\Gamma \cap \{y = \pm 1/2\}$ into generic position, with the indices and conventions as above.
    \end{defn}

    \begin{rmk}
      When $\mathbf{C}$ does not meet the boundary of $M$, the assignment $\Ev_{\mathbf{C}}$ is simply the evaluation functor defined by Reshetikhin and Turaev in \cite{RT90RibbonGraphsTheir}.
    \end{rmk}

    \begin{eg}
      \label{eg-Ev-notation}
      Consider the valid cube of Figure \ref{f-valid-cube}, and its projection shown on Figure \ref{f-valid-cube-project}.
      \begin{itemize}
        \item The boundary edges with negative $y$-coordinate are labelled in the projection by $X_{-3} = L_1 \bt M_1, X_{-2} = L_1 \bt M_2$ and $X_{-1} = L_2 \bt M_2$; $f_{-2} = \id_{L_1} \bt m$ and $f_{-1} = l \bt M_2$. 
        \item The boundary edges with positive $y$-coordinate in the projection are labelled by $X_0 = N_1, X_1 = N_2$ and $f_1 = n$. 
        \item The notation $U_i$ refers to the interior edges with positive $x$-coordinate at $y = t_i$, so $U_{-3} = U_{-2} = A \ot B$, $U_{-1} = C \ot D^{\vee} \ot B$, $U_0 = D^{\vee} \ot B$, $U_1 = D^{\vee}$, $U_2 = G$.
        \item The notation $V_i$ refers to the incoming interior edges at $f_i$, so $V_{-2} = V_{-1} = \One$, $V_0 = C$, $V_1 = B$.
        \item The notation $W_i$ refers to the outgoing interior edges at $f_i$, so  $W_{-2} = W_{-1} = W_0 = \One$, $W_1 = E$.
        \item The morphisms $g_i$ are determined by the graphical calculus for $\cC$, so that the non-identity morphisms in this example are $g_{-2} = h \ot \id_{B}$, $g_0 = \sigma^{-1}_{B, D^{\vee}}$, $g_1 = (\ev_E \ot \id_G) \circ k$.
      \end{itemize} 
    \end{eg}

    \begin{figure}
      \centering
      \includesvg[width=10cm]{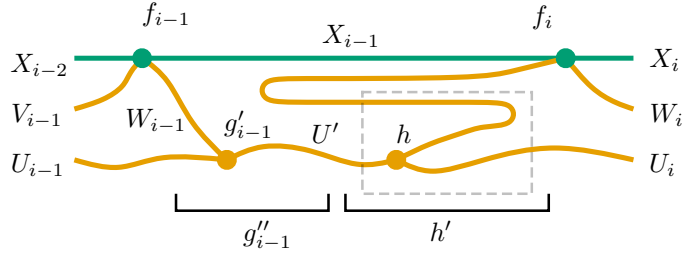}
      \caption{To isotope $h$ past $f_i$, we first put the graph into the position  shown. We will then consider $g_i = h' \circ g_{i-1}''$ in evaluating the graph. We moreover note that $h' = \id_{V_i} \ot h''$ for $h''  : V_i^{\vee} \ot U' \to U_i$ the morphism indicated in the dashed box. Here and in the sequel, the planar projection is rotated so that the $x$-coordinate runs top to bottom and the $y$-coordinate runs left to right; orientations of edges are omitted when they do not serve a clear mathematical or expository purpose.}
      \label{f-coupon-slide}
    \end{figure}

    \begin{lemma}
      \label{l-ev well-defd}
      For $\mathbf{C}$ a valid cube, and $\Gamma$ having good intersection with $\mathbf{C}$, the morphism $\Ev_{\mathbf{C}}(\Gamma)$ is well-defined with respect to the choice of isotopy of $\Gamma$ into generic position.
    \end{lemma}

    \begin{proof}
      First choose any isotopy representative of $\Gamma \cap \mathbf{C}$ which is in generic position. For boundary-vertex $y$-coordinates $t_i$ as above, consider isotopies of $\Gamma$ which are fixed outside of the region $ [0, 1] \times (t_i, t_{i+1}) \times [-1/2, 1/2]$.
      The morphism $g_i \in \Hom_{\cC}(W_i \ot U_i, V_{i+1} \ot U_{i+1})$ is well-defined with respect to these isotopies, which follows from well-definedness of the Reshetikhin--Turaev evaluation functor \cite[\S 5]{RT90RibbonGraphsTheir}. 

      Now consider an isotopy of $\Gamma$ which moves a single interior coupon $h$ in the interval $[t_{i-1}, t_{i}]$ into the interval $[t_{i}, t_{i+1}]$ in the projection $\pi(\Gamma)$. By the previous step, we can assume that the morphism in the strip $[t_{i-1}, t_{i}]$ was of the form shown in Figure \ref{f-coupon-slide}. Assume that $i \neq 0$. Then we note that the segment $[t_{i-1}, t_{i+1}]$ of $\pi(\Gamma)$ describes the morphism
      \[
        (\id_{X_{i}} \lhd g_{i}) \circ \mu^{-1}_{X_i, W_i, U_i} \circ (f_{i} \lhd \id_{U_{i}}) \circ \mu_{X_{i-1}, V_i, U_i} \circ (\id_{X_{i-1}} \lhd h') \circ (\id_{X_{i-1}} \lhd g''_{i-1}).
      \]
      We note that $h' = \id_{V_i} \ot h''$ for $h''$ the morphism indicated in the dashed square in Figure \ref{f-coupon-slide}. By naturality of $\mu$, we have that the above expression is equal to
      \[
        (\id_{X_{i}} \lhd g_{i}) \circ \mu^{-1}_{X_i, W_i, U_i} \circ (f_{i} \lhd \id_{U_{i}}) \circ (\id_{(X_{i-1} \lhd V_i)} \lhd h'') \circ \mu_{X_{i-1}, V_i, V_i^{\vee} \ot U'} \circ (\id_{X_{i-1}} \lhd g''_{i-1}).
      \]
      Using the interchange law
      \[
        (\id \lhd g) \circ (f \lhd \id) = (f \lhd \id) \circ (\id \lhd g)
      \]
      for module categories, which is a result of bifunctoriality of the action, this is equal to the morphism
      \[
        (\id_{X_{i}} \lhd g_{i}) \circ \mu^{-1}_{X_i, W_i, U_i} \circ (\id_{(X_{i} \lhd W_i)} \lhd h'') \circ (f_{i} \lhd \id_{(V_{i}^{\vee} \ot U')}) \circ \mu_{X_{i-1}, V_i, V_i^{\vee} \ot U'} \circ (\id_{X_{i-1}} \lhd g''_{i-1}).
      \]
      By naturality of $\mu^{-1}$ this is equal to  
      \[
        (\id_{X_{i}} \lhd g_{i}) \circ (\id_{X_{i}} \lhd (\id_{W_i} \ot h'')) \circ \mu^{-1}_{X_i, W_i, V_i^{\vee} \ot U'} \circ (f_{i} \lhd \id_{(V_{i}^{\vee} \ot U')}) \circ \mu_{X_{i-1}, V_i, V_i^{\vee} \ot U'} \circ (\id_{X_{i-1}} \lhd g''_{i-1}).
      \]
      This is precisely the morphism described by $\pi(\Gamma) \cap [t_{i-1}, t_{i+1}]$ after the isotopy which drags the coupon $h$ across the line $y = t_i$ from lower to higher $y$-value. It follows that $\Ev_{\mathbf{C}}(\Gamma)$ is unchanged by the isotopy. Clearly, an analogous argument shows that $\Ev_{\mathbf{C}}(\Gamma)$ is unchanged by the isotopy which drags a coupon across $y = t_i$ in the opposite direction.

      Now consider the case where $i = 0$ in the above isotopy. Then $\Gamma \cap [t_{-1}, t_1]$ evaluates to the morphism     
      \begin{equation*}
        (\id_{X_{0}} \lhd g_{0}) \circ \mu^{-1}_{X_0, W_0, U_0} \circ (f_{0} \lhd \id_{U_{0}}) \circ \rho^{-1}_{X_{-1} \lhd V_0, U_0} \circ F(\mu_{X_{-1}, V_0, U_0} \circ (\id_{X_{-1}} \lhd h') \circ (\id_{X_{-1}} \lhd g''_{-1})).
      \end{equation*}
      Manipulations identical to the above $i \neq 0$ replace this with the morphism
      \begin{align*}
        &(\id_{X_{0}} \lhd g_{0}) \circ \mu^{-1}_{X_0, W_0, U_0} \circ (f_{0} \lhd \id_{U_{0}}) \circ \rho^{-1}_{X_{-1} \lhd V_0, U_0} \circ F((\id_{X_{-1} \lhd V_0} \lhd h'') \circ \mu_{X_{-1}, V_0, V_0^{\vee} \ot U'} \circ (\id_{X_{-1}} \lhd g''_{-1}))\\
        &=(\id_{X_{0}} \lhd g_{0}) \circ \mu^{-1}_{X_0, W_0, U_0} \circ (f_{0} \lhd \id_{U_{0}}) \circ \rho^{-1}_{X_{-1} \lhd V_0, U_0} \circ F(\id_{X_{-1} \lhd V_0} \lhd h'')\\
         &\hspace{3em}\circ F(\mu_{X_{-1}, V_0, V_0^{\vee} \ot U'} \circ (\id_{X_{-1}} \lhd g''_{-1}))\\
         &= (\id_{X_{0}} \lhd g_{0}) \circ \mu^{-1}_{X_0, W_0, U_0} \circ (f_{0} \lhd \id_{U_{0}}) \circ \rho^{-1}_{X_{-1} \lhd V_0, U_0} \circ F(\id_{X_{-1} \lhd V_0} \lhd h'')\\
        &\hspace{3em}\circ \rho_{X_{-1} \lhd V_0, V_0^{\vee} \ot U'} \circ \rho^{-1}_{X_{-1} \lhd V_0, V_0^{\vee} \ot U'} \circ F(\mu_{X_{-1}, V_0, V_0^{\vee} \ot U'} \circ (\id_{X_{-1}} \lhd g''_{-1}))
      \end{align*}
      where we have used functoriality of $F$ in the first equality, and inserted $\id_{F((X_{-1} \lhd V_0) \lhd (V_0^{\vee} \ot U'))} = \rho_{X_{-1} \lhd V_0, V_0^{\vee} \ot U'} \circ \rho^{-1}_{X_{-1} \lhd V_0, V_0^{\vee} \ot U'}$ in the second. By naturality of $\rho$ this is equal to
      \begin{align*}
        &(\id_{X_{0}} \lhd g_{0}) \circ \mu^{-1}_{X_0, W_0, U_0} \circ (f_{0} \lhd \id_{U_{0}}) \circ \rho^{-1}_{X_{-1} \lhd V_0, U_0} \circ \rho_{X_{-1} \lhd V_0, U_0} \circ (\id_{F(X_{-1} \lhd V_0)} \lhd h'')\\
         &\hspace{3em}\circ \rho^{-1}_{X_{-1} \lhd V_0, V_0^{\vee} \ot U'} \circ F(\mu_{X_{-1}, V_0, V_0^{\vee} \ot U'} \circ (\id_{X_{-1}} \lhd g''_{-1}))\\
        &=(\id_{X_{0}} \lhd g_{0}) \circ \mu^{-1}_{X_0, W_0, U_0} \circ (f_{0} \lhd \id_{U_{0}}) \circ (\id_{F(X_{-1} \lhd V_0)} \lhd h'')\\
        &\hspace{3em}\circ \rho^{-1}_{X_{-1} \lhd V_0, V_0^{\vee} \ot U'} \circ F(\mu_{X_{-1}, V_0, V_0^{\vee} \ot U'} \circ (\id_{X_{-1}} \lhd g''_{-1}))
      \end{align*}
      on eliminating the pair $\rho^{-1}_{X_{-1} \lhd V_0, U_0} \circ \rho_{X_{-1} \lhd V_0, U_0} = \id_{F(X_{-1} \lhd V_0) \lhd U_0}$. A further use of the interchange law and naturality of $\mu^{-1}$ as in the $i \neq 0$ case, the above can be rewritten as 
      \begin{align*}
        &(\id_{X_{0}} \lhd g_{0}) \circ (\id_{X_0} \lhd ( \id_{W_0} \ot h'')) \circ \mu^{-1}_{X_0, W_0, V_0^{\vee} \ot U'} \circ (f_{0} \lhd \id_{V_0^{\vee} \ot U'})\\
         &\hspace{8em}\circ \rho^{-1}_{X_{-1} \lhd V_0, V_0^{\vee} \ot U'} \circ F(\mu_{X_{-1}, V_0, V_0^{\vee} \ot U'} \circ (\id_{X_{-1}} \lhd g''_{-1})).
      \end{align*}
      The resulting morphism is the result of evaluating $\pi(\Gamma)$ after the isotopy of dragging the $h$-coupon across the $y = t_0$ line from negative to positive $y$-value, so that $\Ev_{\mathbf{C}}(\Gamma)$ is unchanged by this isotopy. Clearly an analogous argument shows that $\Ev_{\mathbf{C}}(\Gamma)$ is unchanged by dragging coupons from positive to negative $y$-value.

      Finally, consider an isotopy which changes the coordinates of the coupons $C_i$. Note that, for $\Gamma$ to remain in generic position after this isotopy, the order of the coupons cannot be changed by the isotopy. Therefore it is clear that this isotopy does not change  $\Ev_{\mathbf{C}}(\Gamma)$.

      Any isotopy of $\Gamma$ into another generic position is a composition of the above isotopies, so it follows that $\Ev_{\mathbf{C}}(\Gamma)$ is well-defined with respect to the choice of isotopy of $\Gamma$ into generic position.  
    \end{proof} 

    \begin{rmk}
      It is possible to relax the condition on the ordering of boundary vertices in Definition \ref{d-generic pos}. This results in an additional isotopy which needs to be checked in Lemma \ref{l-ev well-defd}: namely, the isotopy which moves two vertices in different incoming or outgoing line defects past one another. Checking this isotopy is tedious and unenlightening, and does not change the skein theory, since the canonical order we insist upon always exists and any ribbon graph can be isotoped into this position.
    \end{rmk}
    
    Now we will consider stratified graphs up to isotopy, and linear combinations $\sum_i  \lambda_i [\Gamma_i]$ of isotopy classes of defect-coloured stratified graphs.

    \begin{defn}
      \label{d-ord-skein-rels}
      We define an equivalence relation on linear combinations of isotopy classes of defect-coloured stratified graphs as follows.
      Suppose there exist sets of graphs $\{ \Gamma_i \}, \{ \Gamma'_j \}$ and a cube $\mathbf{C} \subseteq M^3$ which is valid and for which each $\Gamma_i$ and $\Gamma'_j$ have the same good intersection with $\mathbf{C}$ with the same colours, and $\Gamma_i \backslash \mathbf{C} = \Gamma'_i \backslash \mathbf{C}$. Then if
      \[
        \sum_i  \lambda_i \Ev_{\mathbf{C}}(\Gamma_i) = \sum_j  \mu_j \Ev_{\mathbf{C}}(\Gamma'_j)
      \] 
      we declare that
      \[
        \sum_i  \lambda_i \Gamma_i \sim \sum_j  \mu_j \Gamma'_j.
      \]
      We call these relations as $\mathbf{C}$ ranges over all valid cubes in $M$ the \emph{skein relations}.
    \end{defn}

    \begin{defn}
      \label{d-defect-skein-mod}
      Let $\mathbb{M}$ be a boundary-stratified $\mathfrak{D}^{\cC}$-marked 3-manifold. 
      Then we define the defect skein module $\Sk(\mathbb{M})$ to be the vector space spanned by coloured ribbon graphs in $\mathbb{M}$, modulo the skein relations.
    \end{defn}

    Let us remark on some specific instances of the skein relations,  
illustrated in Figure \ref{f-skein-rels}:
    \begin{itemize}
      \item When $\mathbf{C}$ does not meet the boundary of $M$, the skein relations implement relations in the $\Hom$-spaces of $\cC$ (Figure \ref{f-sk-rels-1}). This is the familiar setting of non-defect skein theory, where skein relations implement the graphical calculus of a ribbon category. We call these relations the \emph{interior skein relations} or \emph{$\cC$-skein relations}. Therefore, in the case $\mathbb{M}$ has no line defects, Definition \ref{d-defect-skein-mod} recovers the usual definition of skein modules with marked arcs in the boundary.
      \item When $\mathbf{C}$ does meet the boundary but does not meet a nontrivial point defect, then $\mathbf{C}$ meets a line defect labelled by a module category $\cM$. Then the skein relations for $\mathbf{C}$ implement the relations in the $\Hom$-spaces of $\cM$, or implement the graphical calculus for $\cM$ as a $\cC$-module category (Figure \ref{f-sk-rels-2}). We call these \emph{skein relations along a line defect} or \emph{$\cM$-skein relations}.
      \item  At a point defect, we can by a skein relation absorb an interior coupon into the point defect itself (Figure \ref{f-sk-rels-3}). Where the point defect is an identity defect, we can bring interior morphisms into line defects.
      
      There is moreover a relation which combines vertices along a line defect which have a parallel interior ribbon joining them. One of these vertices may be a nontrivial point defect (Figure \ref{f-sk-rels-4}).

      Iterating these relations, including in the case where the vertices support identity coupons, any interior component of a graph can be attached to a section of defect ribbon to which it is parallel, and this can even take place across point defects. We call this procedure \emph{zipping} and the relations used to perform it the \emph{zipping relations}.
      \item Consider a point defect with one incoming line defect marked $\cM$ and no outgoing line defect, and a ribbon graph $\Gamma$ having a coupon labelled $f_{-1} : X_{-2} \lhd V_{-1} \to  X_{-1} \lhd W_{-1}$ immediately before the point defect, and the  label $u_{F(X_{-1})} : F(X_{-1}) \to \One_{\cC} \ot F(X_{-1})$ at the point defect, for $u$ the inverse to the left unitor of $\cC$. Note that this matches our conventions for point defects with no outgoing line defects. In a valid cube near these vertices, the graph evaluates to
      \begin{align*}
        &u_{F(X_{-1})} \circ \rho^{-1}_{X_{-1}, W_{-1}} \circ F(f_{-1}) \circ \rho_{X_{-2}, V_{-1}}\\
          &\in \Hom_{\cC}(F(X_{-2}) \ot V_{-1}, \One_{\cC} \ot F(X_{-1}) \ot W_{-1}). 
      \end{align*}
      Now consider the graph $\Gamma'$ where the point defect is labelled by $u_{F(X_{-2})}$, and there is an interior vertex just after the point defect labelled by $ \rho^{-1}_{X_{-1}, W_{-1}} \circ F(f_{-1}) \circ \rho_{X_{-2}, V_{-1}}$. This graph evaluates to
      \begin{align*}
        &\rho^{-1}_{X_{-1}, W_{-1}} \circ F(f_{-1}) \circ \rho_{X_{-2}, V_{-1}}  \circ (u_{F(X_{-2})} \ot \id_{V_{-1}})\\ 
        &\in \Hom_{\cC}(F(X_{-2}) \ot V_{-1}, \One_{\cC} \ot F(X_{-1}) \ot W_{-1}). 
      \end{align*}
      Recalling that $u_{X \ot Y} = u_X \ot \id_Y$ \cite[Prop. 2.2.4]{EGNO15TensorCategories} and $u$ is a natural transformation, the above two morphisms are identical, so $\Gamma \sim \Gamma'$. We call these skein relations the \emph{subduction relations} (Figure \ref{f-sk-rels-5}).
    \end{itemize}

    The subduction relations are a key part of our main proof, and we therefore need a lemma showing that the label at any point defect with no outgoing line defects can always be replaced with a unitor.

    \begin{lemma}
      \label{l-unitors-everywhere}
      Let $P$ a point defect with no outgoing defect, marked by $F : \cM \to \cC$, and $\Gamma$ a ribbon graph which is labelled by $f : F(X \lhd V) \to \One \ot W$ at the point defect. Then $\Gamma$ is related, under the skein relations, to a graph $\Gamma'$ having $u_{F(X)}$ at the point defect. 
    \end{lemma}

    \begin{proof}
      Let $\Gamma'$ be the graph having $u_{F(X)}$ at the point defect and a coupon in the interior afterwards with the label
      \[
        u^{-1}_{W} \circ f \circ \rho_{X, V}
      \]
      as in 
the right-hand side of      
      Figure \ref{f-sk-rels-6}. This graph evaluates to the morphism
      \begin{align*}
        &(\id_{\One} \ot (u^{-1}_W \circ f \circ \rho_{X, V}))\circ (u_{F(X)} \ot \id_V)\\
        &=       
        (\id_{\One} \ot u^{-1}_W) \circ (\id_{\One} \ot (f \circ \rho_{X, V}))\circ (u_{F(X)} \ot \id_V)\\
        &= (u^{-1}_{\One \ot W}) \circ (\id_{\One} \ot (f \circ \rho_{X, V}))\circ (u_{F(X) \ot V})\\
        &=(f \circ \rho_{X, V}) \circ u^{-1}_{F(X) \ot V} \circ u_{F(X) \ot V}\\
        &=f \circ \rho_{X, V}.
      \end{align*}
      The first equality is standard for a monoidal category, and the second applies standard properties of the unit \cite[Propositions 2.2.2, 2.2.4]{EGNO15TensorCategories}. The third equality is naturality of the (inverse) unitor. The result is the evaluation of the graph $\Gamma$ 
(the left-hand side of Figure \ref{f-skein-rels} (f))      
      Therefore, $\Gamma \sim \Gamma'$ and $\Gamma'$ has the required form.
    \end{proof}

     \begin{figure}
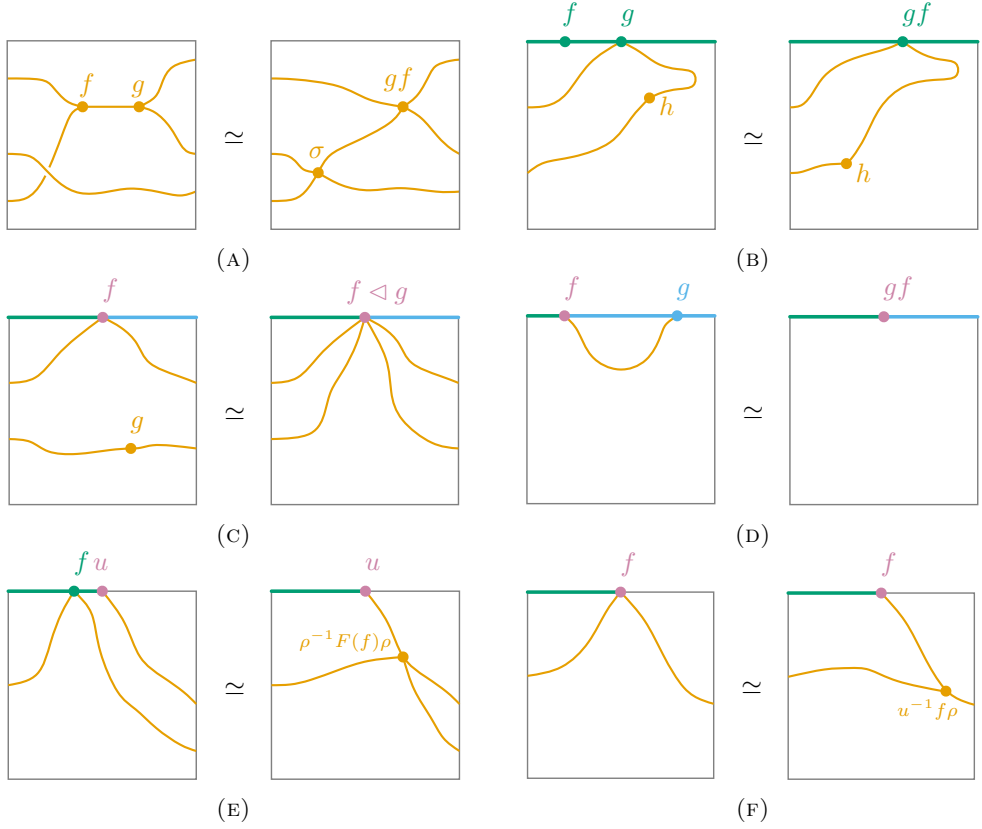

      \begin{subfigure}{0.45\textwidth}
        \centering
        \includesvg[width=6cm]{skein-rels-1.svg}
        \caption{}
        \label{f-sk-rels-1}
      \end{subfigure}
      \begin{subfigure}{0.45\textwidth}
        \centering
        \includesvg[width=6cm]{skein-rels-2.svg}
        \caption{}
        \label{f-sk-rels-2}
      \end{subfigure}
      \begin{subfigure}{0.45\textwidth}
        \centering
        \includesvg[width=6cm]{skein-rels-3.svg}
        \caption{}
        \label{f-sk-rels-3}
      \end{subfigure}
      \begin{subfigure}{0.45\textwidth}
        \centering
        \includesvg[width=6cm]{skein-rels-4.svg}
        \caption{}
        \label{f-sk-rels-4}
      \end{subfigure}
      \begin{subfigure}{0.45\textwidth}
        \centering
        \includesvg[width=6cm]{skein-rels-5.svg}
        \caption{}
        \label{f-sk-rels-5}
      \end{subfigure}
      \begin{subfigure}{0.45\textwidth}
        \centering
        \includesvg[width=6cm]{skein-rels-6.svg}
        \caption{}
        \label{f-sk-rels-6}
      \end{subfigure}
      \caption{Some specific examples of skein relations, shown in the projection $\pi$. (A) $\cC$-skein relations. (B) $\cM$-skein relations. (C) zipping relation of the first kind. (D) zipping relation of the second kind, for a vertex after a point defect (similarly: the case of a vertex before a  point defect, and the case of the identity point defect). (E) subduction relation. (F) relation to replace labels with unitors, when a point defect has no outgoing line defects.}
      \label{f-skein-rels}
    \end{figure}

    \subsection{Examples}

    \begin{eg}
      Write $H_1(T^2) = \Z \alpha \oplus \Z \beta$. Let $\mathbb{T}$ be the $\mathfrak{D}^{\cC}$-marked 3-manifold which has as an underlying manifold the cylinder $T^2 \times I$, with a single line defect along the homology cycle $\alpha$ in $T^2 \times \{ 1\}$, marked by a right $\cC$-module category $\cM$, and a single point defect in this line defect marked by a functor $F : \cM^{\op} \bt_{\cC} \cM \to \cC$. See Figure \ref{f-T2-example-1}.
    \end{eg}

    We will consider the defect skein module of this example. Recall that given a functor
    \[
      G : \cM^{\op} \times \cM \times \cC^{\op} \times \cC \to \Vect
    \]
    we can form the coend  
    $\int^{m \in \cM} G(m, m, -, -) : \cC^{\op} \times \cC \to \Vect$ and take another coend $\int^{c \in \cC} \int^{m \in \cM} G(m, m, c, c)$. We could also form the coend $\int^{m \in \cM} \int^{c \in \cC} G(m, m, c, c)$ and the coend $\int^{(m, c) \in \cM \times \cC} G(m, m, c, c)$ by identifying $\cM^{\op} \times \cM \times \cC^{\op} \times \cC \simeq (\cM \times \cC)^{\op} \times \cM \times \cC$. It is a theorem \cite[Thm.\,1.3.1]{Lor21CoendCalculus} that these three constructions are canonically equivalent.

    Denote by $m \boxtimes_{\cC} n$ the image of $(m,  n)$ under the canonical bilinear $\cC$-balanced functor $\cM^{\op} \times \cM \to \cM^{\op} \bt_{\cC} \cM$. Where $F$ is the functor labelling the unique point defect in $\mathbb{T}$, denote
    \begin{align*}
      \hat{F} : \cM \times \cM^{\op} \times \cC \times \cC^{\op} &\to \Vect\\
      (m, n, c, d) &\mapsto \Hom_{\cC}(F(m \boxtimes_{\cC} n) \otimes c^{\vee} \otimes d,  \One).
    \end{align*}

    Note that the change of where the $\op$ appears in the source of $F$ vs $\hat{F}$ is due to the use of the contravariant $\Hom$ functor to construct $\hat{F}$.

    \begin{figure}
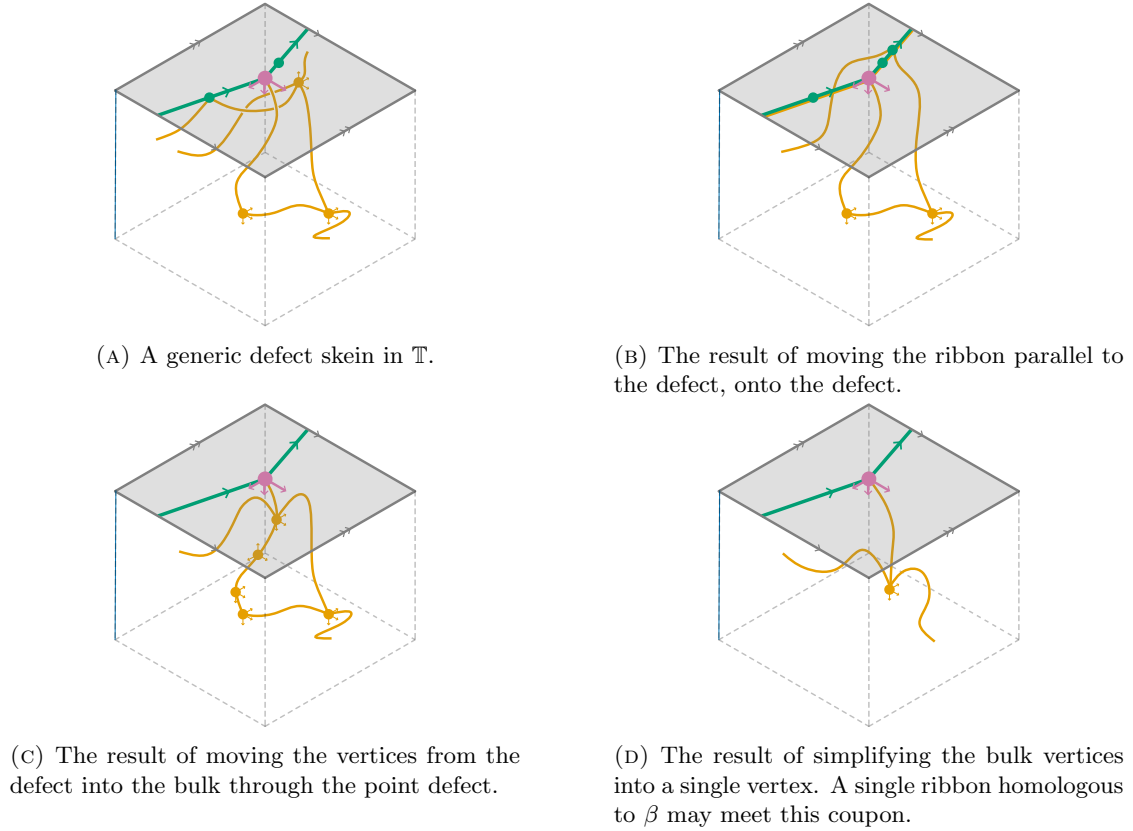

      \begin{subfigure}[t]{0.45\textwidth}
        \centering
        \includesvg[width=4cm]{T-2-skeins-1.svg}
        \subcaption{A generic defect skein in $\mathbb{T}$.}
        \label{f-T2-example-1}
      \end{subfigure}
      \hspace{1cm}
      \begin{subfigure}[t]{0.45\textwidth}
        \centering
        \includesvg[width=4cm]{T-2-skeins-2.svg}
        \subcaption{The result of moving the ribbon parallel to the defect, onto the defect.}
        \label{f-T2-example-2}
      \end{subfigure}
      \\
     \begin{subfigure}[t]{0.45\textwidth}
      \centering
        \includesvg[width=4cm]{T-2-skeins-3.svg}
        \subcaption{The result of moving the vertices from the defect into the bulk through the point defect.}
        \label{f-T2-example-3}
      \end{subfigure}
      \hspace{1cm}
      \begin{subfigure}[t]{0.45\textwidth}
        \centering
        \includesvg[width=4cm]{T-2-skeins-4.svg}
        \subcaption{The result of simplifying the bulk vertices into a single vertex. A single ribbon homologous to $\beta$ may meet this coupon.}
        \label{f-T2-example-4}
      \end{subfigure}
      \caption{The steps in the proof of Lemma \ref{l-skein-example-T}.}
      \label{f-skein-example-T}
    \end{figure}

     \begin{lemma}
      \label{l-skein-example-T}
      There is an  isomorphism
      \[
        \Sk(\mathbb{T}) \cong \int^{(m, c) \in \cM \times \cC} \hat{F}(m, m, c, c) .
      \]
    \end{lemma}

    \begin{proof}
      We will begin by providing a surjection $\int^{(m, c) \in \cM \times \cC} \hat{F}(m, m, c, c) \to \Sk(\mathbb{T})$. Given any skein in $\Sk(\mathbb{T})$, we can bring components which are parallel to $\alpha$ to the defect and apply the zipping relations. We can then combine all coupons along the line defect into a single coupon using the $\cM$-skein relations, and move these into the interior using a subduction relation. The result is that the ribbon  graph on the defect is labelled by a single object $x$ and the point defect by the unitor $u_{F(x \bt_{\cC} x)}$. We can bring all coupons in the interior towards the point defect, and by interior skein relations replace the graph in the neighbourhood of the point defect by one with a single coupon. Therefore, we can replace any ribbon graph with one which has an object $x \in \cM$ labelling the defect ribbon,  an object $y \in \cC$ labelling a ribbon in the interior homologous to $\beta$, and a single coupon $f : F(x \bt_{\cC} x) \ot y^{\vee} \ot y \to \One$ in the interior near the point defect. This gives a surjection
      \[
        \varphi : \bigoplus_{(m , c) \in \cM \times \cC} \hat{F}(m, m, c, c) \to \Sk(\mathbb{T}).
      \]
      The kernel of this surjection is the ambiguity that coupons from the line defect can be moved through the point  defect in two directions (with or against the orientation of $\alpha$), and coupons in the bulk may be isotoped around the $\beta$-cycle towards the point defect with or against the orientation of $\beta$.  The skein relations can always be performed after these isotopies, but since the source of $\varphi$ is a direct sum of $\Hom$-spaces, and the skein relations merely implement the relations in these $\Hom$-spaces, there are no further relations in $\ker \varphi$ after these isotopy ambiguities. So we see that the kernel of $\varphi$ is given by the double coend relations for  $\cM$ and $\cC$.
    \end{proof}

    We will now analyze the defect skein modules of the following examples, which can be obtained from $\mathbb{T}$ by handle additions.

    \begin{eg}
      \label{eg-cakes-via-handles}
      Let $\mathbb{K}_{\pm}$ be the $\mathfrak{D}^{\cC}$-marked solid tori of Example \ref{eg-marked-solid-tori}. Then $\mathbb{K}_{+}$ is obtained from $\mathbb{T}$ by adding a 2-handle along a curve in $T^2 \times \{0\}$ parallel to $\beta$, and then adding a 3-handle; $\mathbb{K}_{-}$ is obtained by adding a 2-handle along a curve in $T^2 \times \{0\}$ parallel to $\alpha$ and then adding a 3-handle.
    \end{eg}

    We will use the following fact to relate the skein modules of $\mathbb{T}$ and $\mathbb{K}_{\pm}$.

    \begin{lemma}
      \label{l-handle-slide}
      Let $ i : \mathbb{M} \hookrightarrow \mathbb{N}$ an embedding of $\mathfrak{D}^{\cC}$-marked 3-manifolds. Then $i$ induces a map $i_{*} : \Sk(\mathbb{M}) \to \Sk(\mathbb{N})$, and:
      \begin{itemize}
        \item If $\mathbb{N}$ is obtained from $\mathbb{M}$ by adding a 3-handle, then $i_{*}$ is an isomorphism.
        \item If $\mathbb{N}$ is obtained from $\mathbb{M}$ by adding a 2-handle, then $i_{*}$ is an epimorphism, and the kernel of $i_{*}$ is generated by the relation of sliding ribbon graphs in $\mathbb{M}$ through the 2-handle.
      \end{itemize}
    \end{lemma}  

    \begin{proof}
      This is a standard result in skein theory, see \cite[Prop.\,2.1, Lem.\,4.1]{PrzKauffmanBracketSkein1999}, and the proof carries over identically in the presence of boundary defects.
    \end{proof}

    We are now ready to understand $\Sk(\mathbb{K}_{+})$. Denote
      \[
        \tilde{F} = \Hom_{\cC}(F(- \boxtimes_{\cC} -), \One) : \cM \times \cM^{\op} \to \Vect.
      \]

    \begin{figure}
      \centering
      \includesvg[width=12cm]{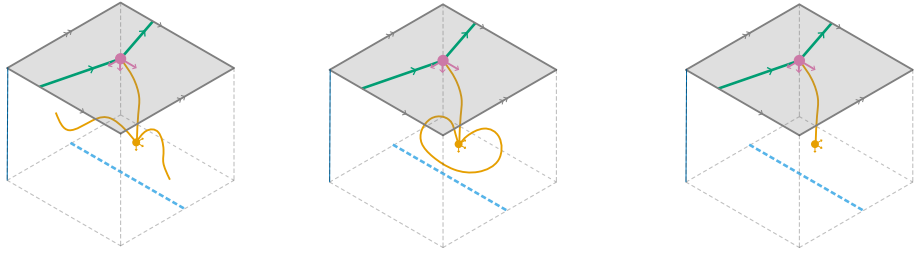}
      \caption{The handle-slide relations for a 2-handle attached along $\beta$ (the dashed blue curve). Left: a skein in the generic position described by Lemma \ref{l-skein-example-T}. Middle: the edge homotopic to $\beta$ has been slid through the handle, and is now in a neighbourhood of the coupon. Right: an interior skein relation replaces this with a new vertex, now with no edge homotopic to $\beta$.}
      \label{f-beta-handle}
    \end{figure}

    \begin{lemma}
      \label{l-solid-torus-example}
      There is an isomorphism
      \[
        \Sk(\mathbb{K}_{+}) \cong \int^{m \in \cM} \tilde{F}(m, m).
      \]
    \end{lemma}

    \begin{proof}
      By Lemma \ref{l-handle-slide}, we only need to understand the handle-slide relations for the 2-handle which is added in the construction of $\mathbb{K}_{+}$ from $\mathbb{T}$. Taking a representative graph in the form given in the proof of Lemma \ref{l-skein-example-T}, the handle-slide relations allow us to slide the $c$-labelled ribbon through the 2-handle and into the neighbourhood of the unique coupon near the point defect. The $c$-ribbon can then be absorbed into the coupon itself by an interior skein relation. See Figure \ref{f-beta-handle}. The handle-slide relations therefore identify the different $\cC$-factors in the double coend of Lemma \ref{l-skein-example-T}, so that after handle-sliding we can replace this expression with a single coend over $m \in \cM$. 
    \end{proof}

    \begin{figure}
      \centering
      \includesvg[width=10cm]{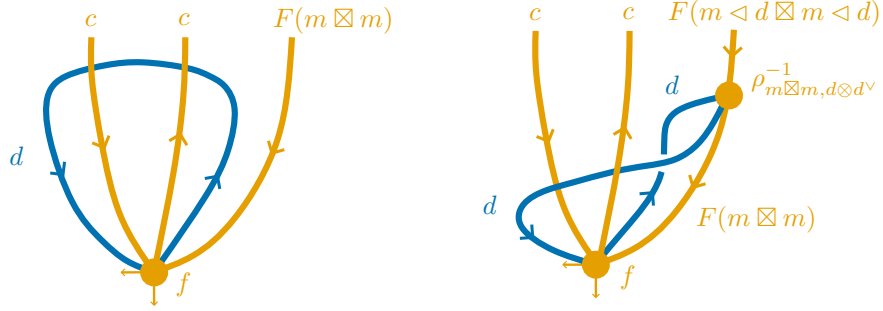}
      \caption{Left: the morphism $f_d \in \hat{F}(m, m, c, c)$ corresponding to $f \in \hat{F}(m, m, c \ot d, c \ot d)$. Right: the morphsim $f^{\rho}_d \in \hat{F}(m \lhd d, m \lhd d, c, c)$. The $x$ and $y$ framing vectors at the $f$-vertex in the plane of the picture are indicated.}
      \label{f-encircling-move}
    \end{figure}

    Given $f \in \hat{F}(m, m, c \ot d, c \ot d) = \Hom(F(m \bt_{\cC} m) \ot d^{\vee} \ot c^{\vee} \ot c \ot d, \One)$ and $d \in \cC$, denote by $f_d  \in \hat{F}(m, m, c, c)$ the morphism on the left of Figure \ref{f-encircling-move}. Denote also by $f^{\rho}_d \in \hat{F}(m \lhd d, m \lhd d, c, c)$ the morphism depicted on the right of Figure \ref{f-encircling-move}, where $\rho$ is the module coherence for $F$.

    \begin{figure}
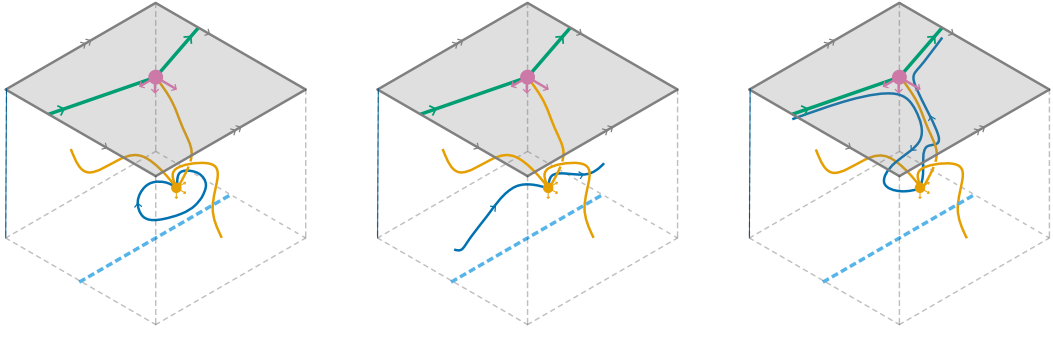
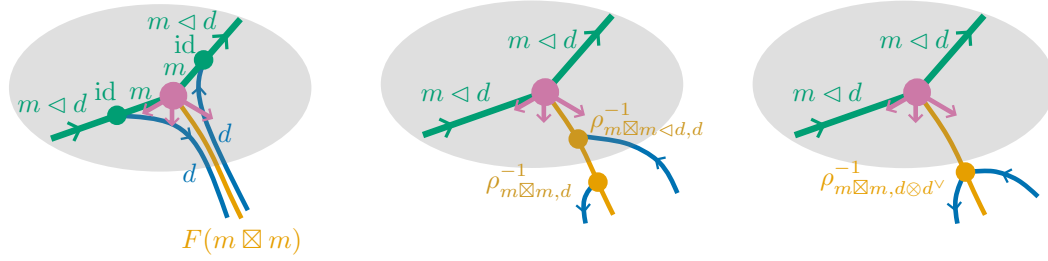

      \begin{subfigure}[t]{0.32\textwidth}
        \centering
        \includesvg[width=4cm]{alpha-handle-1.svg}
        \subcaption{We start with the coupon labelled by $f_d$, and isotope the $d$-ribbon close to the attaching curve.}
        \label{f-alpha-1}
      \end{subfigure}
      \begin{subfigure}[t]{0.32\textwidth}
        \centering
        \includesvg[width=4cm]{alpha-handle-2.svg}
        \subcaption{The result of sliding the ribbon through the curve.}
        \label{f-alpha-2}
      \end{subfigure}
      \begin{subfigure}[t]{0.32\textwidth}
        \centering
        \includesvg[width=4cm]{alpha-handle-3.svg}
        \subcaption{The $d$-ribbon is brought near the line defect, in preparation for zipping.}
        \label{f-alpha-3}
      \end{subfigure}
      \vspace{1cm}
      \\
      \begin{subfigure}[t]{0.32\textwidth}
        \centering
        \includesvg[width=4cm]{alpha-handle-4.svg}
        \subcaption{The situation at the point defect, after zipping. }
        \label{f-alpha-4}
      \end{subfigure}
      \begin{subfigure}[t]{0.32\textwidth}
        \centering
        \includesvg[width=4cm]{alpha-handle-5.svg}
        \subcaption{Subduction is applied to both of the identity vertices.}
        \label{f-alpha-5}
      \end{subfigure}
      \begin{subfigure}[t]{0.32\textwidth}
        \centering
        \includesvg[width=4cm]{alpha-handle-6.svg}
        \subcaption{The vertices marked $\rho^{-1}$ can be combined, since $\rho$ intertwines modulators (which are not drawn).}
        \label{f-alpha-6}
      \end{subfigure}
      \caption{The handle slide relations for a 2-handle attached along $\alpha$ (the dashed blue curve). The resulting configuration in the bulk is then the morphism $f^{\rho}_d$.}  
      \label{f-alpha-handle}
    \end{figure}

    \begin{lemma}
      \label{l-skein-example-K-minus}
      There is an  isomorphism
      \[
        \Sk(\mathbb{K}_{-}) \cong \left. \left(\int^{(m, c) \in \cM \times \cC} \hat{F}(m, m, c, c)\right) \middle/ \left(f_d \simeq f^{\rho}_d :  f \in \hat{F}(m, m, c \ot d, c \ot d), d \in \cC\right) \right. .
      \]
    \end{lemma}

    \begin{proof}
      By Lemma \ref{l-handle-slide} it suffices to understand the handle-slide relations on $\Sk(\mathbb{T})$ for the addition of a 2-handle along a curve homotopic to $\alpha$. Note that the $c$-ribbons cannot be slid through this handle. Given a coupon of the form $f_d$, we can slide the $d$-ribbon through the handle and bring it close to the defect. We can then apply zipping relations to attach the $d$-ribbon along most of the line defect. Near the point defect, we will have vertices labelled by $\id_{m \lhd d^{\vee}}$ and $\id_{m \lhd d}$. Applying subduction to these vertices results in a graph where the whole line defect is labelled by $m \lhd d$, the point defect by $u_{F(m \lhd d \bt_{\cC} m \lhd d)}$, and immediately after the point defect is a coupon labelled by $\rho^{-1}_{m \bt_{\cC} m, d \ot d^{\vee}}$. Bringing the latter coupon towards the coupon labelled $f$ results in the morphism $f^{\rho}_d$. Therefore, the handle-slide relations are precisely those in the statement of the Lemma. See Figure \ref{f-alpha-handle}.
    \end{proof}
    
    \section{Reshetikhin--Turaev with line and surface defects}
    \label{s-defect-RT}

    In this section we assume $\cC$ is a modular fusion category. We begin by recalling the ordinary Reshetikhin--Turaev TQFT, which has ribbon graphs (i.e. line defects) in its bordisms, and its relation to ordinary skein modules. We then sketch the category of bordisms supporting embedded surface defects, as appearing in \cite{CRS18LineSurfaceDefects}. We finally give an overview of the main construction of \cite{CRS18LineSurfaceDefects}, which is the definition of a TQFT from this bordism category based on the ordinary RT theory.   

    \subsection{Reshetikhin--Turaev with line defects}
    We recall that corresponding to any modular fusion category $\cC$ is an oriented TQFT
    \[
      Z^{\mathrm{RT}}: \widehat{\Bord^{\mathrm{rib}}_3}(\cC) \to \Vect
    \]
    where the source is an extension of the category of ribbon bordisms. The objects of $\Bord^{\mathrm{rib}}_3(\cC)$ are closed oriented surfaces $\Surface$ with a collection of $\cC$-marked point defects $\mathbf{D}$, with the points moreover carrying a sign. The morphisms are ribbon bordisms: 3-dimensional cobordisms supporting $\cC$-coloured ribbon graphs compatible with the marked points (i.e. systems of line defects which are connected by interior point defects, hence our terming this ``Reshetikhin--Turaev with line defects''). Compatibility means that ribbons should meet positively signed point defects locally oriented in the direction of bordism, and negatively oriented point defects in the reverse orientation, and  be labelled by the corresponding object.
    
    The hat refers to a particular extension where objects are moreover equipped with a lagrangian in $\mathrm{H}_1(\Surface,\mathbb{R})$, and morphisms equipped with integers which then glue using the Maslov index in the standard way \cite[Ch.\,IV.6]{Tur10QuantumInvariantsKnots}. We will suppress these data in our notation, as they do not play an important role.
    
    We do not recall in detail the construction of the TQFT $Z^{\RT}$ here (see e.g.\ \cite[Ch. IV]{Tur10QuantumInvariantsKnots}), but recall the following important feature. Consider $(\Surface_g, \mathbf{D}) \in \widehat{\Bord^{\mathrm{rib}}_3}(\cC)$, and fix a diffeomorphism $\delta : \Surface_g \xrightarrow{\sim} \partial H_g$ for $H_g$ a handlebody. Then $\delta$ induces a set of point defects in $\partial H_g$, which we also denote by $\mathbf{D}$ by abuse of notation. Then $(H_g, \mathbf{D})$ is a boundary-stratified $\mathfrak{D}^{\cC}$-marked 3-manifold, and we can consider the associated skein module $\Sk(H_g, \mathbf{D})$. Since ribbon graphs $\Gamma$ in $(H_g, \mathbf{D})$ define bordisms $(H_g, \Gamma) : \emptyset \to (\Surface_g, \mathbf{D})$, we have a linear map
    \begin{align*}
      \phi_{\delta} : \Sk(H_g, \mathbf{D}) &\to Z^{\RT}(\Surface_g, \mathbf{D})\\
      [\Gamma] &\mapsto Z^{\RT}(H_g, \Gamma)(1)
    \end{align*}
    This is well-defined since the skein relations occur in a ball, and $Z^{\RT}$ assigns $\Hom$-spaces to balls. 
    
    We may assume that $H_g$ is given with all the 1-handles attached to the boundary of the 0-handle in the eastern hemisphere, and all the defects $V_i$ lying along a great circle in the western hemisphere, giving them a total order. There is then a map 
    \[
        \eta : \Hom_{\cC}(\One, V_1 \ot \dots \ot V_n \ot \cL^{\ot g}) \to \Sk(H_g, \mathbf{D})
    \]
    taking a morphism $f$ to the graph $\Gamma_f$ shown in Figure \ref{f-handlebody}, where $\cL$ is the canonical coend for $\cC$ (see Notation \ref{n-Irr-coend}).

    \begin{figure}
      \centering
      \includesvg[width=8cm]{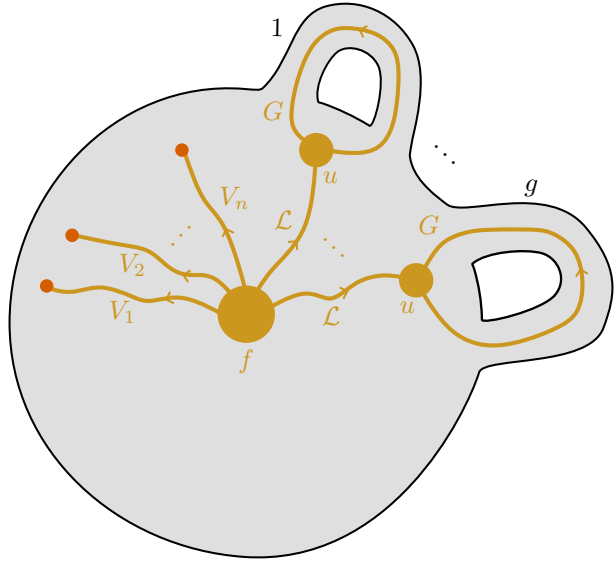}
      \caption{The map $\eta$ takes a morphism $f$ to the skein $[\Gamma_f]$, where $\Gamma_f$ is the graph depicted above. Here $\cL$ is the canonical coend for $\cC$ (see Notation \ref{n-Irr-coend}) and $G = \bigoplus_{i \in \Irr(\cC)} X_i$ is the canonical generator, and $u = \sum_{i \in \Irr(\cC)} \pi_{X_i}^{\vee} \ot \iota_{X_i}$, where $\iota_{X_i}, \pi_{X_i}$ present the simple object $X_i$ as a retract of $G$.}
      \label{f-handlebody}
    \end{figure}
    
    \begin{lemma}
      The map $\eta$ is a surjection.
    \end{lemma}

    \begin{proof}
      By semisimplicity of $\cC$, any skein is a linear combination of skeins where the ribbons are all labelled by simple objects of $\cC$. 
      Given any such skein, we first isotope all coupons into the 0-handle of $H_g$. Without loss of generality, each 1-handle then contains a unique ribbon parallel to its core labelled by some simple object $X$. We note that $X$ is a retract of the generator $G = \bigoplus_{i \in \Irr(\cC)} X_i$, and we insert the coupon $\id_X = \pi_X \iota_X$ into the $X$-ribbon and isotope the inclusion and projection to opposite ends of the 1-handle. The skein now has the label $G$ around this 1-handle, and in a neighbourhood of the attaching region in the 0-handle we have a coupon labelled by $\pi_X^{\vee} \ot \iota_X : X^{\vee} \ot X \to G^{\vee} \ot G$. This morphism factors uniquely through a morphism $u := \sum_{i \in \Irr(\cC)} \pi_{X_i}^{\vee} \ot \iota_{X_i} : \cL \to G^{\vee} \ot G$.      
      We then collect any other coupons within the 0-handle to a single coupon. The resulting skein is in the image of $\eta$.
    \end{proof}
       
    We can then consider the composite $\alpha_{\delta} = \phi_{\delta} \circ \eta$.

    \begin{lemma}
      The map 
      \[
        \alpha_{\delta} : \Hom_{\cC}(\One, \cL^{\ot g} \ot V_1 \ot \dots \ot V_n) \to Z^{\RT}(\Surface_g, \mathbf{D})
      \]
      is an isomorphism.
    \end{lemma}

    \begin{proof}
      This is essentially the definition of the Reshetikhin--Turaev theory, see \cite[\S IV.1.4]{Tur10QuantumInvariantsKnots}.
    \end{proof}

    \begin{cor}
      The maps $\eta$ and $\phi_{\delta}$ are isomorphisms.
    \end{cor}

    We note that descriptions of the Reshetikhin--Turaev state space via skein modules go back to \cite{BHMV95TopologicalQuantumField} in the case where $\Surface$ has no marked points. See \cite[Thm.\,4.2]{CGHP23Skein3+1TQFTsNonsemisimple} for a version where $\cC$ need not be fusion, and \cite[Eqn.\,(9.5)]{Tha21CategoryBoundaryValues} for a statement corresponding to our setting where $\cC$ is fusion and $\Surface$ may support marked points in the boundary.

    In fact, the Reshetikhin--Turaev state space can be described using any 3-manifold bounded by $\Surface$. Recall that, for $M, N$ two 3-manifolds with diffeomorphic boundary, there exists a framed link $L \subset M$ such that the manifold $M_L$ obtained by surgery along $L$ in $M$ produces $N$. The following is well-known, see \cite[Prop.\,5.2, Thm.\,5.8]{CGHP23Skein3+1TQFTsNonsemisimple} for a statement for closed manifolds, and \cite[Def.\,5.11, Thm.\,5.25]{Tha21CategoryBoundaryValues} for a statement corresponding to our setting.

    \begin{lemma}
      \label{l-skeins-2-h-iso}
      Let $(M, \mathbf{D})$ be a connected 3-manifold with $\cC$-marked points in the boundary, $L \subset M$ a framed link, and $(M_L, \mathbf{D})$ the result of surgery along $L$. Then there is an isomorphism
      \begin{align*}
       \zeta_L : \Sk(M, \mathbf{D}) &\xrightarrow{\sim} \Sk(M_L, \mathbf{D})\\
        [\Gamma] &\mapsto [\Gamma \cup \Omega]
      \end{align*}
      where $\Omega$ is the graph $\sum_{i \in \Irr(\cC)} d_i \omega_i$, with $\omega_i$ given by a ribbon parallel to a meridian of a tubular neighbourhood of $L$ coloured by the simple object $X_i$, $d_i = \mathrm{qdim}(X_i) = \mathrm{tr}(\id_{X_i})$.
    \end{lemma}

    Given any 3-manifold $M$ with a diffeomorphism $\delta : \Surface \xrightarrow{\sim} \partial M$, this also induces a diffeomorphism $\delta_L : \Surface \xrightarrow{\sim} \partial M_L$. Similarly to the case of handlebodies described above, we have maps $\phi_{\delta}, \phi_{\delta_L}$ from the respective skein modules to the Reshetikhin--Turaev state space. It follows from the definition of $Z^{\RT}$ on bordisms that
    \[
      \phi_{\delta} = \lambda_L \phi_{\delta_L} \circ \zeta_L
    \]
    where $\lambda_L$ 
    is built out of the gluing anomaly factor and the global dimension of $\cC$
    and can be read off from the link $L$. Defining $\widehat{\zeta_L} = \lambda_L \zeta$, we have the following.

    \begin{lemma}
      \label{l-ev-triangle}
      The triangle
      \[\begin{tikzcd}
        {\Sk(M, \mathbf{D})} && {\Sk(M_L, \mathbf{D})} \\
        \\
        & {Z^{\RT}(\Surface, \mathbf{D})}
        \arrow["{\widehat{\zeta_L}}", from=1-1, to=1-3]
        \arrow["{\phi_{\delta}}"', from=1-1, to=3-2]
        \arrow["{\phi_{\delta_L}}", from=1-3, to=3-2]
      \end{tikzcd}\]
      commutes.
    \end{lemma}

    \begin{cor}
      Let $M$ be a 3-manifold and $\delta : \Surface \xrightarrow{\sim} \partial M$ a diffeomorphism. Then the map $\phi_{\delta}$ is an isomorphism.
    \end{cor}

    \begin{proof}
      Any such $M$ can be presented as surgery along some link in a handlebody $H$ with boundary $\Surface$. Dually, there exists a link $L$ in $M$ such that $M_L$ is a handlebody. Then we apply Lemma \ref{l-ev-triangle}, and use that the evaluation map for a handlebody is an isomorphism.
    \end{proof}

    \begin{notn}
      \label{n-RT-sk-isos}
      We denote by $\psi_{\delta}$ the inverse to $\phi_{\delta}$. 
    \end{notn}

    We do not recall the construction of $Z^{\mathrm{RT}}$ on general bordisms here, but remark on its description for cylinder bordisms, which are important for our constructions, in terms of skein theory. Let $(N, \Gamma) : (\Surface, \mathbf{D}) \to (\Surface, \mathbf{D'})$ a ribbon bordism where $N$ is a cylinder over $\Surface$, and $M$ any manifold with boundary $\Surface$.
    The
    linear map
      \[
        \Sk(N, \Gamma) : \Sk(M, \mathbf{D}) \to \Sk(M, \mathbf{D'})
      \]
      is given for $[\Delta] \in \Sk(M, \mathbf{D})$ by attaching $\Gamma$ to the corresponding representative $\Delta$.
      
      \begin{lemma}
        \label{l-cylinder-surgery-commutes}
        Let $(\Surface, \mathbf{D}) \in \widehat{\Bord^{\mathrm{rib}}_3}(\cC)$, $M$ a 3-manifold, and $\psi : \Surface \xrightarrow{\sim} \partial M$. Let  $(N, \Gamma) : (\Surface, \mathbf{D}) \to (\Surface, \mathbf{D'})$ a cylindrical ribbon bordism. Then the diagram
        \[\begin{tikzcd}
          {\Sk(M, \mathbf{D})} && {Z^{\RT}(\Surface, \mathbf{D})} \\
          \\
          {\Sk(M, \mathbf{D'})} && {Z^{\RT}(\Surface, \mathbf{D'})}
          \arrow["{\phi_{\delta}}", from=1-1, to=1-3]
          \arrow["{\Sk(N, \Gamma)}"', from=1-1, to=3-1]
          \arrow["{Z^{\RT}(N, \Gamma)}", from=1-3, to=3-3]
          \arrow["{\phi_{\delta}}"', from=3-1, to=3-3]
        \end{tikzcd}\]
        commutes.
      \end{lemma}

      \begin{proof}
        The upper composite takes a skein $[\Delta] \in \Sk(M)$ to
        \[
          Z^{\RT}((\Surface, \mathbf{D'}) \xleftarrow{(N, \Gamma)} (\Surface, \mathbf{D}) \xleftarrow{(M, \Delta)} \emptyset)(1)
        \]
        which agrees with the lower composite.
      \end{proof}

      We will use Lemma \ref{l-cylinder-surgery-commutes} to express $Z^{\RT}$ on cylindrical ribbon bordisms using skein theory, as follows:
      \begin{equation}
        \label{eq-RT-bordisms-via-skeins}
          Z^{\RT}(N, \Gamma) = \phi_{\delta} \Sk(N, \Gamma) \psi_{\delta}.
      \end{equation}

      \begin{rmk}
        The diagram of Lemma \ref{l-cylinder-surgery-commutes} commutes up to the gluing anomaly for an arbitrary ribbon bordism. In the case of cylinder bordisms, two of the lagrangians involved are identical and the Maslov cocycle vanishes, so the gluing anomaly is $1$.
      \end{rmk}

      \begin{figure}[h]
        \centering
        \includesvg[width=5cm]{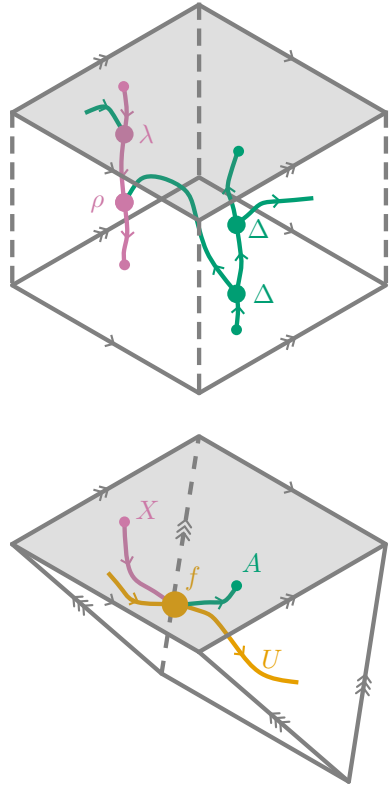}
        \caption{Below: a skein in the solid torus ending at marked points $X, A$, and with a vertex coloured by $f : X \ot U \to A \ot U$.  Above: a ribbon bordism which is topologically a cylinder over the torus. This may  be glued to the lower  part and then induces a linear map on skein modules, hence on Reshetikhin--Turaev state spaces. The vertices in the cylinder are coloured by the comultiplication $\Delta$ of $A$ and the left and right action maps $\lambda, \rho$ of $A$ on the $A-A$-bimodule $X$. Here framings of vertices are omitted and inputs/outputs can be inferred from the labels.}
        \label{f-solid-torus-RT}
      \end{figure}

      \begin{eg}
        \label{eg-solid-torus-RT}
        Let  $\Surface  = T^2$ a torus with two marked points labelled $\{(X, -), (A, +)\}$. Let $H$ be a solid torus and $\delta : \Surface  \xrightarrow{\sim} \partial H$. The bottom part of Figure \ref{f-solid-torus-RT} shows a representative of an element of $\Sk(H, \{X, A\})$, hence an element of $Z^{\mathrm{RT}}(\Surface, \{(X, -), (A, +)\})$ under $\phi_{\delta}$. The top part of the figure shows a cylinder $T^2 \times I$  with an embedded ribbon. Attaching this cylinder to $H$ defines a linear map  
        \[
          \Sk(T^2 \times I, \Gamma) : \Sk(H, \{X, A\}) \to \Sk(H, \{X, A\})
        \] 
        which induces the value of $Z^{\mathrm{RT}} $ on this ribbon bordism, via (\ref{eq-RT-bordisms-via-skeins}).
      \end{eg}

    \subsection{Defect bordisms}
    \label{ss-defect-bord}

    In \S \ref{s-defect-skein}, we described skein modules for 3-manifolds with a network of defects labelled by the defect system $\mathfrak{D}^{\cC}$. 
    In \cite{CRS18LineSurfaceDefects}, a defect bordism category $\widehat{\Bord^{\mathrm{def}}_3}(\overline{\mathfrak{D}^{\cC}})$ 
    is defined, where the defects are labelled by a related system $\overline{\mathfrak{D}^{\cC}}$. In the treatment of \cite{CRS18LineSurfaceDefects}, line defects in surfaces meet at point defects with 
no separation of these line defects into a set of incoming and a set of outgoing segments.
    To connect to the more general setup we established in \S \ref{s-defect-skein}, we regard the stratified surfaces of \cite{CRS18LineSurfaceDefects} as supporting defects with only incoming line defects at point defects.

    \begin{defn}
      Let $(S, \mathbf{L}, \mathbf{P})$ be a p-framed stratified surface. We say that $(S, \mathbf{L}, \mathbf{P})$ is \emph{in-stratified} if at every 0-stratum, all incident 1-strata are incoming in the sense of Definition \ref{d-in-out-surf}. We also say a boundary-stratified 3-manifold $M$ is \emph{in-stratified} if $\partial M$ is.
    \end{defn}
    
    The defect system $\overline{\mathfrak{D}^{\cC}}$ is adapted for marking in-stratified surfaces, so it only has a source map at level 1, unlike $\mathfrak{D}^{\cC}$ which 
    also has a target map. Moreover, the data of $\overline{\mathfrak{D}^{\cC}}$ is all assumed to be semisimple, making it suitable for Reshetikhin--Turaev constructions. In more detail, we have:
    \begin{itemize}
      \item $\overline{\mathfrak{D}^{\cC}_3} = \{ \cC \}$
      \item $\overline{\mathfrak{D}^{\cC}_2} = \{ \Delta\text{-separable symmetric Frobenius algebras in }\cC\}$
     \item $\overline{\mathfrak{D}^{\cC}_1} = \coprod_{N \in \N} \overline{\mathfrak{F}_N}$ where
      \begin{align*}
        \overline{\mathfrak{F}_N} = \{ ((A_i)_{i = 1}^N, X) \mid A_i \in \overline{\mathfrak{D}^{\cC}_2}, X \text{ a right } (A_1, \dots, A_N)\text{-multimodule}\}.
      \end{align*}
      (As before, we usually refer to data in $\overline{\mathfrak{D}^{\cC}_1}$ simply by the multimodule).
      \item The source map:
      \begin{align*}
        \overline{\mathfrak{s}} : \overline{\mathfrak{D}^{\cC}_1} &\to \coprod_{N \in \N} T_N(\overline{\mathfrak{D}^{\cC}_2})\\
        ((A_i)_{i = 1}^N, F) &\mapsto (A_i)_{i = 1}^N.
      \end{align*}
    \end{itemize}

    \begin{rmk}
        We can moreover consider point defects in the 3d-bordisms in $\widehat{\Bord^{\mathrm{def}}_3}(\overline{\mathfrak{D}^{\cC}})$. Note that these point defects will not induce defects in the objects of $\widehat{\Bord^{\mathrm{def}}_3}(\overline{\mathfrak{D}^{\cC}})$, since they do not have boundaries and are entirely contained in the interior of a bordism. In \cite{CRS19OrbifoldsNdimensionalDefect} it is explained how to deduce a canonical set of labels for 0-strata having already specified the labels for higher-dimensional strata. In the case at hand, the 0-strata are labelled by multimodule maps between labels of the incident 1-strata.
    \end{rmk}

    There is a map of defect systems $\Theta : \overline{\mathfrak{D}^{\cC}} \to \mathfrak{D}^{\cC}$, which means a level-wise map of sets which intertwines the source maps. The map at level 3 is trivial. At level 2 we have:
    \begin{align*}
      \Theta_2 : \overline{\mathfrak{D}^{\cC}_2} &\to \mathfrak{D}^{\cC}_2\\
      A &\mapsto A\Mod_{\cC}.
    \end{align*}
    This map is not surjective, as its image is given by just the finitely semisimple $\cC$-module categories.

    To describe the map at level 1, we introduce some notation. Writing $\cM_i = A_i\Mod_{\cC}$, it follows from standard properties of the relative Deligne--Kelly tensor product that 
      \begin{align}
        \cM_{1} \bt_{\cC} \dots \bt_{\cC} \cM_{N} &\xrightarrow{\sim} A_{1} \ot \dots \ot A_{N}\Mod_{\cC}\label{eq-DK-prod}\\
        M_1 \bt \dots M_N &\mapsto M_1 \ot \dots M_N.\nonumber
      \end{align}

    \begin{notn}
      \label{n-combined}
      Write $\cM_{\mathbf{N}}$ for the relative tensor product of module categories on the left-hand side of \eqref{eq-DK-prod}. Write $A_{\mathbf{N}}$ for the algebra on the right-hand side, and $M_{\mathbf{N}}$ for the $A_{\mathbf{N}}$-module appearing there. Where the individual module categories are indexed $(\cM_i)_{i \in \mathbf{j}}$ for $\mathbf{j}$ the vector of line defects incident to a specific point defect in the standard ordering (Definition \ref{d-standard-ordering-surf}), we adopt the notation $\cM_{\mathbf{j}}, A_{\mathbf{j}}, M_{\mathbf{j}}$. Where the indexing vector $\mathbf{j}$ is understood it will be suppressed. 
    \end{notn}
    We note that an $A_{\mathbf{N}}$-module is equivalent to a multimodule over the component algebras. At level 1 we have:
    \begin{align*}
      \Theta_1 : \overline{\mathfrak{D}^{\cC}_1} &\to \mathfrak{D}^{\cC}_1\\
      ((A_i)_{i = 1}^N, X) &\mapsto ((A_i\Mod_{\cC})_{i = 1}^N, \emptyset, X \ot_{A_{\mathbf{N}}} - : {A_{\mathbf{N}}}\Mod_{\cC} \to \cC).
    \end{align*} 

    Recall the Eilenberg--Watts equivalence (see e.g.\ \cite[\S 5.2]{BJS21DualizabilityBraidedTensor}): 
    \begin{align*}
      \Fun_{\cC}(A\Mod_{\cC}, \cC) &\simeq {}_{\cC}\text{Mod-}A\\
      F &\mapsto F(A)\\
      X \otimes_A - &\mapsfrom X,
    \end{align*}
    which is an equivalence of left $\cC$-module categories. So we see that $\Theta_1$ is an isomorphism on equivalence classes onto the part of $\mathfrak{D}^{\cC}_1$ where the source module categories are in the image of $\Theta_2$ (and the target is the trivial module category). Said differently, if we assume the module categories $\cM_i$ labelling line defects are of the form $\cM_i \simeq {A_i}\Mod_{\cC}$, then the module functors $F_j$ labelling point defects are equivalent to multimodules $X_j$.

    \begin{rmk}
      In the course of marking a (in-)stratified surface, we first mark the 1-strata by module categories, and then mark the 0-strata by compatible functors (Definition \ref{d-D^c-marking}). The latter step involves interpreting each incident module category as $\cM_i$ or $\cM_i^{\op}$ depending on a sign. In the case an incoming module category is to be interpreted, via the sign of 1-stratum (Definition \ref{d-tangent-vectors}), as $\cM_i^{\op}$ for $\cM_i = A_i\Mod_{\cC} \in \overline{\mathfrak{D}_2^{\cC}}$, we note that $M \mapsto M^{\vee}$ gives an equivalence
      \[     
        (A_i\Mod_{\cC})^{\op} \simeq A_i^{\op}\Mod_{\cC}
      \]
      which is possible because $\cC$ itself is rigid and pivotal (since it is ribbon). In particular, even accounting for signs, the incident module categories at a 0-stratum are each given by modules over an algebra. Then when using Notation \ref{n-combined} at a point defect $P_j$, we assume that $A_{\mathbf{j}}$ refers to the tensor product of the incident algebras $A_i^{-\epsilon_i}$ where $A_i^1 = A_i, A^{-1}_i = A_i^{\op}$.      
      Similarly $M_{\mathbf{j}}$ is the tensor product of $M_i^{-\epsilon_i}$ where $M_i^1 = M, M^{-1}_i = M^{\vee}$.
    \end{rmk}

    \begin{defn}
      Let $(\Surface, \mathbf{L}, \mathbf{P})$ be an in-stratified surface. If $\Surface$ is equipped with a $\mathfrak{D}^{\cC}$-marking where moreover all of the defect data is in the image of $\Theta$, then we say that $\SSigma = (S, \mathbf{D})$ is a $\overline{\mathfrak{D}^{\cC}}$-marked surface.
    \end{defn}

    The category $\widehat{\Bord^{\mathrm{def}}_3}(\overline{\mathfrak{D}^{\cC}})$ of \cite{CRS18LineSurfaceDefects} is an extension of a category $\Bord^{\mathrm{def}}_3(\overline{\mathfrak{D}^{\cC}})$ given by equipping objects with a choice of Lagrangian subspace in first homology, and morphisms with integers, to absorb the gluing anomaly of the Reshetikhin--Turaev theory. Let us sketch the objects and morphisms of $\Bord^{\mathrm{def}}_3(\overline{\mathfrak{D}^{\cC}})$ as introduced in \cite{CRS19OrbifoldsNdimensionalDefect}, noting that the state spaces for any choice of Lagrangian are all isomorphic, so the extension data is for our purposes unimportant.

    \begin{defn}[{Sketch}]
      The category $\Bord^{\mathrm{def}}_3(\overline{\mathfrak{D}^{\cC}})$ has 
      \begin{itemize}
        \item as objects: regular stratified surfaces, where 0-strata are equipped not with a framing but with a so-called \emph{$\ast$-decoration}, which is a choice of incident 2-stratum in a neighbourhood of each 0-stratum.     
        The strata must be decorated by data from the system $\overline{\mathfrak{D}^{\cC}}$, and the 0-strata are moreover equipped with a sign.
        \item as morphisms: bordisms which are stratified manifolds where every point has a neighbourhood diffeomorphic to either:
        \begin{itemize}
          \item an open cylinder on a stratified surface $S$, with the orientation of each $n$-stratum $X \times (0, 1)$ induced from the orientation of $X$ together with the orientation of $(0, 1)$, or
          \item a cone on a stratified sphere, constructed by taking a half-open cylinder $S \times (0, 1]$ and taking the quotient by the relation $(s, 1) \sim (t, 1)$, and adding a point defect at the cone point.
        \end{itemize}
      \end{itemize}
    \end{defn}

    It is clear that objects of $\Bord^{\mathrm{def}}_3(\overline{\mathfrak{D}^{\cC}})$ are in correspondence with $\overline{\mathfrak{D}^{\cC}}$-marked surfaces. Given an object of $\Bord^{\mathrm{def}}_3(\overline{\mathfrak{D}^{\cC}})$, we replace the $\ast$-decoration with a positive framing where the first vector points into the indicated 2-stratum, and isotope the line defects to be all incoming with respect to the framing; given a $\overline{\mathfrak{D}^{\cC}}$-marked, we replace the framing at point defects with a $\ast$-decoration in the 2-stratum into which the first framing vector points.

    \begin{rmk}
      \label{rk-ast-dec}
      In \cite{CRS18LineSurfaceDefects} the set $\overline{\mathfrak{D}^{\cC}_1}$ is in fact given by \emph{cyclic} multimodules. These are $(A_1, \dots, A_n)$-multimodules, where the sequence $(A_1, \dots, A_n)$ has $C_{n/k}$-symmetry for some $k \mid n$, and the multimodule $X$ is then equipped with $C_{n/k}$-equivariance data. In \cite{CRS18LineSurfaceDefects}, a further choice must be made at each point defect indicating a copy of $A_1$ which is considered the first. 
      Therefore \cite{CRS18LineSurfaceDefects} chooses a total order on the incident strata out of $n/k$ possibilities, which are the only possibilities once the (totally ordered) multimodule $X$ is chosen (by labelling with an $(A_1, \dots, A_n)$-multimodule and not, say, an $(A_2, \dots, A_n, A_1)$-multimodule, we can only begin at a stratum labelled $A_1$). In our approach, we are not interested in $C_{n/k}$-equivariance, and we choose a total order at each point defect from the outset via the framing vector at each point. Since both definitions amount to a specification of total order at point defects, we identify them without further comment.
    \end{rmk}

    \subsection{Reshetikhin--Turaev with surface defects}
    \label{s-defects-in-RT}
    In \cite{CRS18LineSurfaceDefects}, a defect TQFT
    \[
      Z^{\mathrm{def}} : \widehat{\Bord^{\mathrm{def}}_3}(\overline{\mathfrak{D}^{\cC}}) \to \Vect
    \]
    is defined based on $Z^{\mathrm{RT}}$. Let us sketch the definition of this TQFT.

    Given a defect bordism $\SSigma_1 \xrightarrow{\M} \SSigma_2$, we can choose an oriented trivalent graph $s$ called a \emph{skeleton}, for the surface defects in $\M$. Examples of skeleta are obtained as the Poincar\'{e} dual to a triangulation, with orientation induced by a total order on vertices of the triangulation, but skeleta can be more general than this. Just as triangulations can be related by Pachner moves, any two skeleta can be related by a finite sequence of combinatorial moves. See \cite[Appendix A.1]{CMR+21OrbifoldGraphTQFTs} for more details on skeleta.
  
    The skeleton $s$ has its edges decorated with the algebra $A_i$ labelling corresponding surface defect, its trivalent vertices labelled by the (co)multiplication of $A_i$, and vertices which intersect a line defect decorated by the module action. The result is a ribbon bordism
    \[ 
    M_{s} : (\Surface_1, \{(\tau_1^i, A_i)\} \cup \{X_j\}) \to  (\Surface_2, \{(\tau_2^k, A_k)\} \cup \{X_l\})
    \] 
    in $\widehat{\Bord^{\mathrm{rib}}_3}(\cC)$, where $\tau_1^i, \tau_2^k$ are endpoints of the triangulations of the line defects in $\Surface_1, \Surface_2$ induced by $s$, and the $X_j, X_l$ are the point defects in the respective surfaces. We will denote these sets of point defects by $\mathbf{D}^{\tau_1} = \{(\tau_1^i, A_i)\} \cup \{X_j\}$ and so on. For another such bordism $M_{s'}$ obtained using a different choice of skeleton, and having the same triangulation on the boundary, we note that $Z^{\mathrm{RT}}(M_{s}) = Z^{\mathrm{RT}}(M_{s'})$ due to the properties of the algebra $A_i$ being a $\Delta$-separable symmetric Frobenius algebra.

    Given an object $\SSigma \in \widehat{\Bord^{\mathrm{def}}_3}(\overline{\mathfrak{D}^{\cC}})$, we let $\M = \SSigma \times [0, 1]$. Given two sets of triangulation data $\tau_1 = \{ \tau_1^i \}, \tau_2 = \{ \tau_2^i \}$ for the line defects in $\SSigma$, there always exists a skeleton $s$ inducing these triangulations, and hence a well-defined linear map
    \[
    P_{\tau_1, \tau_2} := Z^{\mathrm{RT}}(M_s) :  Z^{\mathrm{RT}}(\Surface, \mathbf{D}^{\tau_1}) \to  Z^{\mathrm{RT}}(\Surface, \mathbf{D}^{\tau_2}). 
    \]
    In fact, these maps are compatible in the sense that
    \[
      P_{\tau_1, \tau_3} = P_{\tau_2, \tau_3} \circ P_{\tau_1, \tau_2}
    \]
    for triangulations $\tau_1, \tau_2, \tau_3$. We then define
    \[
      Z^{\mathrm{def}}(\SSigma) = \colim_{\tau} Z^{\mathrm{RT}}(\Surface, \mathbf{D}^{\tau}).
    \]
    More explicitly, the above colimit can be computed by choosing a triangulation $\tau$ of 1-strata, and then calculating the image of $P_{\tau, \tau}$, which is an idempotent map $Z^{\mathrm{RT}}(\Surface, \mathbf{D}^{\tau}) \to Z^{\mathrm{RT}}(\Surface, \mathbf{D}^{\tau})$ projecting onto the subspace given by the above colimit. We write $\pi_{\tau} : Z^{\mathrm{RT}}(\Surface, \mathbf{D}^{\tau}) \to Z^{\mathrm{def}}(\SSigma)$ for the canonical projection, and $\iota_{\tau}$ for the section of $\pi_{\tau}$.
    
    The action of $Z^{\mathrm{def}}$ on bordisms is straightforward. Given a bordism $\SSigma_1 \xrightarrow{\M} \SSigma_2$, we choose a skeleton $s$ of surface defects in $\M$ inducing triangulations $\tau_1, \tau_2$ as above. We then define
    \[
        Z^{\mathrm{def}}(\M) := Z^{\mathrm{def}}(\SSigma_1) \xrightarrow{\iota_{\tau_1}} Z^{\mathrm{RT}}(\Surface_1, \mathbf{D}^{\tau_1}) \xrightarrow{Z^{\mathrm{RT}}(M_s)} Z^{\mathrm{RT}}(\Surface_2, \mathbf{D}^{\tau_2}) \xrightarrow{\pi_{\tau_2}} Z^{\mathrm{def}}(\SSigma_2).
    \]

        \begin{figure}
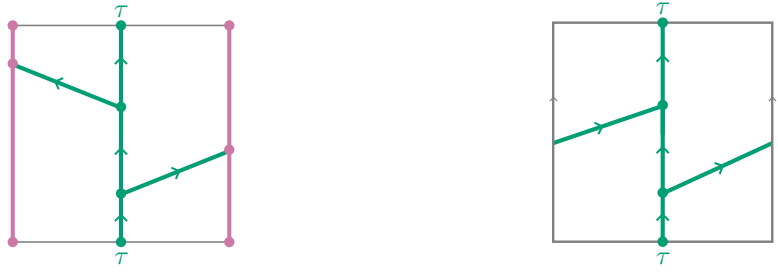

      \centering
      \begin{subfigure}{0.47\textwidth}
        \centering
       \includesvg[width=3cm]{skeleton.svg} 
      \end{subfigure}
      \begin{subfigure}{0.47\textwidth}
        \centering
        \includesvg[width=3cm]{skeleton-circle.svg}
      \end{subfigure}
      \caption{Left: A line defect which ends at point defects, when crossed with an interval, becomes a square. Here we show the point defects crossed with an interval as red lines, and denote the endpoint $\tau$ for the 2-segment triangulation of the line defect. In the interior is an example of a skeleton for the surface defect which is the line defect crossed with an interval. The case where the line defect ends at a single point defect can be obtained by identifying the red lines above. Right: a line defect which is a circle, when crossed with an interval, becomes a cylinder. Here we show a skeleton for the cylinder, ending at the endpoints of the 1-segment triangulation of the circle.}      
      \label{f-triangulation-example}
    \end{figure}

    \begin{eg}
      \label{eg-torus-projector}
      Consider the solid tori $\mathbb{K}_{\pm}$ of Example \ref{eg-marked-solid-tori}. Their boundaries are surfaces $K_{\pm}$ with defects. Assume that the label of the line defect is of the form $\cM = A\Mod_{\cC}$ and the label of the point defect is $X \otimes_{A^{\op} \ot A}  -$ for $X$ an $(A, A)$-bimodule. Then the line defect can be triangulated by $\tau$ made up of two line segments, meeting at a single point which we also denote $\tau$. The state space $Z^{\mathrm{RT}}(K_{\pm}, (\tau, A), X)$ was described in Example \ref{eg-solid-torus-RT}. The thickened line defect has a skeleton as shown in Figure \ref{f-triangulation-example}. The corresponding ribbon bordism is given in the top part of Figure \ref{f-solid-torus-RT}. So the linear map described in Example \ref{eg-solid-torus-RT} is the projector $P_{\tau}$. 
    \end{eg}

    Recall that the point defects in objects of $\widehat{\Bord^{\mathrm{def}}_3}(\overline{\mathfrak{D}^{\cC}})$ carry a sign, so that a point with positive sign becomes a line oriented in the direction of bordism composition on crossing with an interval, and oppositely oriented for a point with negative sign. Recall also that since $\cC$ is a ribbon category, it is pivotal. Write $p : (-) \implies (-)^{\vee \vee}$ for the pivotal structure.

    \begin{figure}
      \centering
      \includesvg[width=4cm]{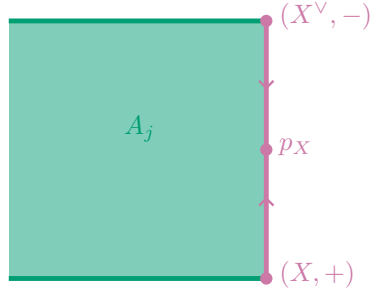}
      \caption{The bordism of Lemma \ref{l-X-flip}, read bottom to top.
Here, $p_X : X \to X^{\vee\vee}$ denotes the pivotal structure.      
      }
      \label{f-X-flip}
    \end{figure}

    \begin{lemma}
      \label{l-X-flip}
      Let $\SSigma^{\pm} \in \widehat{\Bord^{\mathrm{def}}_3}(\overline{\mathfrak{D}^{\cC}})$ contain, among possibly further defects, a positively/negatively oriented point defect labelled by a right $A$-module $X$, 
      for $A$ the algebra obtained from the incident line defects as in Notation \ref{n-combined}. 
      The bordism of Figure \ref{f-X-flip} furnishes an isomorphism 
      \[
        Z^{\mathrm{def}}(\SSigma^{+}) \xrightarrow{\simeq} Z^{\mathrm{def}}(\SSigma^{-})
      \]
      where $\SSigma^{-}$ is the same object but with the corresponding point defect now negatively oriented and labelled $X^{\vee}$.
    \end{lemma}

    \begin{proof}
       For the bordism of Figure~\ref{f-X-flip}  to make sense, we need that $X^{\vee}$ is a right $A$-module, and that $p_X$ is an $A$-module morphism. Writing $\rho$ for the action of $A$ on $X$, we see that $p_X \circ \rho \circ (p_X^{-1} \ot \id_A)$ defines an $A$-action on $X^{\vee}$ which makes $p_X$ an $A$-module morphism. The bordism of Figure \ref{f-X-flip}  is therefore a morphism $\SSigma^{+} \to \SSigma^{-}$ in $\widehat{\Bord^{\mathrm{def}}_3}(\overline{\mathfrak{D}^{\cC}})$. This is clearly an isomorphism under $Z^{\mathrm{def}}$, with the inverse given by a similar bordism containing a point defect labelled by $p_X^{-1}$. 
    \end{proof}

     In our computation of $Z^{\mathrm{def}}(\SSigma)$, a triangulation $\tau$ of the line defects is chosen, where again the triangulation points $\tau^i$ have a sign corresponding to whether incident ribbons in the cylinder bordism are oriented with or against the bordism direction at $\tau^i$. We have the following.

     \begin{figure}
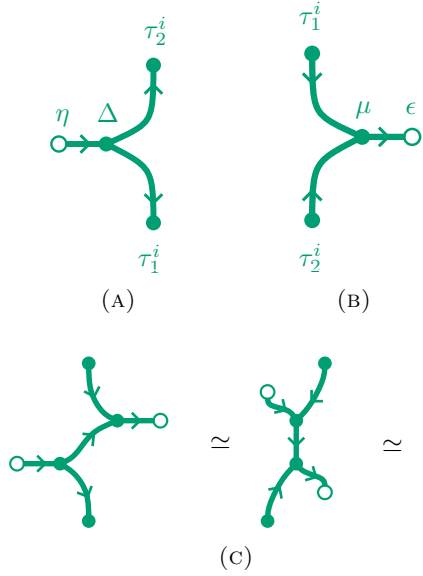

      \centering
      \begin{subfigure}{0.2\textwidth}
        \centering
        \includesvg[width=2cm]{A-flip-1.svg}
        \caption{}
        \label{f-A-flip-1}
      \end{subfigure}
      \begin{subfigure}{0.2\textwidth}
        \centering
        \includesvg[width=2cm]{A-flip-2.svg}
        \caption{}
        \label{f-A-flip-2}
      \end{subfigure}\\
      \vspace{0.5cm}
      \begin{subfigure}{\textwidth}
        \centering
        \includesvg[width=6cm]{A-flip-3.svg}
        \caption{}
        \label{f-A-flip-3}
      \end{subfigure}
      \caption{(A) The bordism of Lemma \ref{l-A-flip}. (B) The inverse bordism. (C) An illustration that the composition of these bordism in one order is the identity, using the Frobenius property of $A$. The composition in the other order is similar.}
      \label{f-A-flip}
    \end{figure}

    \begin{lemma}
      \label{l-A-flip}
      Let $\SSigma \in \widehat{\Bord^{\mathrm{def}}_3}(\overline{\mathfrak{D}^{\cC}})$ and let $\tau_{1}$ be a triangulation of the line defects of $\SSigma$ such that the triangulation point $\tau_1^i$ is negatively oriented and labelled by $A$. The bordism of Figure \ref{f-A-flip} furnishes an isomorphism 
      \[
        Z^{\mathrm{RT}}(\Surface, \mathbf{D}^{\tau_1}) \xrightarrow{\simeq} Z^{\mathrm{RT}}(\Surface, \mathbf{D}^{\tau_2})
      \]
      where $\tau_2$ is the same triangulation but with the corresponding triangulation point $\tau^i_2$ positively oriented.
    \end{lemma}

    \begin{proof}
      The inverse bordism is shown in Figure \ref{f-A-flip}. That this is the inverse easily follows because the labelling algebra $A$ is a Frobenius algebra. For example, that the two bordisms composed in one order gives the identity is shown ion Figure \ref{f-A-flip}, the other composition is similar.
    \end{proof}

    We will finish this section with the value of $Z^{\mathrm{def}}$ on certain stratified spheres, for which we need the following lemma. 

    \begin{figure}
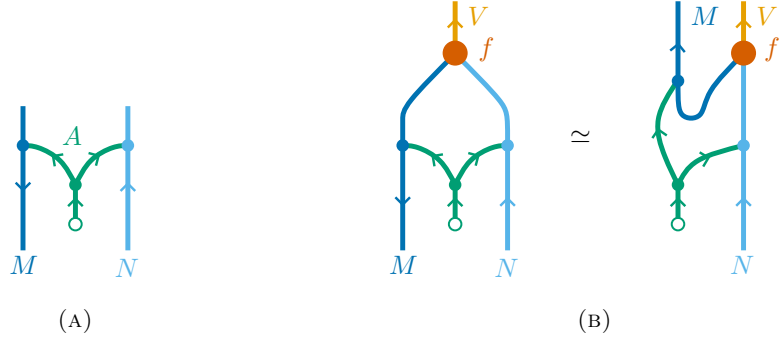

      \centering
      \begin{subfigure}{0.3\textwidth}
        \centering
        \includesvg[height=4cm]{rel-otimes-1.svg}
        \caption{}
        \label{f-rel-otimes-1}
      \end{subfigure}
      \begin{subfigure}{0.6\textwidth}
        \centering
        \includesvg[height=4cm]{rel-otimes-2.svg}
        \caption{}
        \label{f-rel-otimes-2}
      \end{subfigure}
      \caption{Graphical proof of lemma \ref{l-Hom-C-vs-A}. (A) This morphism in $\cC$ is the projection $M^{\vee} \ot N \to M^{\vee} \ot_A N$. (B) Left: morphisms from $M^{\vee} \ot_A N$ correspond to such morphisms, for $f$ any morphism in $\cC$ (note that if $f$ is already balanced with respect to the $A$-action, then the projector shown here can be removed). Right: morphisms $N \to M \ot V$ of $A$-modules correspond to such morphisms in $\cC$ (again, if $f$ is already an $A$-module map, the $A$-part of the diagram can be removed).}
      \label{f-rel-otimes}
    \end{figure}

    \begin{lemma}
      \label{l-Hom-C-vs-A}
      Let $A$ be a $\Delta$-separable symmetric Frobenius algebra in $\cC$, and $M, N \in A\Mod_{\cC}$. Then
      \[
        \Hom_{\cC}(M^{\vee} \ot_A N, V) \cong \Hom_A(N, M \ot V).
      \]
    \end{lemma}

    \begin{proof}
    Consider the map $M^{\vee} \ot N \to M^{\vee} \ot N$ of Figure \ref{f-rel-otimes} (a). It is easily checked that this is an idempotent morphism, and moreover that its image has the universal property of $M^{\vee} \ot_A N$. There is therefore a correspondence between $\Hom_A(M^{\vee} \ot_A N, V)$ and maps of the form shown on the left of Figure \ref{f-rel-otimes} (b), for $f \in \Hom_{\cC}(M^{\vee} \ot N, V)$. There is similarly a bijection between elements of $\Hom_A(N, M \ot V)$ and morphisms in $\cC$ of the form  shown on the right of Figure \ref{f-rel-otimes} (b). Therefore, the desired isomorphism is shown graphically in Figure \ref{f-rel-otimes} (b).
    \end{proof}

    \begin{eg}
      \label{eg-sphere}
      Consider the sphere with a point defect labelled $(N, -)$ at the north pole, and a point defect labelled $(M, -)$ at the south pole, connected by an $A$-labelled line defect oriented south to north, and suppose that there exists a disjoint point defect labelled by an object $(U, +)$ of $\cC$. Denote such a sphere by $S^{-}_{M, N; U}$. If the orientation of the $N$-defect is reversed, denote this sphere by $S^{+}_{M, N; U}$. 
    \end{eg}

    \begin{lemma}
      \label{l-strat-sphere-space}
      \begin{enumerate}
        \item \label{sphere-basic} There is an isomorphism
        \[
          Z^{\mathrm{def}}(S^{+}_{M, N; \One}) \cong \Hom_{A}(M, N).  
        \]
        \item \label{sphere-reversed} There is an isomorphism
        \[
          Z^{\mathrm{def}}(S^{-}_{M, N; U}) \cong \Hom_{\cC}(N \ot_{A} M, U).  
        \]
      \end{enumerate}
    \end{lemma}

    \begin{proof}
      Statement \ref{sphere-basic} is \cite[Lem.\,5.10]{CRS18LineSurfaceDefects}. For statement 2, applying Lemma \ref{l-X-flip} gives $Z^{\mathrm{def}}(S^{-}_{M, N; U}) \cong  Z^{\mathrm{def}}(S^{+}_{M, N^{\vee}; U})$. Then the same argument as in \cite[Lem.\,5.10]{CRS18LineSurfaceDefects} applies mutatis mutandis to show that $ Z^{\mathrm{def}}(S^{+}_{M, N^{\vee}; U}) \cong \Hom_A(M, N^{\vee} \ot U)$. Finally, Lemma \ref{l-Hom-C-vs-A} shows that $\Hom_A(M, N^{\vee} \ot U) \cong \Hom_{\cC}(N \ot_A M, U)$, which concludes the proof od statement  \ref{sphere-reversed}. 
    \end{proof}

    \section{The state-skein correspondence with defects}
    \label{s-main-thm}

    Let $\SSigma \in \widehat{\Bord^{\mathrm{def}}_3}(\overline{\mathfrak{D}^{\cC}})$, and $M$ a 3-manifold with a diffeomorphism $\delta : \Surface \xrightarrow{\sim} \partial M$. Then $\delta$ induces on $M$ the structure of a boundary-stratified $\overline{\mathfrak{D}^{\cC}}$-marked 3-manifold which we denote $\mathbb{M}$. In this section we define linear maps
    \[
      \Sk(\mathbb{M}) \leftrightarrows Z^{\mathrm{def}}(\SSigma)
    \]
    and show they are isomorphisms.

    \subsection{Definition of linear maps}

    \begin{conv}
      \label{conv-X-signs}
      We will assume all of the point defects in $\SSigma$ are negatively oriented. This simplifies the exposition 
      and
      the general case can be recovered using Lemma \ref{l-X-flip}.
    \end{conv}

    The map $\Phi : \Sk(\mathbb{M}) \to Z^{\mathrm{def}}(\SSigma)$ is defined as follows. 
    Consider the defect cylinder $\SSigma \times I$,
    where the line defects $\{L_i\}$ of $\SSigma$ become surface defects $\{L_i \times I\}$, and the point defects $\{P_j\}$ become line defects $\{ P_j \times I \}$ connecting surfaces,
    and by Convention \ref{conv-X-signs} these line defects are oriented against the interval direction (i.e. downward in a bottom-to-top bordism picture).
    
    Let $\Gamma$ be a ribbon graph in $\M$. We can attach $\SSigma \times I$ along $\SSigma \times \{ 0\}$ to $(\M, \Gamma)$. A surface defect $L_i \times I$ in $\SSigma \times I$ is labelled by a $\Delta$-separable symmetric Frobenius algebra $A_i$, and meets $\Gamma$ along a line labelled by left $A_i$-modules (and $A_i$-module maps). Such module-labelled lines (supporting module-map-labelled points) are appropriate ways to label the boundary of a surface defect in the defect system $\overline{\mathfrak{D}^{\cC}}$. A line defect $P_j \times I$ in $\SSigma \times I$ is labelled by a multimodule $X_j$, and meets $\Gamma$ at a point defect labelled by
    \[
      f_j : X_j \ot_{A_{\mathbf{j}}} N_{\mathbf{j}} \ot V_j \to \One_{\cC}  \ot  W_j
    \]
    in $\cC$, where we use Notation \ref{n-combined}. This is an appropriate label for a point defect between the line defects $\{N_k\}, X_j, V_j, W_j$ in the defect system $\overline{\mathfrak{D}^{\cC}}$. Therefore, the cylinder $\SSigma \times I$ attached to $(\M, \Gamma)$ along $\SSigma \times \{ 0 \}$ is a well-defined defect bordism 
    \[
      \Cylinder{\Gamma} : \emptyset \to \SSigma \in \mathrm{mor}\big( \widehat{\Bord^{\mathrm{def}}_3}(\overline{\mathfrak{D}^{\cC}})\big).
    \]
    See Figure \ref{f-phi-def}. We obtain an element of the state space $Z^{\mathrm{def}}(\SSigma)$ by setting 
    \begin{equation}
      \label{eq-Phi}
      \Phi(\Gamma) := Z^{\mathrm{def}}(\SSigma \xleftarrow{\Cylinder{\Gamma}} \emptyset)(1).
    \end{equation}

    \begin{prop}
      \label{p-Phi-well-def}
      Extending the assignment (\ref{eq-Phi}) to linear combinations of graphs gives a well-defined linear map \[
        \Phi : \Sk(\mathbb{M}) \to Z^{\mathrm{def}}(\SSigma).
      \]
    \end{prop}

    \begin{proof}
      Given linear combinations $\Gamma = \sum_i  \lambda_i \Gamma_i$ and $\Gamma' = \sum_j  \mu_j \Gamma'_j$ of isotopy classes of graphs, with $\Gamma \sim \Gamma'$, we want to check that $\Phi(\Gamma) = \Phi(\Gamma')$. It suffices to prove that $\Phi(\Gamma) = \Phi(\Gamma')$ whenever there exists a valid cube $\mathbf{C}$ such that $\sum_i \lambda_i \Ev_{\mathbf{C}}(\Gamma) = \sum_j \mu_j \Ev_{\mathbf{C}}(\Gamma')$.
      
      Suppose that $\mathbf{C}$ intersects a nontrivial point defect marked by $X \ot_A - : A\Mod_{\cC} \to \cC$. Then $\Gamma \cap \mathbf{C}$ and  $\Gamma' \cap \mathbf{C}$ describe identical morphisms in 
      \[
        \Hom_{\cC}(X \ot_A N \ot V, \One \ot U) \cong \Hom_{\cC}(X \ot_A N \ot V, U)
      \]
      where $N$ is the label of the defect ribbon at the $y = -1/2$ face of $\mathbf{C}$, and $V$ the product of the objects of $\cC$ at that face; $U$ is the product of the objects of $\cC$ at the $y = 1/2$ face; $A$ is the algebra labelling the line defect meeting $y = -1/2$, and $X$ is the bimodule labelling the point defect. Here we have applied the simplification of Notation \ref{n-combined}.
      
      \begin{figure}
        \centering
        \includesvg[width=10cm]{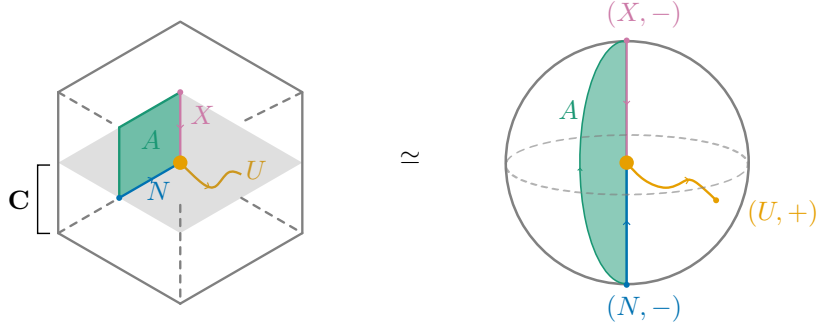}
        \caption{Left: In the proof of Proposition \ref{p-Phi-well-def}, a valid cube at a point defect (red) is enlarged to intersect a surface defect in $\SSigma \times I$. The original valid cube is indicated, with the $x=0$ face shaded. Right: this cube is homeomorphic to the defect ball shown, whose boundary is the sphere $S^{-}_{N, X; U}$ of Example \ref{eg-sphere}. 
        }
        \label{f-cube-collared}
      \end{figure}

      Note that by a zipping relation near the $y = - 1/2$ face, we can absorb $V$ into the line defect, so that up to re-labelling, the relation between $\Gamma$ and $\Gamma'$ is an identity in the space
      \[
        \Hom_{\cC}(X \ot_A N, U).
      \]
      Now consider a collar neighbourhood of $\partial M$ inside $\M \cup_{\SSigma \times \{ 0 \}} \SSigma \times I$ which intersects $\SSigma \times I$. Extending the cube $\mathbf{C}$ in the direction of this collaring, we obtain an embedded cube in $\M \cup_{\SSigma \times \{ 0 \}} \SSigma \times I$, see Figure \ref{f-cube-collared}. The boundary of this cube is homeomorphic to a sphere $S^{-}_{N, X; U}$ of the form of Example \ref{eg-sphere}. For all $i, j$, we have that $\Gamma_i$ and $\Gamma_j'$ define defect balls $B_{\Gamma_i}, B_{\Gamma_j'} : \emptyset \to S^{-}_{N, X; U}$. By Lemma \ref{l-strat-sphere-space}(\ref{sphere-reversed}), we have $Z^{\mathrm{def}}(S^{-}_{N, X; U}) = \Hom_{\cC}(X \ot_A N, U)$. It then follows that
      \[
        \sum_i \lambda_i Z^{\mathrm{def}}( S^{-}_{N, X; U} \xleftarrow{B_{\Gamma_i}} \emptyset) =  \sum_j \mu_j Z^{\mathrm{def}}( S^{-}_{N, X; U} \xleftarrow{B_{\Gamma_j'}} \emptyset)
      \]
      since $\Gamma$ and $\Gamma'$ describe identical morphisms in this $\Hom$-space.

      The bordisms $\Cylinder{\Gamma_i}$ and $\Cylinder{\Gamma_j'}$ agree outside of $S^{-}_{N, X; U}$ for all $i, j$. Since we can compute $Z^{\mathrm{def}}(\Cylinder{\Gamma_i})$ as the composite of  $Z^{\mathrm{def}}(S^{-}_{N, X; U} \xleftarrow{B_{\Gamma_i}} \emptyset)$ with $Z^{\mathrm{def}}$ evaluated on the complement of this ball, it follows from the above that
      \[
        \sum_i \lambda_i Z^{\mathrm{def}}(\SSigma \xleftarrow{\Cylinder{\Gamma_i}} \emptyset) = \sum_{j} \mu_j Z^{\mathrm{def}}(\SSigma \xleftarrow{\Cylinder{\Gamma'_i}} \emptyset)
      \]
      and hence $\Phi(\Gamma) = \Phi(\Gamma')$.

      In the case that $\mathbf{C}$ does not intersect a nontrivial point defect, then we can assume it intersects a point defect labelled $\Id_{A\Mod_{\cC}}$ (possibly with $A = \One_{\cC}$). A similar argument to the above shows that $\Phi(\Gamma) = \Phi(\Gamma')$, except that now the enlargement of $\mathbf{C}$ is the sphere $S^{+}_{N_1, N_2; \One}$. Therefore $\Phi$ is well-defined with respect to the skein relations holding in any valid cube.
    \end{proof}

    We now define the inverse to $\Phi$. 
    We will first define a map     
    \[
    \Psi :  Z^{\mathrm{RT}}(\Surface, \mathbf{D}^{\tau}) \to \Sk(\mathbb{M}) ,
    \]
whose domain is the state space where line defects have been replaced with their triangulation $\tau$, 
    using the notation of \S \ref{s-defects-in-RT}. There is a projector $P_{\tau}$ on $Z^{\RT}(\Surface, \mathbf{D}^{\tau})$ and maps 
    \[
      Z^{\RT}(\Surface, \mathbf{D}^{\tau}) \xtwoheadrightarrow{\pi_{\tau}} \Img P_{\tau} \cong Z^{\mathrm{def}}(\SSigma) \xhookrightarrow{\iota_{\tau}} Z^{\mathrm{RT}}(\Surface, \mathbf{D}^{\tau}). 
    \]
    The inverse to $\Phi$ will then be $\Psi \circ \iota_{\tau}$.

    \begin{conv}
      \label{conv-A-signs}
      By Lemma \ref{l-A-flip}, we can assume $\tau$ has only positively signed points. Let us fix such a $\tau$ with only one triangulation point in each line defect. Henceforth we suppress $\tau$ in the notation.
    \end{conv}
    
    We consider $\Sk(M, \mathbf{D})$, i.e. the skein module of the underlying manifold $M$ with skeins ending at the triangulation points (oriented outwards, as in Convention \ref{conv-A-signs}) or at the original point defects (oriented inwards, as in Convention \ref{conv-X-signs}). As in Notation \ref{n-RT-sk-isos}, we have an isomorphism $\psi_{\delta} : Z^{\mathrm{RT}}(\Surface, \mathbf{D}) \xrightarrow{\sim} \Sk(M, \mathbf{D})$. 

    We define $\Psi$ to take an element $v \in Z^{\mathrm{RT}}(\Surface, \mathbf{D})$ to the defect skein in $\mathbb{M}$ given by connecting any representative of $\psi_{\delta}(v) \in \Sk(M, \mathbf{D})$ at the ribbon labelled $A_i$ to the defect $L_i$ with an $A_i$-label, where the connection is by the multiplication map for $A_i$ at the triangulation point; and at the ribbon labelled $X_i$ to the point defect, where the point defect is given the label $u_{X_i}$ for $u$ the inverse of the left unitor of $\cC$. This is manifestly well-defined since two representatives for $\psi_{\delta}(v)$ can only differ by skein relations in the interior of $M$, and these are a subset of the skein relations in $\Sk(\M)$.

     \begin{figure}
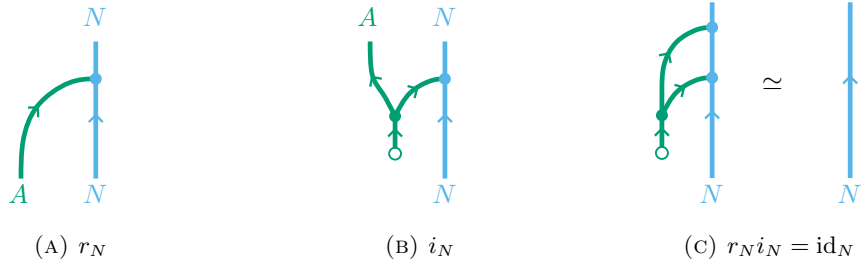

      \centering
      \begin{subfigure}{0.3\textwidth}
        \centering
        \includesvg[width=2cm]{separable-retract-1.svg}
        \caption{$r_N$}
        \label{f-separable-retract-1}
      \end{subfigure}
      \begin{subfigure}{0.3\textwidth}
        \centering
        \includesvg[width=2cm]{separable-retract-2.svg}
        \caption{$i_N$}
        \label{f-separable-retract-2}
      \end{subfigure}
      \begin{subfigure}{0.3\textwidth}
        \centering
        \includesvg[width=3.5cm]{separable-retract-3.svg}
        \caption{$r_N i_N = \id_N$}
        \label{f-separable-retract-3}
      \end{subfigure}
      \caption{The morphisms exhibiting any $A$-module as a retract of a free module. The empty circle denotes the unit of $A$.}
      \label{f-separable-retract}
    \end{figure}

    \subsection{Proof of the isomorphism}

     We divide the main theorem into several lemmas, beginning with the following simple observation.

    \begin{lemma}
      \label{l-separable-retract}
      For $A$ a separable Frobenius algebra,    
      then any module $N$ is a retract of a free module.
    \end{lemma}

    \begin{proof}
      Let  $r_N : A \ot N \to N$ be the action map. Let $i_N : N \to A  \ot N$ be given by the diagram of Figure \ref{f-separable-retract}. Then it is clear to see that $r_N i_N = \id_N$ using separability of $A$.
    \end{proof}

    \begin{lemma}
      \label{l-psi-surjective-0-simple}
      The map $\Psi$ is surjective.
    \end{lemma}

    \begin{proof}
      The image of $\Psi$ consists of defect skeins represented by a graph having any form in the interior of $M$, and having a unique coupon labelled by multiplication on any defect $L_i$ which is labelled with an algebra $A_i$, and where the point defects are labelled by the morphism $u_{X_i}$. If a line defect $L_i$ ends in point defects, we may moreover bring the attaching multiplication coupons to the point defects, and apply a subduction relation to bring these coupons into the interior.

      Consider any defect skein $[\Gamma]$. We will show that $\Gamma$ is skein equivalent to a graph of the form described above. Consider the part of $\Gamma$ restricted to any line defect $L$, which is labelled by an algebra $A$. Note $L$ is either a copy of $S^1$ or ends at point defects. In the first case, we can assume without loss of generality that $\Gamma \cap L$ has a single coupon $f$ in $L$, which may meet the interior part of $\Gamma$. Then let $N$ be the object labelling the only ribbon of $\Gamma \cap L$. Since $A$ is separable, we can always write $N$ as a retract of $A \ot N$, as in Lemma \ref{l-separable-retract}. Let $i_{N}$ be the inclusion and $r_{N}$ the retraction. We can insert an identity coupon near $f$, and rewrite $\id_{N} = r_{N} \circ i_{N}$. We isotope these coupons apart from each other and wrap one of them around the defect. Since $f$ is a morphism of $A$-modules, we can perform the manipulations of Figure \ref{f-psi-surjective-no-specials} to change the ribbon graph into one which is in the image of $\Psi$ at the defect $L$.

      Now we consider the case when $L$ ends in point defects. By Lemma \ref{l-unitors-everywhere}, we can assume these are always labelled by unitors. If $\Gamma \cap L$ has coupons, we can isotope them to a point defect, and apply a subduction relation  to bring them into the interior. So without loss of generality $\Gamma$ can be assumed to have a single object $N$ labelling $L$ and no coupons on $L$. As before, we insert $\id_{N} = r_{N} \circ i_{N}$ in the $N$-ribbon. We isotope the $i_{N}$ and $r_{N}$ coupons away from each other and towards the point defects at the ends of $L$. Then at each point defect we apply a subduction relation to bring $i_N$ and $r_N$ into the interior. Performing this procedure simultaneously in each 1-stratum which ends at point defects, we will have replaced $\Gamma$ with a graph where every 1-stratum is labelled by the corresponding algebra and has no coupons, and the outgoing ribbons at point defects are labelled with the corresponding multimodules. The result is then in the image of $\Psi$. See Figure \ref{f-psi-surjective}.
    \end{proof}

    \begin{figure}
      \begin{subfigure}{\textwidth}
        \centering
        \includesvg[width=0.75\textwidth]{psi-surjective-no-specials.svg}
        \caption{ Relations which hold in $A\Mod_{\cC}$, shown in the graphical calculus within $\cC$. This uses that $f$ is a morphism of left $A$-modules Here, $N$ is an $A$-module and $f$ may be a morphism $N \lhd V \to N \lhd W$, with the objects of $\cC$ denoted by the lower inputs/outputs to $f$. We denote the rightmost morphism by $\tilde{f}$.}
        \label{f-psi-surjective-no-specials}
      \end{subfigure}\\
      \hspace{0.5cm}
      \begin{subfigure}{\textwidth}
        \centering
        \includesvg[width=0.8\textwidth]{circle-rels.svg}
        \caption{This sequence of skein relations brings the ribbon graph into the image of $\Psi$. Firstly the relation above is applied. Secondly we apply (un)zipping relations. Thirdly we expand the definition of $\tilde{f}$. Fourthly: we complete the unzipping process (isotope the identity coupons away from each other and combine them again in another place along the line defect).}
        \label{f-circle-rels}
      \end{subfigure}  
      \caption{}
      \label{f-psi-surjective-no-specials}
    \end{figure}

    \begin{figure}
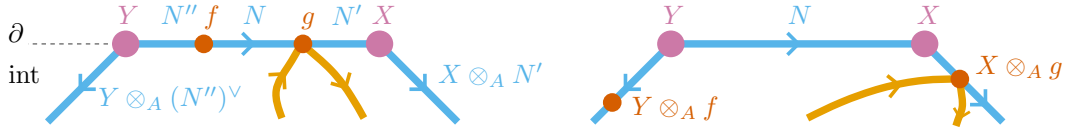
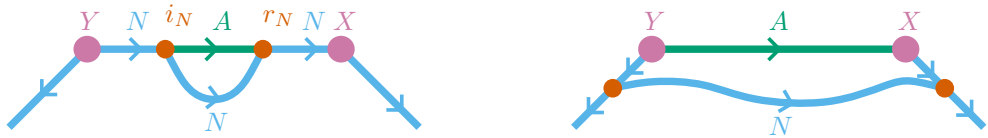

      \centering
      \begin{subfigure}[t]{0.49\textwidth}
        \centering
        \includesvg[width=6cm]{psi-surjective-1.svg}
        \caption{ A graph meeting a line defect ending at the pink point defects. The defects are marked $X$ and $Y$, and we assume the graph is coloured by unitors at these points.}
        \label{f-psi-surjective-1}
      \end{subfigure}
      \begin{subfigure}[t]{0.49\textwidth}
        \centering
        \includesvg[width=6cm]{psi-surjective-2.svg}
        \caption{All vertices have been moved into the interior by subduction.}
        \label{f-psi-surjective-2}
      \end{subfigure}\\
      \vspace{0.5cm}
      \begin{subfigure}[t]{0.49\textwidth}
        \centering
        \includesvg[width=6cm]{psi-surjective-3.svg}
        \caption{A copy of $\id_N = r_N i_N$ is inserted in the line defect.}
        \label{f-psi-surjective-3}
      \end{subfigure}
      \begin{subfigure}[t]{0.49\textwidth}
        \centering
        \includesvg[width=6cm]{psi-surjective-4.svg}
        \caption{The $i_N$ and $r_N$ coupons have been subducted into the bulk, leaving only an $A$-label along the line defect.}
        \label{f-psi-surjective-4}
      \end{subfigure}
      \caption{The manipulations used at 1-strata ending in point defects to bring any graph into the form of one in the image of $\Psi$. }
      \label{f-psi-surjective}
    \end{figure}

    \begin{lemma}
      \label{l-phi-surjective-0-simple}
      The map $\Phi$ is surjective.
    \end{lemma}

    \begin{proof}
      We observe that the diagram below is commutative:
      \[\begin{tikzcd}
        && {Z^{RT}(\Surface, \mathbf{D})} \\
        {\Sk_{\cC}(\mathbb{M})} \\
        && {Z^{\mathrm{def}}(\Surface, \mathbf{D})}
        \arrow["\Psi"', from=1-3, to=2-1]
        \arrow["\pi", from=1-3, to=3-3]
        \arrow["\Phi"', from=2-1, to=3-3]
      \end{tikzcd}\]
      This follows from the definition of $\Phi, \Psi$ and the construction of $Z^{\mathrm{def}}(\SSigma)$. 
Indeed,      
      the composite maps a vector $v \in Z^{\RT}(\Surface, \mathbf{D})$ to
      \begin{align*}
        &Z^{\mathrm{def}}(\SSigma \xleftarrow{\Cylinder{\Psi(v)}} \emptyset)(1)\\
        &= \pi Z^{\RT}((\Surface, \mathbf{D}) \xleftarrow{(\Surface \times I, \Delta)} (\Surface, \mathbf{D}) \xleftarrow{(M, \psi_{\delta}(v))} \emptyset)(1)\\
        &= \pi Z^{\RT}(\Surface \times I, \Delta)(Z^{\RT}(M, \psi_{\delta}(v))(1))\\
        &= \pi Z^{\RT}(\Surface \times I, \Delta)(\phi_{\delta}\psi_{\delta}(v))\\
        &= \pi Z^{\RT}(\Surface \times I, \Delta)(v)\\
        &= \pi(v).
      \end{align*}
       The first equality uses that $Z^{\mathrm{def}}$ can be computed using $Z^{\RT}$.  
       Here $(\Surface \times I, \Delta)$ denotes the ribbon bordism obtained by replacing the surface defects in $\SSigma \times I$ by their skeleta, with endpoint defects $\mathbf{D}$ at either end of the cylinder. The second equality is notational, the third equality is the definition of $\phi_{\delta}$, and the fourth equality uses that $\phi_{\delta}$ and $\psi_{\delta}$ are inverse to each other. The fifth equality uses that $\pi$ is the universal map to a colimit, and $Z^{\RT}(\Surface \times I, \Delta)$ is one of the arrows in the diagram defining this colimit.
      
      Now let $\iota : Z^{\mathrm{def}}(\SSigma) \to Z^{\mathrm{RT}}(\Surface, \mathbf{D})$ be the inclusion, it follows that $\Phi \Psi \iota = \pi \iota = \id_{Z^{\mathrm{def}}(\SSigma)}$. So $\Phi$ has a section and is therefore surjective.
    \end{proof}

    \begin{lemma}
      \label{l-psi-res-surjective-0-simple}
      The map $\Psi \circ \iota$ is surjective. 
    \end{lemma}

    \begin{proof}
      We already showed that $\Psi$ is surjective in Lemma \ref{l-psi-surjective-0-simple}. We will show that the diagram below commutes:
      \[\begin{tikzcd}
        && {Z^{RT}(\Surface, \mathbf{D})} \\
        {\Sk_{\cC}(\mathbb{M})} \\
        && {Z^{RT}(\Surface, \mathbf{D})}
        \arrow["\Psi"', from=1-3, to=2-1]
        \arrow["P", from=1-3, to=3-3]
        \arrow["\Psi", from=3-3, to=2-1]
      \end{tikzcd}\]
      Then the claim follows since by Lemma \ref{l-psi-surjective-0-simple}, we know that for any skein $y$, we there exists $x$ with $\Psi(x) = y$. But by commutativity of the above diagram we have $\Psi P(x) = y$, but $P = \iota \circ \pi$, so $(\Psi \circ \iota)(\pi(x)) = y$ and $\Psi \circ \iota$ is surjective.   

      To see that the above diagram commutes, we consider any element $\Psi P(x)$. We will assume the skeleta of Figure \ref{f-triangulation-example} are chosen to define $P$. Then along any line defect which is a copy of $S^1$, the sequence of skein relations shown in Figure \ref{f-psi-p-is-psi-1} shows that $\Psi P(x) = \Psi(x)$. Along any line defect which ends in point defects (where for the sake of exposition that this is the only incident line defect at its endpoints - the general case follows), the sequence of moves shown in Figure \ref{f-psi-p-is-psi-2} transforms $\Psi P(x)$ into $\Psi(x)$. This uses that $X \ot_{A} -$ sends the multiplication of $A$ to the action of $A$ on $X$. 
    \end{proof}

    \begin{figure}
      \centering
      \begin{subfigure}{\textwidth}
        \centering
        \includesvg[width=12cm]{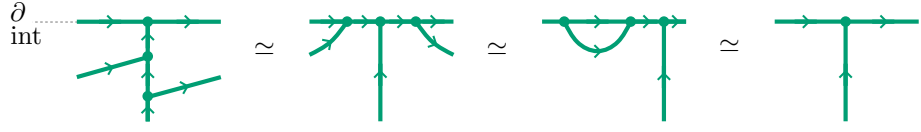}
        \caption{The case of a skein $x$ ending at a circular line defect. All lines here are labelled by a $\Delta$-separable symmetric Frobenius algebra $A$ which marked the line defect. Leftmost shows $\Psi P(x)$. The first equivalence uses skein relations along a line defect to move vertices onto the defect, using that $A$ is a Frobenius algebra. The second equivalence simply rotates the figure and uses periodicity. The third equivalence uses that $A$ is $\Delta$-separable. Rightmost then shows $\Psi(x)$.}
        \label{f-psi-p-is-psi-1}
      \end{subfigure}\\
      \vspace{0.5cm}
      \begin{subfigure}{\textwidth}
        \centering
        \includesvg[width=12cm]{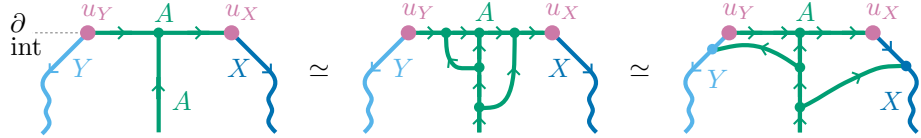}
        \caption{The case of a skein $x$ ending at a line defect which ends in point defects. Left: $\Psi(x)$. Middle: this is equivalent under the skein relations along the line defect due to the properties of $A$ being a $\Delta$-separable symmetric Frobenius algebra. Right: The result of subducting the additional multiplication coupons into the interior. This is a depiction of $\Psi P(x)$.}
        \label{f-psi-p-is-psi-2}
      \end{subfigure}
      \caption{The manipulations used to show $\Psi P =\Psi$.} 
      \label{f-psi-p-is-psi}
    \end{figure}

    \begin{thm}
      \label{t-main-thm-in-text}
      The maps $\Phi$ and $\Psi \circ \iota$ are mutually inverse isomorphisms.     
    \end{thm}

    \begin{proof}
First note that $\Sk(\mathbb{M})$ is finite-dimensional, since $\Psi$ is surjective and that $Z^{\mathrm{def}}(\SSigma)$ is finite-dimensional by construction.
      Since the spaces $Z^{\mathrm{def}}(\SSigma), \Sk(\mathbb{M})$ are finite-dimensional,       it suffices to show that the maps are surjective. We have shown this in Lemmas \ref{l-phi-surjective-0-simple} \ref{l-psi-res-surjective-0-simple}. 
    \end{proof}

    We close by illustrating an application of this theorem with our running example.

    \begin{eg}
      Consider the $\mathfrak{D}^{\cC}$-marked 3-manifold $\mathbb{K}_{+}$ of Example \ref{eg-cakes-via-handles}, with boundary $\mathbb{T}$. Suppose that the unique line defect is marked by the algebra $A$, and the unique point defect by $F = X_{A^{\op} \ot A}$. Then $\mathbb{K}_{+}$ is $\overline{\mathfrak{D}^{\cC}}$-marked. In the notation of Lemma \ref{l-solid-torus-example},    
      we have that for any $M \in A\Mod_{\cC}$,
      \[
        \tilde{F}(M, M) = \Hom_{\cC}(X_{A^{\op} \ot A}(M \ot M), \One).
      \]
      Rewriting $X_{A^{\op} \ot A}(M \ot M) \cong M^{\vee} \ot_{A} X \ot_A M$, an application of Lemma \ref{l-Hom-C-vs-A} gives
      \[
        \Hom_{\cC}(M^{\vee} \ot_{A} X \ot_A M, \One) \cong \Hom_A(X \ot_A M, M).
      \]
      Then by Lemma \ref{l-solid-torus-example}, we have that
      \[
        \Sk(\mathbb{K}_{+}) \cong \int^{M \in A\Mod_{\cC}} \tilde{F}(M, M) \cong \bigoplus_{M \in \Irr(A\Mod_{\cC})} \Hom_A(X \ot_A M, M)
      \]
      where we use semisimplicity of $A\Mod_{\cC}$.

      Now consider the manifold $\mathbb{K}_{-}$ of Example \ref{eg-cakes-via-handles}. By a similar argument to the above, we have 
      \[
        \hat{F}(M, M, U, U) \cong \Hom_A(X \ot_{A} M \ot U, M \ot U).
      \]
      for $\hat{F}$ as introduced before Lemma \ref{l-skein-example-T}.     
      Given $U \cong V \ot W$, then the relations $f_W \sim f_W^{\rho}$ of Lemma \ref{l-skein-example-K-minus} are equivalent to relations on $\Hom_A(X \ot_{A} M \ot U, M \ot U)$. The quotient ranging over all isomorphisms $U \cong V \ot W$ gives the (non-coend) relations of Lemma \ref{l-skein-example-K-minus} on the space $\Hom_A(X \ot_{A} M \ot U, M \ot U)$. Denote by $Q_{M,U}$ the projector to the quotient by these relations. It then follows from Lemma \ref{l-skein-example-K-minus} that
      \[
        \Sk(\mathbb{K}_{-}) \cong \bigoplus_{\substack{M \in \Irr(A\Mod_{\cC})\\U \in \Irr(C)}} \Img Q_{M, U}.
      \]
      Finally, it follows from Theorem \ref{t-main-thm-in-text} that 
      \[
        \Sk(\mathbb{K}_{+}) \cong Z^{\mathrm{def}}(\SSigma) \cong \Sk(\mathbb{K}_{-}).
      \]
      Therefore, we have an isomorphism 
      \[
       \bigoplus_{M \in \Irr(A\Mod_{\cC})} \Hom_A(X \ot_A M, M) \cong \bigoplus_{\substack{M \in \Irr(A\Mod_{\cC})\\U \in \Irr(C)}} \Img Q_{M, U}.
      \]
      The complicated expression on the right-hand side, with a direct sum over two variables and a quotient, can therefore be expressed by the comparatively simpler expression on the left-hand side. Taking $X = A$, then at the level of dimensions we can conclude that
      \[
        \sum_{\substack{M \in \Irr(A\Mod_{\cC})\\U \in \Irr(C)}} \dim \Img Q_{M, U} = \# \Irr(A\Mod_{\cC}).
      \]
    \end{eg}

    \sloppy

    \printbibliography

@article{AFTFactorizationHomologyStratified2016,
  title = {Factorization Homology of Stratified Spaces},
  author = {Ayala, David and Francis, John and Tanaka, Hiro Lee},
  year = 2016,
  month = apr,
  journal = {Selecta Mathematica},
  volume = {23},
  number = {1},
  eprint = {1409.0848},
  pages = {293--362},
  doi = {10.1007/s00029-016-0242-1},
  url = {http://arxiv.org/abs/1409.0848},
  urldate = {2021-08-03},
  archiveprefix = {arXiv},
  langid = {english}
}

@article{BBJ18IntegratingQuantumGroups,
  title = {Integrating Quantum Groups over Surfaces},
  author = {Ben-Zvi, David and Brochier, Adrien and Jordan, David},
  year = 2018,
  month = dec,
  journal = {Journal of Topology},
  volume = {11},
  number = {4},
  eprint = {1501.04652},
  pages = {874--917},
  issn = {1753-8416, 1753-8424},
  doi = {10.1112/topo.12072},
  url = {https://onlinelibrary.wiley.com/doi/10.1112/topo.12072},
  urldate = {2024-04-12},
  archiveprefix = {arXiv},
  langid = {english}
}

@article{BH24SkeinCategoriesNonsemisimple,
  title = {Skein Categories in Non-Semisimple Settings},
  author = {Brown, Jennifer and Ha{\"i}oun, Benjamin},
  year = 2026,
  month = mar,
  journal = {Selecta Mathematica},
  volume = {32},
  number = {2},
  eprint = {2406.08956},
  publisher = {arXiv},
  doi = {10.1007/s00029-026-01143-z},
  url = {http://arxiv.org/abs/2406.08956},
  urldate = {2024-09-20},
  archiveprefix = {arXiv},
  langid = {english}
}

@article{BHMV95TopologicalQuantumField,
  title = {Topological Quantum Field Theories Derived from the Kauffman Bracket},
  author = {Blanchet, Christian and Habegger, Nathan and Masbaum, Gregor and Vogel, Pierre},
  year = 1995,
  month = oct,
  journal = {Topology},
  volume = {34},
  number = {4},
  pages = {883--927},
  issn = {00409383},
  doi = {10.1016/0040-9383(94)00051-4},
  url = {https://linkinghub.elsevier.com/retrieve/pii/0040938394000514},
  urldate = {2024-04-29},
  copyright = {https://www.elsevier.com/tdm/userlicense/1.0/},
  langid = {english}
}

@misc{BJ25ParabolicSkeinModules,
  title = {Parabolic Skein Modules},
  author = {Brown, Jennifer and Jordan, David},
  year = 2025,
  month = may,
  number = {arXiv:2505.14836},
  eprint = {2505.14836},
  primaryclass = {math},
  publisher = {arXiv},
  doi = {10.48550/arXiv.2505.14836},
  url = {http://arxiv.org/abs/2505.14836},
  urldate = {2025-08-21},
  archiveprefix = {arXiv},
  langid = {english}
}

@article{BJS21DualizabilityBraidedTensor,
  title = {On Dualizability of Braided Tensor Categories},
  author = {Brochier, Adrien and Jordan, David and Snyder, Noah},
  year = 2021,
  month = mar,
  journal = {Compositio Mathematica},
  volume = {157},
  number = {3},
  eprint = {1804.07538},
  pages = {435--483},
  doi = {10.1112/S0010437X20007630},
  url = {http://arxiv.org/abs/1804.07538},
  urldate = {2021-02-11},
  archiveprefix = {arXiv},
  langid = {english}
}

@article{CGHP23Skein3+1TQFTsNonsemisimple,
  title = {Skein (3+1)-TQFTs from Non-Semisimple Ribbon Categories},
  author = {Costantino, Francesco and Geer, Nathan and Ha{\"i}oun, Benjamin and Patureau Mirand, Bertrand},
  year = 2026,
  month = apr,
  journal = {Symmetry, Integrability and Geometry: Methods and Applications},
  eprint = {2306.03225},
  primaryclass = {math},
  publisher = {arXiv},
  doi = {10.3842/SIGMA.2026.034},
  url = {http://arxiv.org/abs/2306.03225},
  urldate = {2025-08-21},
  archiveprefix = {arXiv}
}

@misc{CGP23AdmissibleSkeinModules,
  title = {Admissible Skein Modules},
  author = {Costantino, Francesco and Geer, Nathan and {Patureau-Mirand}, Bertrand},
  year = 2023,
  month = feb,
  number = {arXiv:2302.04493},
  eprint = {2302.04493},
  publisher = {arXiv},
  url = {http://arxiv.org/abs/2302.04493},
  urldate = {2023-04-19},
  archiveprefix = {arXiv},
  langid = {english}
}

@article{CM26OrbifoldCompletion3Categories,
  title = {Orbifold Completion of 3-Categories},
  author = {Carqueville, Nils and M{\"u}ller, Lukas},
  year = 2026,
  month = jan,
  journal = {Communications in Mathematical Physics},
  volume = {407},
  number = {1},
  eprint = {2307.06485},
  pages = {8},
  issn = {0010-3616, 1432-0916},
  doi = {10.1007/s00220-025-05434-y},
  url = {https://link.springer.com/10.1007/s00220-025-05434-y},
  urldate = {2026-06-05},
  archiveprefix = {arXiv},
  langid = {english}
}

@misc{CMR+21OrbifoldGraphTQFTs,
  title = {Orbifold Graph TQFTs},
  author = {Carqueville, Nils and Mulevicius, Vincentas and Runkel, Ingo and Schaumann, Gregor and Scherl, Daniel},
  year = 2021,
  month = jan,
  number = {arXiv:2101.02482},
  eprint = {2101.02482},
  primaryclass = {math},
  publisher = {arXiv},
  doi = {10.48550/arXiv.2101.02482},
  url = {http://arxiv.org/abs/2101.02482},
  urldate = {2026-04-20},
  archiveprefix = {arXiv}
}

@article{CMS203dimensionalDefectTQFTs,
  title = {3-Dimensional Defect TQFTs and Their Tricategories},
  author = {Carqueville, Nils and Meusburger, Catherine and Schaumann, Gregor},
  year = 2020,
  month = apr,
  journal = {Advances in Mathematics},
  volume = {364},
  eprint = {1603.01171},
  primaryclass = {math},
  pages = {107024},
  issn = {00018708},
  doi = {10.1016/j.aim.2020.107024},
  url = {http://arxiv.org/abs/1603.01171},
  urldate = {2025-10-27},
  archiveprefix = {arXiv}
}

@article{Coo23ExcisionSkeinCategories,
  title = {Excision of Skein Categories and Factorisation Homology},
  author = {Cooke, Juliet},
  year = 2023,
  month = feb,
  journal = {Advances in Mathematics},
  volume = {414},
  eprint = {1910.02630},
  pages = {108848},
  issn = {00018708},
  doi = {10.1016/j.aim.2022.108848},
  url = {https://linkinghub.elsevier.com/retrieve/pii/S000187082200665X},
  urldate = {2024-04-10},
  archiveprefix = {arXiv},
  langid = {english}
}

@article{CRS18LineSurfaceDefects,
  title = {Line and Surface Defects in Reshetikhin-Turaev TQFT},
  author = {Carqueville, Nils and Runkel, Ingo and Schaumann, Gregor},
  year = 2018,
  month = oct,
  journal = {Quantum Topology},
  volume = {10},
  number = {3},
  eprint = {1710.10214},
  primaryclass = {hep-th, physics:math-ph},
  pages = {399--439},
  issn = {1663-487X, 1664-073X},
  doi = {10.4171/QT/121},
  url = {http://arxiv.org/abs/1710.10214},
  urldate = {2024-09-10},
  archiveprefix = {arXiv}
}

@article{CRS19OrbifoldsNdimensionalDefect,
  title = {Orbifolds of N-Dimensional Defect TQFTs},
  author = {Carqueville, Nils and Runkel, Ingo and Schaumann, Gregor},
  year = 2019,
  month = apr,
  journal = {Geometry \& Topology},
  volume = {23},
  number = {2},
  eprint = {1705.06085},
  primaryclass = {math},
  pages = {781--864},
  issn = {1364-0380, 1465-3060},
  doi = {10.2140/gt.2019.23.781},
  url = {http://arxiv.org/abs/1705.06085},
  urldate = {2025-06-16},
  archiveprefix = {arXiv}
}

@article{CRS20OrbifoldsReshetikhinTuraevTQFTs,
  title = {Orbifolds of Reshetikhin-Turaev TQFTs},
  author = {Carqueville, Nils and Runkel, Ingo and Schaumann, Gregor},
  year = 2020,
  month = apr,
  journal = {Theory and Applications of Categories},
  volume = {35},
  number = {15},
  eprint = {1809.01483},
  primaryclass = {math},
  pages = {513--561},
  doi = {10.48550/arXiv.1809.01483},
  url = {http://www.tac.mta.ca/tac/volumes/35/15/35-15abs.html},
  urldate = {2025-08-21},
  archiveprefix = {arXiv}
}

@incollection{CY93CategoricalConstruction4d,
  title = {A Categorical Construction of 4d Topological Quantum Field Theories},
  booktitle = {Quantum Topology},
  author = {Crane, Louis and Yetter, David},
  editor = {Kauffman, Louis H and Baadhio, Randy A},
  year = 1993,
  month = sep,
  series = {Series on Knots and Everything},
  volume = {3},
  eprint = {hep-th/9301062},
  pages = {120--130},
  publisher = {World Scientific},
  doi = {10.1142/9789812796387_0005},
  url = {http://www.worldscientific.com/doi/abs/10.1142/9789812796387_0005},
  urldate = {2024-04-27},
  archiveprefix = {arXiv},
  isbn = {978-981-02-1544-6 978-981-279-638-7},
  langid = {english}
}

@incollection{Del90CategoriesTannakiennes,
  title = {Cat\'egories Tannakiennes},
  booktitle = {The Grothendieck Festschrift},
  author = {Deligne, P.},
  editor = {Cartier, Pierre and Katz, Nicholas M. and Manin, Yuri I. and Illusie, Luc and Laumon, G{\'e}rard and Ribet, Kenneth A.},
  year = 1990,
  pages = {111--195},
  publisher = {Birkh\"auser Boston},
  address = {Boston, MA},
  doi = {10.1007/978-0-8176-4575-5_3},
  url = {http://link.springer.com/10.1007/978-0-8176-4575-5_3},
  urldate = {2026-02-17},
  isbn = {978-0-8176-4567-0 978-0-8176-4575-5},
  langid = {english}
}

@article{DGG+223DimensionalTQFTsNonsemisimple,
  title = {3-Dimensional TQFTs from Non-Semisimple Modular Categories},
  author = {De Renzi, Marco and Gainutdinov, Azat M. and Geer, Nathan and {Patureau-Mirand}, Bertrand and Runkel, Ingo},
  year = 2022,
  month = jan,
  journal = {Selecta Mathematica},
  volume = {28},
  number = {2},
  eprint = {1912.02063},
  doi = {10.1007/s00029-021-00737-z},
  url = {http://arxiv.org/abs/1912.02063},
  urldate = {2022-09-13},
  archiveprefix = {arXiv},
  langid = {english}
}

@book{EGNO15TensorCategories,
  title = {Tensor Categories},
  author = {Etingof, Pavel and Gelaki, Shlomo and Nikshych, Dmitri and Ostrik, Victor},
  year = 2015,
  series = {Mathematical Surveys and Monographs},
  number = {volume 205},
  publisher = {American Mathematical Society},
  address = {Providence, Rhode Island},
  isbn = {978-1-4704-2024-6},
  langid = {english},
  lccn = {QA612 .T46 2015}
}

@article{ENO10FusionCategoriesHomotopy,
  title = {Fusion Categories and Homotopy Theory},
  author = {Etingof, Pavel and Nikshych, Dmitri and Ostrik, Victor},
  year = 2010,
  month = aug,
  journal = {Quantum Topology},
  volume = {1},
  number = {3},
  eprint = {0909.3140},
  pages = {209--273},
  issn = {1663-487X, 1664-073X},
  doi = {10.4171/qt/6},
  url = {https://ems.press/doi/10.4171/qt/6},
  urldate = {2026-02-17},
  archiveprefix = {arXiv}
}

@article{FFRS04KramersWannierDualityConformal,
  title = {Kramers-Wannier Duality from Conformal Defects},
  author = {Fr{\"o}hlich, J{\"u}rg and Fuchs, J{\"u}rgen and Runkel, Ingo and Schweigert, Christoph},
  year = 2004,
  month = aug,
  journal = {Physical Review Letters},
  volume = {93},
  number = {7},
  eprint = {cond-mat/0404051},
  pages = {070601},
  issn = {0031-9007, 1079-7114},
  doi = {10.1103/PhysRevLett.93.070601},
  url = {http://arxiv.org/abs/cond-mat/0404051},
  urldate = {2025-08-21},
  archiveprefix = {arXiv},
  langid = {english}
}

@article{FMT23TopologicalSymmetryQuantum,
  title = {Topological Symmetry in Quantum Field Theory},
  author = {Freed, Daniel S. and Moore, Gregory W. and Teleman, Constantin},
  year = 2024,
  month = oct,
  journal = {Quantum Topology},
  volume = {15},
  number = {3},
  eprint = {2209.07471},
  pages = {779--869},
  publisher = {arXiv},
  doi = {10.4171/QT/223},
  url = {http://arxiv.org/abs/2209.07471},
  urldate = {2023-08-08},
  archiveprefix = {arXiv},
  langid = {english}
}

@article{FSV13BicategoriesBoundaryConditions,
  title = {Bicategories for Boundary Conditions and for Surface Defects in 3-d TFT},
  author = {Fuchs, J{\"u}rgen and Schweigert, Christoph and Valentino, Alessandro},
  year = 2013,
  month = jul,
  journal = {Communications in Mathematical Physics},
  volume = {321},
  number = {2},
  eprint = {1203.4568},
  pages = {543--575},
  issn = {0010-3616, 1432-0916},
  doi = {10.1007/s00220-013-1723-0},
  url = {http://link.springer.com/10.1007/s00220-013-1723-0},
  urldate = {2026-06-05},
  archiveprefix = {arXiv},
  copyright = {http://www.springer.com/tdm},
  langid = {english}
}

@article{FSY23StringnetModelsPivotal,
  title = {String-Net Models for Pivotal Bicategories},
  author = {Fuchs, J{\"u}rgen and Schweigert, Christoph and Yang, Yang},
  year = 2025,
  month = may,
  journal = {Theory and Applications of Categories},
  volume = {44},
  number = {17},
  eprint = {2302.01468},
  pages = {474--543},
  url = {http://www.tac.mta.ca/tac/volumes/44/17/44-17abs.html},
  urldate = {2026-06-05},
  archiveprefix = {arXiv},
  langid = {english}
}

@misc{Gar26HOMFLYParabolicRestriction,
  title = {HOMFLY Parabolic Restriction, Defect Skein Theory and the Turaev Coproduct},
  author = {Garc{\'i}a G{\'o}mez, Juan Ram{\'o}n},
  year = 2026,
  month = jan,
  number = {arXiv:2601.03196},
  eprint = {2601.03196},
  primaryclass = {math},
  publisher = {arXiv},
  doi = {10.48550/arXiv.2601.03196},
  url = {http://arxiv.org/abs/2601.03196},
  urldate = {2026-04-21},
  archiveprefix = {arXiv}
}

@article{GKSW15GeneralizedGlobalSymmetries,
  title = {Generalized Global Symmetries},
  author = {Gaiotto, Davide and Kapustin, Anton and Seiberg, Nathan and Willett, Brian},
  year = 2015,
  month = feb,
  journal = {Journal of High Energy Physics},
  volume = {2015},
  number = {2},
  eprint = {1412.5148},
  primaryclass = {hep-th},
  pages = {172},
  issn = {1029-8479},
  doi = {10.1007/JHEP02(2015)172},
  url = {http://arxiv.org/abs/1412.5148},
  urldate = {2025-08-21},
  archiveprefix = {arXiv},
  langid = {english}
}

@book{Kel05BasicConceptsEnriched,
  title = {Basic Concepts of Enriched Category Theory},
  author = {Kelly, G.M.},
  year = 2005,
  series = {Reprints in Theory and Applications of Categories},
  number = {10},
  url = {http://www.tac.mta.ca/tac/reprints/articles/10/tr10abs.html},
  langid = {english}
}

@article{Kel89ElementaryObservations2categorical,
  title = {Elementary Observations on 2-Categorical Limits},
  author = {Kelly, G.M.},
  year = 1989,
  month = apr,
  journal = {Bulletin of the Australian Mathematical Society},
  volume = {39},
  number = {2},
  pages = {301--317},
  issn = {0004-9727, 1755-1633},
  doi = {10.1017/S0004972700002781},
  url = {https://www.cambridge.org/core/product/identifier/S0004972700002781/type/journal_article},
  urldate = {2026-02-17},
  copyright = {https://www.cambridge.org/core/terms},
  langid = {english}
}

@article{KMRS22DomainWalls3d,
  title = {Domain Walls Between 3d Phases of Reshetikhin--Turaev TQFTs},
  author = {Koppen, Vincent and Mulevi{\v c}ius, Vincentas and Runkel, Ingo and Schweigert, Christoph},
  year = 2022,
  month = dec,
  journal = {Communications in Mathematical Physics},
  volume = {396},
  number = {3},
  eprint = {2105.04613},
  pages = {1187--1220},
  issn = {0010-3616, 1432-0916},
  doi = {10.1007/s00220-022-04489-5},
  url = {https://link.springer.com/10.1007/s00220-022-04489-5},
  urldate = {2026-06-05},
  archiveprefix = {arXiv},
  langid = {english}
}

@article{Lop13TensorProductsFinitely,
  title = {Tensor Products of Finitely Cocomplete and Abelian Categories},
  author = {L{\'o}pez Franco, Ignacio},
  year = 2013,
  month = dec,
  journal = {Journal of Algebra},
  volume = {396},
  eprint = {1212.1545},
  pages = {207--219},
  issn = {00218693},
  doi = {10.1016/j.jalgebra.2013.08.015},
  url = {https://linkinghub.elsevier.com/retrieve/pii/S0021869313004626},
  urldate = {2026-02-17},
  archiveprefix = {arXiv},
  copyright = {https://www.elsevier.com/tdm/userlicense/1.0/},
  langid = {english}
}

@book{Lor21CoendCalculus,
  title = {(Co)End Calculus},
  author = {Loregian, Fosco},
  year = 2021,
  series = {London Mathematical Society Lecture Note Series},
  eprint = {1501.02503},
  publisher = {Cambridge University Press},
  address = {Cambridge},
  doi = {10.1017/9781108778657},
  url = {https://www.cambridge.org/core/books/coend-calculus/C662E90767358B336F17B606D19D8C43},
  urldate = {2023-04-13},
  archiveprefix = {arXiv},
  isbn = {978-1-108-74612-0}
}

@article{Meu23StateSumModels,
  title = {State Sum Models with Defects Based on Spherical Fusion Categories},
  author = {Meusburger, Catherine},
  year = 2023,
  month = sep,
  journal = {Advances in Mathematics},
  volume = {429},
  eprint = {2205.06874},
  pages = {109177},
  issn = {00018708},
  doi = {10.1016/j.aim.2023.109177},
  url = {https://linkinghub.elsevier.com/retrieve/pii/S0001870823003201},
  urldate = {2026-05-29},
  archiveprefix = {arXiv},
  langid = {english}
}

@article{MW12BlobComplex,
  title = {The Blob Complex},
  author = {Morrison, Scott and Walker, Kevin},
  year = 2012,
  month = jul,
  journal = {Geometry \& Topology},
  volume = {16},
  number = {3},
  eprint = {1009.5025},
  primaryclass = {math},
  pages = {1481--1607},
  issn = {1364-0380, 1465-3060},
  doi = {10.2140/gt.2012.16.1481},
  url = {http://arxiv.org/abs/1009.5025},
  urldate = {2023-08-22},
  archiveprefix = {arXiv}
}

@article{Ost01ModuleCategoriesWeak,
  title = {Module Categories, Weak Hopf Algebras and Modular Invariants},
  author = {Ostrik, Victor},
  year = 2003,
  month = jun,
  journal = {Transformation Groups},
  volume = {8},
  number = {2},
  eprint = {math/0111139},
  pages = {177--206},
  publisher = {arXiv},
  doi = {10.1007/s00031-003-0515-6},
  url = {http://arxiv.org/abs/math/0111139},
  urldate = {2024-10-30},
  archiveprefix = {arXiv},
  langid = {english}
}

@article{PrzKauffmanBracketSkein1999,
  title = {Kauffman Bracket Skein Module of a Connected Sum of 3-Manifolds},
  author = {Przytycki, J{\'o}zef H.},
  year = 2000,
  month = feb,
  journal = {Manuscripta Mathematica},
  volume = {101},
  number = {2},
  eprint = {math/9911120},
  pages = {199--207},
  publisher = {arXiv},
  doi = {10.1007/s002290050014},
  url = {http://arxiv.org/abs/math/9911120},
  urldate = {2025-01-14},
  archiveprefix = {arXiv},
  langid = {english}
}

@misc{RST24ExcisionSpacesAdmissible,
  title = {Excision for Spaces of Admissible Skeins},
  author = {Runkel, Ingo and Schweigert, Christoph and Tham, Ying Hong},
  year = 2024,
  month = jul,
  number = {arXiv:2407.09302},
  eprint = {2407.09302},
  primaryclass = {math},
  publisher = {arXiv},
  doi = {10.48550/arXiv.2407.09302},
  url = {http://arxiv.org/abs/2407.09302},
  urldate = {2025-08-21},
  archiveprefix = {arXiv}
}

@article{RT90RibbonGraphsTheir,
  title = {Ribbon Graphs and Their Invaraints Derived from Quantum Groups},
  author = {Reshetikhin, Nicolai Y. and Turaev, Vladimir G.},
  year = 1990,
  month = jan,
  journal = {Communications in Mathematical Physics},
  volume = {127},
  number = {1},
  pages = {1--26},
  issn = {0010-3616, 1432-0916},
  doi = {10.1007/BF02096491},
  url = {http://link.springer.com/10.1007/BF02096491},
  urldate = {2024-04-23},
  copyright = {http://www.springer.com/tdm},
  langid = {english}
}

@article{Sch13TracesModuleCategories,
  title = {Traces on Module Categories over Fusion Categories},
  author = {Schaumann, Gregor},
  year = 2013,
  month = apr,
  journal = {Journal of Algebra},
  volume = {379},
  eprint = {1206.5716},
  primaryclass = {math},
  pages = {382--425},
  issn = {00218693},
  doi = {10.1016/j.jalgebra.2013.01.013},
  url = {http://arxiv.org/abs/1206.5716},
  urldate = {2025-10-10},
  archiveprefix = {arXiv},
  langid = {english}
}

@phdthesis{Tha21CategoryBoundaryValues,
  title = {On the Category of Boundary Values in the Extended Crane-Yetter TQFT},
  author = {Tham, Ying Hong},
  year = 2021,
  month = aug,
  eprint = {2108.13467},
  primaryclass = {math},
  doi = {10.48550/arXiv.2108.13467},
  url = {http://arxiv.org/abs/2108.13467},
  urldate = {2026-02-12},
  archiveprefix = {arXiv},
  school = {Stony Brook University}
}

@book{Tur10QuantumInvariantsKnots,
  title = {Quantum Invariants of Knots and 3-Manifolds},
  author = {Turaev, Vladimir G.},
  year = 2010,
  series = {De Gruyter Studies in Mathematics},
  edition = {2nd rev. ed},
  number = {18},
  publisher = {De Gruyter},
  address = {Berlin ; New York},
  isbn = {978-3-11-022183-1},
  langid = {english},
  lccn = {QC174.52.C66 T87 2010}
}

\end{document}